\documentclass[11pt,final]{article}
\usepackage{amsmath,amsthm,amsfonts,amssymb,graphicx,hyperref,enumerate,psfrag}
\usepackage[a4paper,margin=2.5cm]{geometry}
\usepackage{fullpage}
\usepackage{enumitem}
\usepackage{algorithm}
\usepackage{algorithmic}
\usepackage{nicefrac}
\usepackage{tikz}

\usepackage[utf8]{inputenc} 
\usepackage{stmaryrd}

\newtheorem{definition}{Definition}
\newtheorem{lemma}{Lemma}[section]

\newtheorem{remark}{Remark}

\newtheorem{proposition}[lemma]{Proposition}
\newtheorem{theorem}[lemma]{Theorem}

\newtheorem{corollary}[lemma]{Corollary}

\newcommand{\dE}{\mathbb {E}}
\newcommand{\dP}{\mathbb {P}}

\newcommand{\dN}{\mathbb {N}}

\newcommand{\cE}{\mathcal {E}}
\newcommand{\cF}{\mathcal {F}}

\newcommand{\cT}{\mathcal {T}}

\newcommand{\bmu}{{\boldsymbol{\mu}}}

\newcommand{\1}{1\!\!{\sf I}}

\newcommand{\Var}{\text{Var}}

\newcommand{\Pmu}{\mathrm P_{\bmu}}
\newcommand{\Emu}{\mathrm E_{\bmu}}
\newcommand{\bP}{\mathbf{P}}
\newcommand{\bE}{\mathbf{E}}
\newcommand{\bnu}{\boldsymbol \nu}

\usepackage{xcolor}

\numberwithin{equation}{section}

\title{Scaling limit of critical random trees in random environment}
\author{Guillaume Conchon-\,\hspace{-0.3mm}-Kerjan\thanks{Department of Mathematical Sciences, University of Bath, Claverton Down, BA2 7AY Bath, UK.\newline Email: \texttt{galck20/d.kious/c.mailler@bath.ac.uk}}  \thanks{GCK and DK are grateful to EPSRC for support through the grant EP/V00929X/1.} \and Daniel Kious\footnotemark[1] \footnotemark[2]  \and C\'ecile Mailler\footnotemark[1] \thanks{CM is grateful to EPSRC for support through the fellowship EP/R022186/1.}}
\date{}
\begin{document}
\maketitle

\begin{abstract}
We consider Bienaym\'e-Galton-Watson trees in random environment, 
where each generation $k$ is attributed a random offspring distribution $\mu_k$, and $(\mu_k)_{k\geq 0}$ is a sequence of 
independent and identically distributed random probability measures.
We work in the ``strictly critical'' regime where, for all $k$, 
the average of $\mu_k$ is assumed to be equal to~$1$ almost surely, 
and the variance of $\mu_k$ has finite expectation.
We prove that, for almost all realizations of the environment 
(more precisely, under some deterministic conditions that the random environment satisfies almost surely),
the scaling limit of the tree in that environment, conditioned to be large, 
is the Brownian continuum random tree.
The habitual techniques used for standard Bienaym\'e-Galton-Watson trees, or trees with exchangeable vertices, do not apply to this case. Our proof therefore provides alternative tools.
\end{abstract}
%
%
%


\section{Introduction}
In this paper, we consider Bienaym\'e-Galton-Watson trees in random environment, in the sense that each generation $k\ge0$ is assigned a random offspring distribution $\mu_k$ sampled from a common measure $\Lambda$, independently over $k$. 
Throughout the paper, we work in the strictly-critical case, i.e.~we assume that, for all $k\ge 1$,  the average of $\mu_k$ is equal to 1 almost surely. 
However, we allow the variance $\sigma_k^2$ to be a non-degenerate random variable.
See Section \ref{sec:def} for a precise definition of the model.

We prove that, if the expectation of $\sigma_k^2$ is finite, then a.s., the environment $\bmu=(\mu_k)_{k\geq 0}$ is such that the scaling limit of the tree in environment $\bmu$, conditioned to have at least $n$ vertices with $n\rightarrow \infty$, is Aldous' continuum random tree, see Theorem \ref{th:main}.

The main technical challenge is that standard techniques, 
developed in the case of a constant environment, do not apply here, 
because the classical exploration algorithms become more difficult to analyse. 
For instance, the exploration process (in depth-first order) is a Markovian random walk when the environment is constant, whereas it becomes a non-Markovian process in our case. 
This is due to the fact that the offspring distribution of the $n$-th vertex (in depth-first order) 
depends on its height, 
which in turn depends on the whole exploration up to that step, 
see Section \ref{sub:novelty}. Another challenge posed by the varying environment is that degrees are not exchangeable, as they do not appear with the same probability at different heights. Hence arguments resorting to exchangeability, that have in particular been developed in the literature of random trees with prescribed degrees (see e.g.~\cite{AddarioDonderwinkel, AldousI, BlancRenaudie, BroutinMarckert}), are not relevant in our setting. 
Other approaches 
developed in recent papers on the width of a tree in varying environment
at a given large height and genealogical properties below that height 
(see~\cite{HarrisPalauPardo, BoenkostFRSchertzer} for the latter
and Section~\ref{sec:disc} for more references)
do not seem to be of use when trying to characterise the structure of the entire tree, and thus its scaling limit.
%
%
%
%
Consequently, we provide an alternative approach for proving scaling limits of random trees. 

Our model is naturally related to Bienaymé-Galton-Watson trees in varying environment, 
which are a generalisation of Bienaymé-Galton-Watson trees in which the offspring distribution
varies from generation to generation according to a deterministic sequence of distributions. These can be thought of as a quenched version of random trees in random environment. 
In Section \ref{sub:ideas}, we provide deterministic conditions on the environment (which are a.s.\ satisfied by our random environment) 
for our main result to hold. 

Both these models, where the offspring distribution depends on the generation, 
are natural extensions of the original Bienaymé-Galton-Watson trees.
Indeed, they provide a framework that takes into account 
the possibility of a variation in the reproduction rate of a population depending, for example, 
on seasons (to simplify, one can think of a population reproducing faster in the summer and slower in the winter).
Therefore, Bienaym\'e-Galton-Watson trees in random environment or in varying environment have been extensively studied in the literature.
We refer the reader to Kersting and Vatutin's book~\cite{KerstingVatutinBook} 
for a comprehensive review of the state of the art on these processes;
more discussions will follow in Section~\ref{sec:disc}.

\subsection{Definition of the model and notation} \label{sec:def}

In this section, we define Bienaym\'e-Galton-Watson trees in varying environment and in random environment,  abbreviated respectively BPVE and BPRE for branching processes in varying environment and in random environment.

All the probability measures and random variables below are defined on a common measurable space $(\Omega, \mathcal{F})$.
An environment is a sequence $\bnu = (\nu_k)_{k\geq 0}$ of probability distributions 
on $\mathbb N = \{0, 1, \ldots\}$.
For all environments $\bnu$,
we define a collection of offspring random variables $(\xi_k^{(i)}, i\ge 1, k\ge 0)$ under a product measure $\mathrm P_{\bnu}$ so that these are independent and, for each $k\ge0$, $\xi_k^{(i)}$ has distribution $\nu_k$, for all $i\ge1$.
Under the measure ${\mathrm P}_{\bnu}$, we define a BPVE  in the environment $\bnu$, or $\bnu$-BPVE, recursively as follows: 
Let $Z_0 = z_0\in\mathbb{N}$ and, for all $k\geq 0$,
\[
Z_{k+1} = \sum_{i=1}^{Z_k} \xi_k^{(i)}.
\]
Note that $(\xi_k^{(i)})_{i\geq 1}$ above is a sequence of i.i.d.\ random variables of distribution $\nu_k$, independent of $Z_k$.
It is classical to interpret this process as the evolution of a population which has  $Z_k$ individuals at generation~$k$ and where each individual $i$, $1\le i\le Z_k$, has $\xi_i^{(k)}$  offspring in generation~$k+1$.

We now give the definition of a BPRE.
Let $\Lambda$ be a probability measure on $\mathcal P(\mathbb N)$, the set of probability measures 
on $\mathbb{N}$, and define the product measure ${\bf P}= \Lambda^{\otimes \mathbb N}$.
 We choose the random environment $\bmu=(\mu_k)_{k\ge0}$ to be distributed under ${\bf P}$, hence it consists of a sequence of i.i.d.~probability measures.
A branching process is a BPRE with environment distribution ${\bf P}$, or ${\bf P}$-BPRE, if it is a BPVE in the environment~$\bmu$ with~$\bmu$ a random variable of distribution $\bf P$. 

We let $\mathrm P_{\bmu}$ denote the {\it quenched} measure and ${\dP}$ the {\it annealed} measure of $\mathcal T^{\bmu}$.
 I.e., $\mathrm P_{\bmu}$ is the distribution of $\mathcal T^{\bmu}$ given $\bmu$, while $\dP$ is the unconditional distribution of $\mathcal T^{\bmu}$.
For all $k\ge0$, we define the average offspring and the variance at generation $k$ by 
\[
\bar\mu_k=\mathrm E_{\bmu}[\xi_k^{(1)}]
\quad\text{ and }\quad
\sigma^2_k=\mathrm E_{\bmu}\left[\left(\xi_k^{(1)}-\bar\mu_k\right)^2\right].
\]
Note that if $\mathbf P$ is the distribution of $\bmu$, 
then $((\bar\mu_k,\sigma_k^2))_{k\geq 0}$ is a sequence of i.i.d.\ random variables taking values in $[0,+\infty]^2$. 
In our main result, we assume that $\bar\mu_k = 1$ almost surely, 
and discuss this condition in Section \ref{subsec:critstrictcrit}.

Scaling limit theorems for random trees are classically expressed as convergence of metric spaces; therefore, we see a BPRE or a BPVE as a metric space, endowed with a measure. 
More precisely, we consider triplets of the form $(\mathcal{T}^{\bmu}, d_{\bmu}, \pi_{\bmu})$,
where $\mathcal T^{\bmu}$ is the set of all the nodes (individuals) created by that process, 
$d_{\bmu}$ is the distance induced by the graph distance on the family tree of the branching process, 
and $\pi_{\bmu}$ is the uniform distribution on the vertices, with total mass~1.

From now on, we may simply refer to $(\mathcal{T}^{\bmu},d_{\bmu},\pi_{\bmu})$ or $\mathcal{T}^{\bmu}$ as a BPRE or a $\bmu$-BPVE, often overlooking the metric space structure.

There is a natural metric on the Polish space of such compact measured metric spaces, 
called the Gromov-Hausdorff-Prokhorov metric, 
which is often used for scaling limits of random geometric objects, and in particular random trees.
For a definition of the Gromov-Hausdorff-Prokhorov metric, we refer the reader to, e.g., \cite[Section~2]{ABBGM17}, 
and the references therein. 

\subsection{Statement of the main result}

In this section, we state our main theorem. We are interested in the family tree of a BPRE, 
seen as a measured metric space, when this tree is large.
Let $\bmu$ be a sequence of distributions drawn under $\mathbf{P}$.
On the environment $\bmu$, for all $n\ge1$, we define the tree $(\mathcal{T}_n,d_n,\pi_n)$ as the $\bmu$-BPVE distributed under the conditional measure $\mathrm P_{\bmu}\left(\left. \ \cdot \ \right\vert |\mathcal{T}^{\bmu}|\ge n\right)$, where $|\cdot|$ denotes the cardinality of a set. 
As in the previous section, 
$\mathcal T_n$ is the set of the $\vert \mathcal T_n\vert$ individuals in the total population, 
$d_n$ is the distance induced by the graph distance on the family tree of the branching process, 
and $\pi_n$ is the uniform distribution on the vertices with total mass~1.

\begin{theorem}\label{th:main}
Assume that  
$\bar\mu_0=1$ $\mathbf{P}$-a.s.~and that $\sigma^2:={\bf E}[\sigma_0^2]\in(0,\infty)$.
Let $\bmu$ be a sequence of i.i.d.~distributions drawn under $\mathbf{P}$
and let $(\mathcal T_n, d_n,\pi_n)$ be a $\bmu$-BPVE.
For ${\bf P}$-almost all $\bmu$, 
\[\Big(\mathcal T_n, \frac1{\sigma\sqrt{\vert \mathcal T_n\vert} } d_n, \pi_n\Big) 
\stackrel{(d)}{\longrightarrow} (\mathcal T, d, \pi),\]
as $n\to\infty$,
on the set of compact metric spaces equipped with the Gromov-Hausdorff-Prokhorov metric,
and where $(\mathcal T, d, \pi)$ is Aldous' continuum random tree.
\end{theorem}
Note that the result above is {\it quenched}, in the sense that it holds for almost all environment drawn under $\mathbf{P}$. In fact, we prove the convergence for all environment $\bmu$ satisfying the assumptions \ref{(I)}-\ref{(V)} presented in Section~\ref{sec:disc}, and prove that these conditions are satisfied $\bP$-almost surely.

Moreover, observe that, in Theorem \ref{th:main}, we assume that $\bar\mu_0=1$ $\mathbf{P}$-a.s., which puts us in the case called {\it strictly critical}. We discuss this assumption in Section \ref{subsec:critstrictcrit} below. In particular, this implies that $\mathbf{P}$-a.s., a $\bmu$-BPVE is a.s.~finite.

For the definition and properties of the continuum random tree (CRT), 
we refer to the seminal series of papers by Aldous~\cite{AldousI, AldousII, AldousIII}.
Recalling that ${\bf P}= \Lambda^{\otimes \mathbb N}$, and  choosing $\Lambda$ to be a  Dirac measure 
putting mass on a single probability distribution with finite variance, it is straightforward to see that all of our proofs apply to critical Bienaym\'e-Gatlon-Watson trees, and one can thus recover the classical result of Aldous~\cite{AldousIII} from Theorem \ref{th:main}.
Although our results can be seen as a generalisation of the classical Bienaym\'e-Galton-Watson case,
the classical proofs are not robust when the environment varies, and new ideas are needed. 
This is discussed in more details in Section~\ref{sub:novelty}.


\subsection{Critical and strictly critical BPREs}\label{subsec:critstrictcrit}
In Theorem \ref{th:main}, we give a scaling limit of a {\it critical} random tree conditioned to be large. Nevertheless, contrary to the case of usual Bienaymé-Galton-Watson trees, the notion of criticality for BPRE and BPVE is more delicate, as it is well-explained in~\cite{KerstingVatutinBook}.
The definition below provides two levels of criticality for a BPRE: in Theorem \ref{th:main}, we work in the strictly critical case. (Recall that, under $\mathbf{P}$, $\bmu$ is a sequence of i.i.d.~random distributions, and that $\bar\mu_k$ denotes the average of $\mu_k$.)

\begin{definition}\label{def:crit}
A BPRE of environment distribution $\mathbf{P}$ is called critical if ${\bf E}[\log\bar\mu_0] = 0$. 
It is called strictly critical if $\mu_0 = 1$ $\mathbf{P}$-almost surely.
\end{definition}
Note that a critical BPRE is eventually almost surely extinct,  
see ~\cite[Theorem~2.1]{KerstingVatutinBook} for the critical case, 
and \cite[Section~2.5]{KerstingVatutinBook} for the strictly critical case.
See \cite{KerstingVatutinBook} for a more properties of critical and strictly critical BPREs. 

The assumption of strict criticality made in Theorem \ref{th:main} is necessary. 
We now provide a short argument for the fact that the convergence to Aldous' CRT does not hold in the non-strictly critical case: 
By~\cite[Theorem 2.6]{KerstingVatutinBook}, the probability that a strictly-critical BPRE survives up to generation~$n$ is of order~$\nicefrac1n$, 
which is consistent with critical Bienaym\'e-Galton-Watson trees. 
In contrast, by~\cite[Theorem 5.1]{KerstingVatutinBook}, the probability that a non-strictly critical BPRE survives 
up to generation~$n$ is of order $\nicefrac1{\sqrt n}$. 
This extra large probability is due to the fact that the environment over the first~$n$ generations can be very favorable to survival by looking like a super-critical environment, 
which happens with probability $\nicefrac1{\sqrt n}$. 
Hence, one cannot expect that our theorem holds for non-strictly-critical BPREs.

\section{Sketch of proof and discussion}
In this section, we provide the structure of the proof of Theorem~\ref{th:main}, define the key objects needed, and explain why the classical techniques used for scaling limits of standard Bienaym\'e-Galton-Watson trees do not apply in varying environment. We start by a technical section providing the necessary tools and notation. Then, we describe the plan of the proof in Section~\ref{sub:novelty}. We develop our ideas and detail our technical assumptions in Section~\ref{sub:ideas}. We discuss the existing literature in Section~\ref{sec:disc} and open questions in Section~\ref{sub:open}.

\subsection{Orders on  forests and exploration processes}\label{sectorder}

%
%
Throughout the paper, we work on forests instead of single trees, 
as it is usually easier to observe a large tree occurring naturally in a forest, 
rather than conditioning a single tree to be large.

A forest $\mathcal F$ is a collection $(\mathcal T^{(i)})_{i\geq 1}$ 
of finite trees.
For a finite tree $T$ we often let $(Z_m(T))_{m\ge0}$ denote the sizes of its successive generations. We define the height of the tree as
\begin{equation}
h(T)=\max\{m\ge0: Z_m(T)>0\},
\end{equation} 
and the width of the tree as
\begin{equation}
w(T)= \max_{m\ge0} Z_m(T).
\end{equation}
Similarly, if $\mathcal{F}$ is made of finitely many trees, say $\ell$, we define the height and width of $\mathcal{F}$ by
\begin{equation}\label{def:heightwidth}\begin{split}
h(\mathcal{F})&=\max_{1\le i\le \ell} h(T_i)
\quad \text{ and }\quad
w(\mathcal{F})=\max_{m\ge0} \sum_{i=1}^{\ell} Z_m(T_i).
\end{split}
\end{equation}

When applied to a vertex $x$ of a tree $T$, $h(x)$ will denote the height of $x$ in $T$, i.e.~the graph distance from the root of $T$ to $x$.

Now, we define a lexicographical order on a forest $\mathcal{F}$, which is a labeling defined  as follows. We label the trees one by one, starting with $T^{(1)}$. Start by calling $v_{1,0}$ the root of the first tree $T^{(1)}$ of $\cF$, and $v_{i,1}$, $1\le i \le Z_1(T^{(1)})$, the children of this root. Let $k\ge0$ and assume that we have assigned a label to all the vertices at height at most $k$. If the generation $k+1$ is empty then we have labelled the whole tree. If not, assign the labels  $v_{i,k+1}$, $1\le i \le Z_{k+1}(T^{(1)})$, to the vertices at height $k+1$ of $T^{(1)}$ in such a way that, for all $i_1<i_2$, if $v_{j_1,k+1}$  and $v_{j_2,k+1}$ are offspring of $v_{i_1,k}$ and $v_{i_2,k}$ respectively then $j_1<j_2$. For $i\ge1$, once we have fully labelled the trees $T^{(j)}$ for $1\le j \le i$, we now label the tree $T^{(i+1)}$ in a similar fashion, using the label $\tilde{v}_{\ell,k}=v_{\theta_{i,k}(\ell),k}$, for $\ell\ge 1$ and $k\ge0$, where $\theta_{i,k}(\ell)=\ell+\sum_{j=1}^i Z_k(T^{(j)})$. This ordering is illustrated on Figure~\ref{fig:exploforest}. With this definition, $v_{i,j}$ is the $j$-th vertex (from left to right) at generation $i$.

\begin{figure}
\begin{center}
\includegraphics[width=13cm]{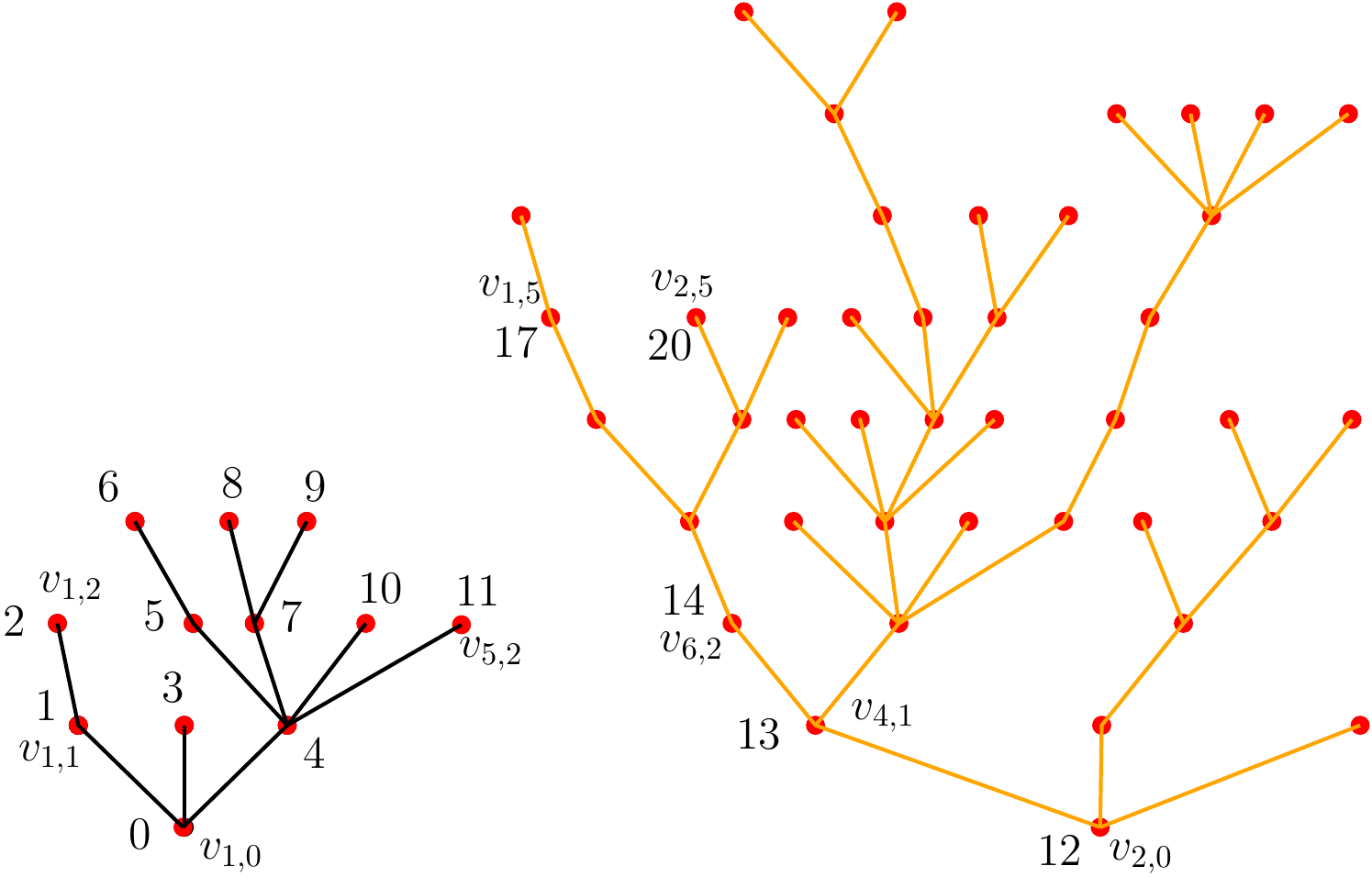}
\end{center}
\caption{A forest with its lexicographical order ($v_{i,j}$'s) and depth-first order (labels $0-14, 17$, and $20$).}
\label{fig:exploforest}
\end{figure}

Over most of the paper, we use another ordering of a forest: the depth-first order. 
This ordering is a priori unrelated to the lexicographical order above, 
but we define them jointly so that they satisfy some relations. 
They are both represented in Figure~\ref{fig:exploforest}. 
The depth-first order consists in assigning labels~$1, 2, 3, \ldots $ to the nodes of the forest
according to a depth-first search. 
The particularity here is that we do the depth-first search in such a way that, for $k\ge0$ and $1\le i<j$, the depth-first label of $v_{i,k}$ is smaller than the one of $v_{j,k}$, see Figure~\ref{fig:exploforest}.

\begin{remark} \label{deflexico} The lexicographical order gives access to a graphical construction of the forest $\mathcal{F}$. Indeed, in a discrete grid with coordinates $(i,k)$ for $i\ge1$ and $k\ge0$, put the node $v_{i,k}$ at position $(i,k)$. Let $\{\xi_{i,k}: i\ge1, k\ge0\}$ be a collection of independent random variables such that for all $i\ge1$ and $k\ge0$, $\xi_{i,k}$ has distribution $\mu_k$. We interpret $\xi_{i,k}$ as the number of children of $v_{i,k}$. Now, we construct the forest $\mathcal{F}$ by assigning $\xi_{i,k}$ edges between $v_{i,k}$ and nodes at height $k+1$ in a planar manner (i.e.~no two edges can cross each other), such that all nodes at height at least one have a parent (i.e.~we use all of them) and such that each connected component is a tree. Consequently, $(v_{i,k})_{i,k}$ corresponds to the lexicographical labelling of that forest, see Figure \ref{fig:forest} for another graphical representation. 
\end{remark}

We now define  the {\L}ukasiewicz path $(X_n)_{n\geq 0}$
and the height process $(H_n)_{n\geq 0}$ associated to the forest $\mathcal{F}$.
For all $n\geq 0$, let $L_n$ be the number of children of the node $n$, and let $H_n$ be its height, defined as the graph-distance from this node to the root of its tree. We let
\begin{align}\label{def:Luka}
X_n = \sum_{i=0}^{n-1} (L_i-1).
\end{align}

Last, we define the {\it forest explored at step $n$} as follows. Recall that we labelled the forest $\mathcal{F}$ in depth-first order, starting from 0. 
For $n\ge0$, let $i_n$ be the unique integer such that the node labelled $n$ belongs to $\mathcal{T}^{(i_n)}$. For all $n\ge0$, we define the explored forest at step $n$ as the finite forest
\begin{align}\label{def:expl_forest}
\mathcal{F}_n=\bigcup_{i=0}^{i_n} \mathcal{T}^{(i)}. 
\end{align}
In words, the explored forest $\mathcal{F}_n$ consists of the trees of $\mathcal{F}$ that have been explored or partially explored at step $n$. Note that by construction, the node $n$ is always included in $\mathcal{F}_n$. 

From now on, $\cF$ will denote an infinite sequence of i.i.d.~$\bmu$-BPVE, unless we state explicitely otherwise.

\subsection{Backbone of the proof}\label{sub:novelty}

As in the case of  Bienaym\'e-Galton-Watson trees, our proof relies on the fact that convergence for the Gromov-Hausdorff-Prokhorov metric is implied by
(and in fact equivalent to) convergence of the normalised height process to a Brownian motion 
(see, e.g.,
~\cite[Equation~(8)]{CKG}).
We prove the convergence in Theorem \ref{th:main} in three main steps:
\begin{enumerate}
\item convergence of the {\L}ukasiewicz path of the forest $\cF_n$ to a reflected Brownian motion;
\item convergence of the height process of $\cF_n$ to the same Brownian motion, up to a multiplicative constant;
\item deducing the convergence of one tree conditioned on having total population at least~$n$ from the convergence of the forest.
\end{enumerate}
Although these steps are identical to those taken in the classical proof, for constant environments, 
Steps 1.\ and 2.\ become much more involved when the environment is random.
Let us now describe a few difficulties one encounters when dealing with random environments.

\medskip
{\bf 1.\ Convergence of the {\L}ukasiewicz path.} As explained above, it is classical  to start by considering a forest, i.e.~a sequence of unconditioned copies of the random tree we want to study, and explore it depth-first using the {\L}ukasiewicz path defined in \eqref{def:Luka}.\\
The reason for introducing this process is that, in the case of a constant environment, this process is a random walk with i.i.d.~steps. Hence the step 1 above is a direct consequence of Donsker's invariance principle, and we obtain that the {\L}ukasiewicz path, properly normalised, converges to a Brownian motion.\\
In the random environment framework, this step is far from trivial. Indeed, the {\L}ukasiewicz path is still a random walk but its increments are no longer identically distributed or independent. More precisely, the distribution of the increment at a given step depends on the height of the vertex being explored, which depends on the whole past trajectory of the exploration. Thus the {\L}ukasiewicz path is no longer Markovian.

To prove convergence of the rescaled {\L}ukasiewicz path,
we use a martingale convergence theorem from~\cite[Theorem 2.1(ii)]{Whitt}, 
which is an extension of~\cite[Theorem~7.1]{EthierKurtz}.
For this theorem to apply, we need to control the global shape of the explored forest: in Section~\ref{sec:basic}, we prove that the height and width of the explored forest at step $n$ are both of order $\Theta(\sqrt{n})$. We also need that, $\bmu$-almost surely,
\begin{equation}\label{eq:crucial1}
\frac1n\sum_{i=1}^n \sigma^2_{H_i} \to \sigma^2,
\end{equation}
in probability as $n\to\infty$, recalling that $\sigma^2_k$ is the variance of $\mu_k$, for all $k\geq 0$.
For a non-varying environment, the limit above would actually be a straightforward equality, while in the case of a varying environment the sum above corresponds to the sum of the variances collected along a random path whose trajectory itself depends on those variances. Here, considering the explored forest is not enough anymore and we need to track the exact position in the forst of the node with depth-first label $n$. Establishing \eqref{eq:crucial1} is far from trivial and is the most tedious step of the proof.

\medskip
{\bf 2.\ Convergence of the height process.} 
After having proved convergence of the normalised {\L}ukasiewicz path to a reflected Brownian motion,
one needs to prove that the height process 
also converges to a Brownian motion. Recall that the height process $H_n$ at step $n$ was defined above \eqref{def:Luka} as the height of the node $n$, i.e.~its graph-distance to the root of its tree.
In the classical case of non-varying environments, this is proved in \cite[Theorem 2.2.1]{DuquesneLeGall}  by noting that:
\begin{enumerate}
\item[(a)] the height of the $n$-th node is equal to the number 
of running minima up to time~$n$ of the {\L}ukasiewicz path $(X_k)_{k\geq 0}$;
\item[(b)]   $(X_n-X_k)_{0\leq k\leq n}$ has the same distribution as $(X_k)_{0\leq k\leq n}$, for all $n\geq 0$;
\item[(c)]  consequently, the height of the $n$-th node has the same law as the number of running maxima (starting from time 0) of $(X_n-X_k)_{0\leq k\leq n}$;
\item[(d)]  the pieces of trajectory of $(X_k)$ between two maxima are i.i.d.;
\item[(e)]  the number of records of $(X_k)_{0\leq k\leq n}$ is asymptotically proportional to $X_n$ with high probability by the Law of Large Numbers.
\end{enumerate}
Crucially, the Steps (b) and (d) both fail in the case of a varying environment. 
\\
Fortunately, there is another link between $X_n$ and $H_n$, as $X_n$ counts the number of unexplored children of each ancestor of the $n$-th vertex (we detail this in the next section, see Figure~\ref{fig:spine}). It turns out that the offspring of the ancestor of the $n$-th vertex at height $k$ has a distribution close to the size-biased version of $\mu_k$, almost independently for different heights. From this, we can derive a Law of large numbers in the annealed environment - hence for almost all realisations of the environment. The combination of this link between $X$ and $H$ with a decomposition of the tree along a spine with size-biased degrees  is used in~\cite{BroutinMarckert} for trees with a prescribed degree sequence (and to some extent in~\cite{AddarioDevroyeJanson} for Galton-Watson trees, in particular in the proof of Theorem 1.2). Compared to both these works, a major additional difficulty in our case is that degrees are not exchangeable, due to the fact that the distribution of the degree of a vertex is conditioned by its height.

\medskip
{\bf 3.\ The tree conditioned to have at least $n$ vertices.}
So far, we have worked on the forest $\mathcal{F}$ where  trees of all sizes naturally appear, but now we would like to obtain information on a single tree conditioned to be large. This is done by extracting the largest excursions of the height process, corresponding to the largest trees of the forest $\cF_n$. The method developed in \cite{Aldous} and \cite{AldousLimic} in the case of non-varying environments still applies in our case. 

%
%

\subsection{Auxiliary results and conditions for varying environments}\label{sub:ideas}
In this section, we give more details on the steps 1 and 2 explained in Section \ref{sub:novelty} above. 
We start by giving deterministic conditions \ref{(I)}-\ref{(V)} on the varying environment, and we will then explain how they imply Theorem~\ref{th:main}. We will show later in Lemma~\ref{lem:deterministicconditions} that these conditions are satisfied $\bP$-a.s.~by a strictly critical BPRE. \\

We need a few more definitions in order to state our conditions. Let  $\bmu$ be a deterministic strictly critical environment, that is, for all $k\ge0$, $\mu_k$ has expectation 1 and variance $\sigma^2_k$. In that case, we define
\[
\sigma^2=\liminf_{n\to\infty}\frac{1}{n}\sum_{k=0}^{n-1}\sigma^2_k.
\]
For all $k\ge0$, let $\bar \xi_k$ be distributed according to the size-biased version of $\mu_k$, 
i.e.~for all $i\ge0$, $\Pmu(\bar \xi_k = i) = i\mu_k(\{i\})$. For all $k\ge0$, let $\zeta_k$ be a uniform random variable in $\{0, \ldots, \bar \xi_k-1\}$. Impose the independence of the family $(\zeta_k)_{k\geq 0}$.
{Following~\cite{KerstingVatutinBook}, we define, for all $k\geq 0$ and $s\in [0,1)$,
\begin{equation}\label{eq:phik}
\varphi_k(s)=\frac{1}{1-f_k(s)}-\frac{1}{1-s},
\end{equation}
where $f_k$ is the generating function of $\mu_k$, and we set $\omega_k(\varepsilon):=\sup_{1-\varepsilon\leq s\leq t< 1}\vert \varphi_k(s)-\varphi_k(t)\vert$} for all $\varepsilon>0$.\\
We are now ready to define our five conditions on the environment $\bmu$:
\begin{enumerate}[label=(\Roman*)]
\item \label{(I)}\textbf{Averaging of the variance along generations.} We have that $\sigma^2>0$ and
\[
\frac{\sigma_0^2+\ldots +\sigma_{n-1}^2}{n}\underset{n\rightarrow +\infty}{\longrightarrow} \sigma^2;
\]
\item \label{(II)} \textbf{Non-degeneracy of the offspring distributions.} There exists a constant $c>0$  such that, for all $0\leq a<b$, 
\[
\liminf_{k\rightarrow +\infty}\frac{\mu_{\lfloor ak\rfloor}(\{0\})+\ldots+ \mu_{\lfloor bk\rfloor}(\{0\})}{(b-a)k}\geq c;
\]
\item \label{(III)}\textbf{Control of the large degrees I.}  There exists a sequence $(h_n)_{n\geq 1}$ such that we have $\lim_{n\rightarrow +\infty}h_n=+\infty$ and such that, for all $\varepsilon >0$, 
\[
\frac{1}{\sqrt{n}}\sum_{k=0}^{\lfloor h_n\sqrt{n}\rfloor}\mathrm E_{\bmu}[\xi_{k}^2\mathbf{1}_{\{\xi_{k}^2\geq \varepsilon n\}}]\underset{n\rightarrow +\infty}{\longrightarrow} 0,
\]
where $\xi_{k}\sim \mu_k$ for all $k\ge 0$; 
\item \label{(IV)} \textbf{Position of the spine.} We have that, $\Pmu$-a.s.,
\[
\frac{1}{n}\sum_{k=0}^{n-1}\zeta_k\underset{n\rightarrow +\infty}{\longrightarrow}\frac{\sigma^2}{2},
\]
where $\sigma^2$ is given by \ref{(I)};
\item \label{(V)} \textbf{Control of the large degrees II.} We have that
\[
\lim_{\varepsilon\rightarrow 0} \limsup_{n\rightarrow +\infty}\frac{1}{n}\sum_{k=0}^n\omega_k(\varepsilon) =0.
\]
\end{enumerate}

We comment below on Conditions \ref{(I)}-\ref{(V)} and in particular explain how they imply our main result.

\noindent
{\bf 1.\ Convergence of the {\L}ukasiewicz path.} 
This part uses only Conditions \ref{(I)}, \ref{(II)} and \ref{(III)}. Condition \ref{(I)} ensures that a high tree should not have inhomogeneities between macroscopic slices of different heights (which might prevent convergence towards a self-similar object like the Continuum Random Tree). 
Without Condition \ref{(II)}, there could be too many offspring distributions close to a Dirac mass at~$1$. 
This would mean that the tree is likely to be long and thin as many of its vertices would have exactly one child. 
Condition \ref{(III)} prevents large degrees, which would create jumps in the scaling limit of the {\L}ukasiewicz path, in a rectangle of height and width of order at least $\sqrt{n}$ that encompasses $\cF_n$. Note that Condition \ref{(III)} is not implied by Conditions \ref{(I)} and \ref{(II)} in general.


As a first step in the proof of Theorem~\ref{th:main}, we show that the {\L}ukasiewicz path converges, after normalisation, to a Brownian motion.

\begin{theorem}\label{thm:Luka}
For any environment $\bmu$ satisfying \ref{(I)}, \ref{(II)} and \ref{(III)}, we have that, under $\Pmu$,
\[
\left(\frac{X_{\lfloor nt\rfloor}}{\sqrt{n}}\right)_{0\leq t\leq 1}\overset{(d)}{\longrightarrow}(\sigma B_t)_{0\leq t\leq 1}\]
as $n\rightarrow \infty$ in the Skorokhod space $D([0,1])$, where $B$ is a standard Brownian motion.
\end{theorem}

We prove Theorem~\ref{thm:Luka} in two steps, 
focussing on $\cF_n$, the forest $\cF$ restricted to the first $n$ vertices explored. 
We start by controlling the scaling of $\cF_n$: 
we show in Proposition~\ref{prop:firstbounds} that,under \ref{(I)} and \ref{(II)}, 
$\cF_n$ has height and width (the size of the largest generation) of order $\Theta(\sqrt{n})$. 
Then, we bootstrap this estimate. 
We use a martingale convergence theorem from~\cite[Theorem 2.1 $(ii)$]{Whitt}, 
which is an extension of~\cite[Theorem~7.1]{EthierKurtz}.
This theorem gives convergence to a Brownian motion under three relatively weak conditions:  the first condition of~\cite[Theorem 2.1 $(ii)$]{Whitt} holds if the largest degree in $\cF_n$ is $o(\sqrt{n})$, ensured in our case by Condition~\ref{(III)};
the second condition holds if one can control the maximal variance of a degree, which is implied by Condition~\ref{(I)};
the third and last condition follows from \eqref{eq:crucial1}, which is established in Lemma~\ref{mart3}. In the proof,  we divide the tree vertically and horizontally into mesoscopic boxes and control the behavior of the forest on those boxes. While this may sound as a discretized version of a proof that would be easier or more elegant on the scaling limit, we have not found such a proof (referring e.g.~to the limit given by the generation sizes of a forest of $m$ trees, with $m\rightarrow\infty$, is a Feller diffusion, see the discussion in Section~\ref{sec:disc}). In particular, one difficulty is that $\cF_n$ cuts with high probability the last tree, and that the latter has a typical volume $\Theta(n)$, as the tree containing the $n$-th vertex has a much higher probability to be large than an unconditionned tree. In other words, we would need to cut a Feller diffusion on the right, along a curve that seems difficult to handle (corresponding to the ancestry line of the $n$-th vertex of $\cF$ in the discrete setting), so that the left part has a fixed volume. This complicates significantly the matter. On the other hand, if we stop the exploration of $\cF$ after $n$ trees have been fully seen (denote $t_n$ the corresponding stopping time), then many limit theorems, including the one of~\cite{Whitt}, can not be used. Moreover, controlling the random variations of $t_n$ is not an easy task.

Below, we state the following corollary of Theorem~\ref{thm:Luka} and the continuous mapping theorem, proved in Section \ref{sec:Luka}.
Recall that the law of a Brownian motion reflected above zero is that of its absolute value and 
let $I_n:=\min_{j\leq n}X_j$, so that $X_n-I_n$ is the {\L}ukasiewicz path reflected above its running minimum. 

\begin{corollary}\label{cor:Luka}
Let $\bmu$ be a strictly critical environment satisfying Conditions \ref{(I)}, \ref{(II)} and \ref{(III)}. We have that, under $\Pmu$,
\[
\left(\frac{X_{\lfloor nt\rfloor}-I_{\lfloor nt\rfloor}}{\sqrt{n}}\right)_{0\leq t\leq 1}\overset{(d)}{\longrightarrow}(\sigma \vert B_t\vert)_{0\leq t\leq 1}
\]
as $n\rightarrow \infty$ in Skorokhod topology, where $B$ is a standard Brownian motion.
\end{corollary}

\noindent
{\bf 2.\ Convergence of the height process.} 
The second part of our proof is to deduce from Corollary~\ref{cor:Luka} the joint convergence of the (reflected) normalised {\L}ukasiewicz path and the height process to the same reflected Brownian motion, up to a multiplicative constant. For this purpose, we require all five Conditions \ref{(I)}-\ref{(V)}.

\begin{theorem}\label{thm:jointcvLukaHeight}
Let $\bmu$ be a strictly critical environment satisfying Conditions \ref{(I)}-\ref{(V)}.
Then, under $\Pmu$,
\[
\frac{1}{\sqrt{n}}(X_{\lfloor nt\rfloor}-I_{\lfloor nt\rfloor},H_{\lfloor nt\rfloor})_{0\leq t\leq 1}\overset{(d)}{\longrightarrow}\left(\sigma \vert B_t\vert,\frac{2}{\sigma}\vert B_t\vert\right)_{0\leq t\leq 1}
\]
as $n\rightarrow \infty$ in Skorokhod topology for both coordinates, where $B$ is a standard Brownian motion.
\end{theorem}
As explained in  Section~\ref{sub:novelty}, the classical approach used for Bienaym\'e-Galton-Watson trees  does not apply, since the law of the $n$-th increment of the {L}ukasiewicz path for $n\geq 0$ depends potentially on the whole past of the trajectory. However, there is still one Markovian object in the exploration: the spine between the $n$-th vertex and the root of its tree, keeping track of the number of offspring of each of its ancestors that have not yet been explored (the number of red vertices at each generation in Figure~\ref{fig:spine}). The sequence of successive spines for $n\geq 0$ is a discrete version of the L\'evy snake introduced in~\cite{LeGallsnake}. 

\begin{figure}
\begin{center}
\includegraphics[width=14cm]{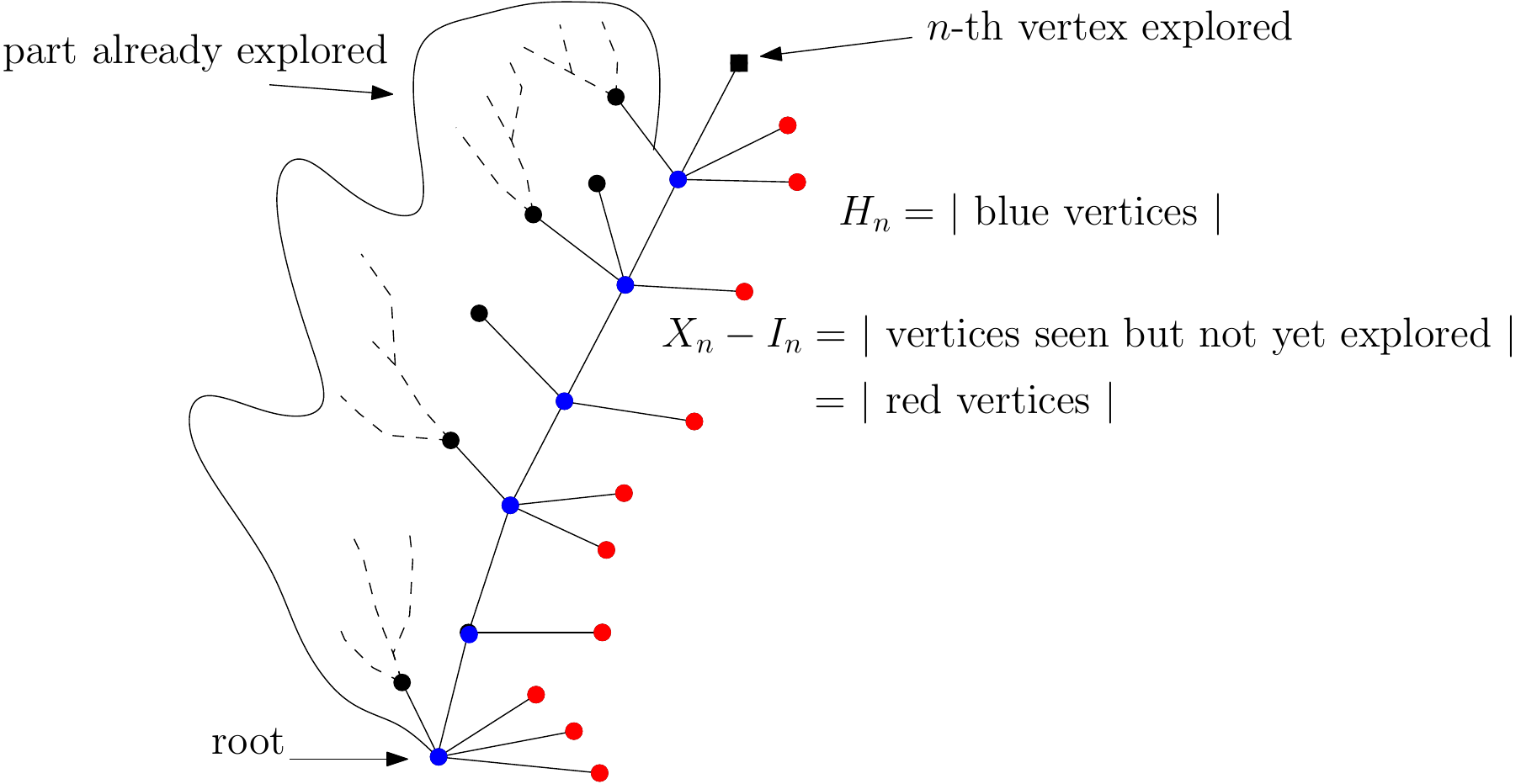}
\caption{Picture of the tree being currently explored. Vertices explored earlier in the depth-first order are on the left. The subtrees of the black circle vertices (dashed) have already been fully explored when visiting the $n$-th vertex of the forest (the black square).}
\label{fig:spine}
\end{center}
\end{figure}

The spine carries an important relation between $H_n$ and $X_n-I_n$, as illustrated in Figure~2. The idea is to show that the ancestors of the $n$-th vertex (blue vertices on Figure~\ref{fig:spine}) 
have  essentially independent offspring, with a size-biased distribution, as vertices with a larger offspring are more likely to have more descendants. This defines a spine from the root to the $n$-th explored vertex.
Intuitively, once a vertex $y$ at height $k$ is known to be along the spine, the next vertex in the spine is chosen uniformly among its children. In average, such a vertex $y$ has $\sigma_k^2+1$ children, as its offspring is size-biased. The two precedent facts put together imply that, in average, among the children of $y$, $\sigma_k^2/2$ will be on the left of it (black vertices) and $\sigma_k^2/2$ will be on the right of it (red vertices). This is rigorously done in Section~\ref{subsec:spinedecompoheight}. Averaging over $k$ and using Condition~\ref{(I)}, this heuristic explains why the limiting ratio between $X_n-I_n$ and $H_n$ is $\sigma^2/2$. For standard Galton-Watson tree, 
Note that, with $m=h(y)$, $\bar \xi_m$ defined
before Condition {\ref{(IV)}} stands for the number 
of children of the ancestor of the $n$-th vertex at height $m$, 
and $\zeta_m$ stands for the number of these children that are on the right of the spine, that is, black vertices on Figure \ref{fig:spine}.
\\
The delicate part of the proof is to replace the ``essentially independent offspring'' 
in the heuristic argument above by a rigorous claim. 
To do so, we show that the distribution of $\cT$ is close to that of $\cT^*$, the tree built by growing independent BPVE trees along an infinite spine with size-biased offspring (i.e.\ Kesten's tree in the varying environment, see Section~\ref{subsec:spinedecompoheight} for definition and details). 
On $\cT^*$, the ancestors of a given vertex have the desired distribution. 
This allows us to show that, for most vertices $x$ of $\cT$ at some large height (conditionally on $\cT$ reaching that height), the ratio $(X_x-I_x)/H_x$ between the reflected {\L}ukasiewicz path and the height process at $x$ is arbitrarily close to $\sigma^2/2$ (Proposition~\ref{prop:fewbadvertices}).
\\
From this, we prove that, for an arbitrarily large proportion of vertices of $\cF_n$, 
this ratio is close to $\sigma^2/2$. 
We conclude by a continuity argument to extend the convergence to the whole forest $\cF_n$ in Section~\ref{sect:prooflukaheight}. We need fine tail bounds on the height of $(\mu_k)_{k\geq k_0}$-BPVEs, uniformly in $k_0=O(\sqrt{n})$, to control increments of the form $H_i-H_j$ with $i-j\lesssim \varepsilon n$, uniformly in the generation $H_j$ (in Lemma~\ref{KingCharles}, with auxiliary estimates from Lemma~\ref{lem:heightsurvivalproba}).
\\
\\
The idea of exploiting this connection between $X_n-I_n$ and $H_n$ by approximating the distribution of the spine has been used in~\cite{BroutinMarckert} for finding the scaling limit of trees with a prescribed degree sequence (this also applies to standard Galton-Watson trees). We explain that we cannot resort to the same arguments, mainly because in our setting, degrees are not exchangeable due to the varying environment and we lose a lot of symmetry. For instance, to link the finite dimensional marginal of $(H_n)$ to the spine of Kesten's tree, the main difficulty is 
that there is no clear way to say something about a vertex selected uniformly at random in $\cF_n$. Also, it is a priori very hard to condition our tree to have (exactly or at least) a given large height, which is why we resort to $\cF_n$. By contrast, this information comes for free in a tree with prescribed degree sequence. Finally, the idea to show the tightness of $H$ in~\cite{BroutinMarckert} (inspired by~\cite{AddarioDevroyeJanson}) is to dominate the variations of $H$ by those of $X-I$ (which is tight, by the convergence of the {\L}ukasiewicz). It is crucial to have a good control of the location of the vertices with only one child (as they add an increment in $H$, but not in $X-I$). Due to the exchangeability, these nodes can be inserted ''uniformly at random'' in the skeleton made by the nodes of degree at least 2 (see the multinomial formula in the proof of Lemma 8 in \cite{BroutinMarckert}). This does not hold in our setting, as these nodes of degree 1 appear with different probabilities at each generation. 

\subsection{Discussion of the existing literature on BPVEs and BPREs}\label{sec:disc}
We give a brief review of existing results on BPVEs and BPREs; we refer the reader to Kersting and Vatutin's book~\cite{KerstingVatutinBook} for more details.

BPVEs have been investigated for more than half a century. First works focussed on the survival or extinction, and on the long-term behaviour. 
Lindvall~\cite{Lindvall} proved that, as long as the varying environment is non-degenerate,
i.e.\ $\mathbb P(\xi_n^{(1)} = 0)<1$ for all $n\geq 0$, 
then $Z_n$ converges almost surely to a possibly infinite random variable $Z_\infty$ 
(convergence in distribution was already known, see~\cite{AthreyaNey71, Church71, Jagers74}).
Lindvall also showed that, under some conditions on the varying environment, $0<Z_\infty<\infty$, which means that a BPVE does not obey the usual dichotomy of extinction or divergence to infinity. 
Jagers~\cite{Jagers74}, and later Kersting~\cite{Kersting20}, are interested in regular BPVEs, i.e.~when $Z_\infty\in \{0,+\infty\}$ almost surely. They classify these regular BPVEs into sub-critical, critical and super-critical (see, e.g.,~\cite[Proposition~1]{Kersting20}).

The first scaling limits to have been obtained for BPVEs were the limits of generation sizes: one starts with a large initial population and tracks the evolution of its size generation after generation. This does not provide information on the genealogical structure of the forest.
Such limits are the most accessible. First results track back to Kurtz~\cite{Kurtz78} and Helland~\cite{Helland}. 
They show that the scaling limit for the successive generation sizes is a Feller diffusion both for strictly critical BPREs and critical Bienaym\'e-Galton-Watson trees. While a third moment assumption was required in~\cite{Kurtz78}, these results have been considerably extended in the last decade, to cases with infinite variance, and bottlenecks (corresponding to a catastrophy in the environment killing almost all the population). See e.g., Bansaye and Simatos~\cite{BansayeSimatos} and Fang, Li and Liu~\cite{Lietal})
Contrary to Bienaym\'e-Galton-Watson trees, the limiting process is then no longer a Continuous State Branching Process and can be difficult to characterize.

In the last couple of years,  finer scaling limits have been proved for regular, critical BPVEs conditioned to survive, showing that these trees 
behave somewhat similarly to critical Bienaym\'e-Galton-Watson trees. 
On the one hand, Cardona-Tob\'on and Palau~\cite{CordonaPalau} give an asymptotic result 
for the size of the $n$-the generation of a regular, critical BPVE conditioned on non-extinction.
On the other hand, Kersting~\cite{Kersting21} shows that the height of the last common ancestor of all individuals at generation~$n$ is asymptotically uniform in $[0,n]$.

Weeks before this preprint was finished, more results were published in the same direction. Harris, Palau and Pardo~\cite{HarrisPalauPardo} analyze the genealogy of an arbitrary fixed number $k\in \dN$ of individuals chosen uniformly in the $n$-th generation, conditionally on survival until that height. They draw $k$ spines through the first $n$ generations of the tree, extending the 2-spine decomposition in~\cite{CordonaPalau}. Boenkost, Foutel-Rodier and Schertzer~\cite{BoenkostFRSchertzer}, using a spinal decomposition similar to~\cite{FRSchertzer}, obtain the genealogy in Gromov-Hausdorff-Prokhorov metric of the first $n$ generations towards the (possibly time-changed) Brownian coalescent point process. The latter tracks where the spines meet, i.e.~where pairs of individuals of generation $n$ have their last common ancestor.
All these results (for critical BPVEs in~\cite{HarrisPalauPardo}, and for nearly-critical BPVEs in~\cite{BoenkostFRSchertzer}, both under a second moment assumption) are consistent with the fact that a strictly critical BPRE converges to the continuum random tree, as we prove in our main result.
\\
Let us stress that our results and approach are conceptually different: in these papers, the idea is to sample several coalescing spines of fixed length (forward in time in~\cite{HarrisPalauPardo}, backward in~\cite{BoenkostFRSchertzer}) to analyze the genealogy of the tree. In our work, we track the evolution of a dynamic spine, with varying length, which follows the depth-first exploration of the tree. This is in the spirit of the works of Duquesne, Le Gall and Le Jan~\cite{LeGallsnake,LeGallLeJanExplo,DuquesneLeGall}.

\subsection{Open questions}\label{sub:open}

{\bf Deterministic conditions on the environment $\bmu$.} As mentioned in Section~\ref{sub:ideas}, Theorem~\ref{th:main} is valid under Conditions \ref{(I)}-\ref{(V)}. These conditions are satisfied by the branching process in random environment we consider, but it would be nice to find necessary and sufficient conditions on the deterministic varying environment $\bmu$. 
Condition \ref{(I)} ensures that the laws of successive generations are not too erratic, and Condition \ref{(II)} rules out cases where the tree would have long portions with vertices with only one offspring, which would take us out of the Brownian regime. These conditions could perhaps be relaxed by normalizing the height as a function of the cumulant sum of the variances on successive generations, as in the assumption (2) of~\cite{BoenkostFRSchertzer}, which may accordingly stretch the limiting object. Condition \ref{(III)} is essentially necessary in order to not see vertices with large degrees in the limit, and it should be noted that it is not a consequence of Conditions \ref{(I)} and \ref{(II)}. Note that 
Condition \ref{(III)} corresponds to assumption (3') in~\cite{BoenkostFRSchertzer}.
On the other hand, Conditions \ref{(IV)} and \ref{(V)} seem to be of technical nature. In fact, one can show that Condition \ref{(V)} is not necessary once one already has Conditions \ref{(I)}, \ref{(II)} and \ref{(III)}: we only use it to estimate the probability that a $\bmu$-BPVE survives at least $m$ generations for large $m$, in \eqref{eqn:heightsurvivalproba}. The latter can be shown by adapting the proof of the Kolmogorov estimate in~\cite{BoenkostFRSchertzer} (Theorem 4.1 - taking $m=\kappa_N$, one can replace $\sqrt{\kappa_N}$ by $\kappa_N$ in Lemma~4.3, using Condition \ref{(II)} to get a lower bound on the $\varphi_k(0)$'s in the first equation p.17).  

\medskip\noindent
{\bf Extension to other regimes.} As mentioned in Section~\ref{subsec:critstrictcrit}, in the non-strictly critical regime, the structure of $\cT^{\bmu}$ heavily depends on the environment $\bmu$, in particular on the sequence $(\log \overline{\mu_k})_{k\geq 0}$.  Let $S_i:=\log\overline{\mu_1}+\ldots \log\overline{\mu_i}$, for $i\ge0$. In the strictly critical regime, $(S_i)$ is a.s.~identically equal to $0$, while in the non-strictly critical case, $(S_i)$ is a random walk with i.i.d.~increments. Some finite but rather long portions of its trajectory (of order $i$ on $[0,i]$ if the increments of $(S_i)$ have a finite second moment, for instance) look as if it were induced by a super-critical environment, and a tree can typically exhibit super-polynomial growth on the corresponding generations, so that the Continuum Random Tree cannot be its scaling limit.  One can ask what would be good conditions on $(S_i)$ in order to still have convergence to the CRT, or a slightly modified version of it. For instance, the aforementioned result of~\cite{BoenkostFRSchertzer} holds for BPVEs in the near critical regime, i.e.~when $(S_i)$ is bounded, hence one can conjecture that the scaling limit of a large tree would be an  untruncated version of a Coalescent Point Process.  
\\
Moreover, one could investigate distributions with weaker moment assumptions, for instance heavy-tailed offspring distributions having an infinite variance. For standard Bienaym\'e-Galton-Watson trees, this has already been done in~\cite{DuquesneLeGall,Duquesne03,Igor13}. In varying or random environment, only the scaling limits the successive generation sizes are known. One intriguing question, asked by Bansaye and Simatos in~\cite{BansayeSimatos} (Section 2.5.1), is what happens when offspring distributions have different heavy tails.

\medskip\noindent
{\bf Conditioning on the total population size to be exactly~$n$ instead of at least~$n$.}
Generally speaking, when studying scaling limits of critical random trees, it is noticeably more delicate to condition the total population size to be \textit{exactly}~$n$. 
%
For usual critical Bienaym\'e-Galton-Watson trees with finite variance, this was proved by Aldous~\cite[Section~5]{AldousIII}, 
and later Marckert and Mokkadem~\cite{MarckertMokkadem} proposed a simpler proof,
but under the assumption that the offspring distribution has exponential moments.
The case of critical Bienaym\'e-Galton-Watson trees with heavy-tailed offspring distributions was treated by Duquesne~\cite{Duquesne03} and Kortchemski~\cite{Igor13}.
The arguments in the works mentioned above rely heavily on  exact algebraic or distributional identities that do not hold in varying environment. Hence, one would need here again to come up with alternative arguments in order to obtain this refined conditioning.\\
 To conclude, we believe that our result holds under this stronger conditioning, but its proof would require significant effort, hence we choose to leave this question for further work.


\subsection{Plan}
In Section~\ref{sec:basic}, we prove preliminary results 
whose counterparts for critical Bienaym\'e-Galton-Watson trees are either straightforward or classic, but not for a BPRE. 
In particular, Lemma~\ref{lem:deterministicconditions} states that Conditions \ref{(I)}-\ref{(V)} hold $\bP$-a.s., and Proposition~\ref{prop:firstbounds} states that $\cF_n$ has height and width $\Theta(\sqrt{n})$.
In Section~\ref{sec:Luka}, we prove Theorem~\ref{thm:Luka}. 
Finally, in Section~\ref{sec:heightprocess}, we prove Theorem~\ref{thm:jointcvLukaHeight}, 
and conclude the proof of Theorem~\ref{th:main}.

\section{Basic properties}\label{sec:basic}
The main results of this section are Lemma \ref{lem:deterministicconditions} and Proposition~\ref{prop:firstbounds}.

\subsection{A strictly critical BPRE satisfies Conditions \ref{(I)}-\ref{(V)}}
{Before proving in Lemma \ref{lem:deterministicconditions} that a BPRE satisfies almost surely our deterministic conditions, we provide below a technical lemma which will be useful in different places.}

\begin{lemma}\label{lem:deterministicLLNcuttopieces}
Let $(s_k)_{k\geq 0}$ be a sequence such that $\ell:=\lim_{k\rightarrow +\infty}\frac{s_1+\ldots +s_k}{k}$ exists and is finite. For all $0\le a<b$, 
\[
\frac{s_{\lfloor ak\rfloor+1}+s_{\lfloor ak\rfloor+2}+\ldots +s_{\lfloor bk\rfloor}}{(b-a)k}\underset{k\rightarrow+\infty}{\rightarrow}\ell.
\]
Moreover, for any $c>0$ and $M\in \dN$,
\[
\max_{1\leq j\leq M}\left\vert\frac{s_{\lfloor (j-1)ck\rfloor+1}+\ldots +s_{\lfloor jck\rfloor}}{ck}-\ell\right\vert\underset{k\rightarrow+\infty}{\rightarrow}0.
\]
\end{lemma}

\begin{proof}
Let us fix $0\le a<b$. Fix $\varepsilon >0$.  There exists $K_{\varepsilon}\in\mathbb{N}$ such that for all  $k\ge K_{\varepsilon}$, and for $c\in\{a,b\}$, $\vert s_1+\ldots +s_{\lfloor ck\rfloor} -\ell ck\vert \leq \varepsilon (b-a)k/2$, as can be shown by a straightforward computation, using the definition of $\ell$. 
Using the triangle inequality, we obtain
\begin{align*}
\vert s_{\lfloor ak\rfloor+1}+s_{\lfloor ak\rfloor+1}+\ldots +s_{\lfloor bk\rfloor}-\ell(b-a)k \vert \leq &\vert s_1+\ldots +s_{\lfloor ak\rfloor} -\ell ak\vert 
\\
&+\vert s_1+\ldots +s_{\lfloor bk\rfloor} -\ell bk\vert
\\
\leq &\varepsilon (b-a)k.
\end{align*}
The above implies that for all $\varepsilon>0$, there exists $K_\varepsilon\in\mathbb{N}$ such that, for all $k\ge K_\varepsilon$, we have
\[
\left\vert \frac{s_{\lfloor ak\rfloor+1}+\ldots +s_{\lfloor bk\rfloor}}{(b-a)k}-\ell\right\vert \leq \varepsilon.
\]
This proves  the first statement. 
The second statement is a straightforward consequence of the first one by choosing $a=(j-1)c$, $b=jc$, for $1\leq j\leq M$.
\end{proof}

Recall that, as defined in Section \ref{sec:def}, under $\mathbf{P}$, $\bmu$ is a sequence of i.i.d.~random distributions, and that $\bar\mu_k$ denotes the average of $\mu_k$ and $\sigma^2_k$ its variance, for $k\ge0$.

\begin{lemma}\label{lem:deterministicconditions}
Consider a $\mathbf{P}$-BPRE such that 
$\bar\mu_0 = 1$ $\mathbf{P}$-almost surely, and $\sigma^2:={\bf E}[\sigma_0^2]<\infty$ $\mathbf{P}$-almost surely.
Then, $\bmu$ satisfies  Conditions~{\ref{(I)}}-{\ref{(V)}}  $\bP$-almost surely, where the constant $c>0$ in Condition \ref{(II)} is deterministic under $\bP$.
\end{lemma}

\begin{proof}
{\bf Condition \ref{(I)}.} Under $\bP$, $(\sigma_k^2)_{k\ge0}$ is a sequence of i.i.d.~random variables with mean $\sigma^2<\infty$, hence it satisfies the Law of Large Numbers and therefore Condition \ref{(I)} holds $\bP$-almost surely.

{\bf Condition \ref{(II)}.} We will prove a statement that is stronger than Condition~\ref{(II)}, which is that there exists a constant $c>0$ such that $\bP$-a.s., for $0\le a<b$,
\begin{align}\label{cond22}
\lim_{k\rightarrow +\infty}\frac{\mu_{\lfloor ak\rfloor}(\{0\})+\ldots+ \mu_{\lfloor bk\rfloor}(\{0\})}{(b-a)k}= c.
\end{align}
First, note that it is enough to prove that this statement holds $\bP$-a.s.~for all pairs $(a,b)$ of  rationals. By union bound, it is enough to prove that for a pair of rationals $(a,b)$, \eqref{cond22} holds $\bP$-almost surely.\\
By definition of $\bP$ and $\bmu$, we have that $\mu_0(\{0\})\in[0,1]$ $\bP$-almost-surely, and  $(\mu_k(\{0\})_{k\ge0}$ is thus a sequence of bounded random variables. Hence, using  the Law of Large Numbers, Lemma \ref{lem:deterministicLLNcuttopieces} and  the fact that $\mu_{\lfloor ak\rfloor}(\{0\})$ is bounded by 1, we have that \eqref{cond22} is satisfied with $c=\bE[\mu_0(\{0\})]$. To conclude that Condition \ref{(II)} holds, it remains to prove that $\bE[\mu_0(\{0\})]>0$. For this purpose, assume by contradiction that $\bE[\mu_0(\{0\})]=0$. As $\mu_0(\{0\})\ge0$ and $\bar\mu_0=1$ $\bP$-a.s., this implies that $\mu_0(\{0\})=0$ $\bP$-a.s.~and thus that $\mu_0(\{1\})=1$ $\bP$-almost surely by strict criticality of $\mu_0$. This yields that $\sigma_0^2=0$ $\bP$-a.s., which contradicts that $\sigma^2=\bE[\sigma^2_0]>0$. This concludes the proof that Condition \ref{(II)} holds almost surely.

{\bf Condition \ref{(III)}.} Let us fix $\varepsilon>0$. 
For all $k\geq 0$ and $n\geq 1$, let 
\[e_{k,n}:=\mathrm E_{\bmu}[\xi_k^2\mathbf{1}_{\xi_k^2\geq \varepsilon n}].\]
We start by the following fact: since, for all $n\ge0$, the random variables $e_{k,n}$, $k\ge0$, are i.i.d.~and since $\xi_0^2$ is $\mathbb{P}$-integrable, we have that
\[
\lim_{n\to\infty}\bE[e_{k,n}] =\lim_{n\to\infty}\bE[e_{0,n}]=\lim_{n\to\infty}\mathbb E[\xi_0^2\mathbf{1}_{\xi_0^2\geq \varepsilon n}]=0.
\]
Hence, for all $m\ge1$, there exists $n_m\in\mathbb{N}$ such that
\[
\bE[e_{0,n_m}]\le \frac{1}{m^2}.
\]
Note again that $(e_{k,n_m})_{k\ge0}$ is an i.i.d.~sequence of integrable random variables, hence it satisfies the Law of Large Numbers. Moreover, for all $k\ge0$ and for all $n\ge n_m$, we have $e_{k,n}\le e_{k,n_m}$. Therefore,  we have that $\mathbf{P}$-almost surely, there  exists a random (but $\bmu$-mesurable) increasing sequence $(N_m)_{m\geq 1}$ of natural integers such that for all $m\ge1$ and all $n\ge N_m$, we have
\[
\inf_{h\ge1}\frac{1}{h\sqrt{n}}\sum_{k=0}^{\lfloor h \sqrt{n}\rfloor} e_{k,n} \leq
\inf_{h\ge1}\frac{1}{h\sqrt{n}}\sum_{k=0}^{\lfloor h\sqrt{n} \rfloor} e_{k,N_m} 
\leq \frac{2}{m^2}.
\]
For all $n\ge 1$, let $h_n$ be the unique index such that $N_{h_n} \le n<N_{h_n+1}$, letting $N_0=0$. Note that, because $N_m<\infty$ $\mathbf{P}$-a.s.~for all $m\ge0$, we have that $h_n$ goes to infinity with $n$ and $h_n\ge1$ for all $n\ge N_1$. Hence, all the previous implies that, $\mathbf{P}$-a.s., we have for all $n\ge N_1$,
\[
\frac{1}{h_n\sqrt{n}}\sum_{k=0}^{\lfloor h_n \sqrt{n}\rfloor} e_{k,n} 
\leq \frac{2}{h_n^2}.
\]
Altogether, we prove that for all $\varepsilon>0$,  there exists $\mathbf{P}$-a.s.~a sequence of real numbers $(h_n)_{n\geq 1}$ diverging to infinity such that
\[
\lim_{n\to\infty} \frac{1}{\sqrt{n}}\sum_{k=0}^{\lfloor h_n \sqrt{n}\rfloor} e_{k,n} 
\leq \lim_{n\to\infty} \frac{2}{h_n}=0,
\]
which concludes the proof of that Condition \ref{(III)} holds $\mathbf{P}$-almost surely.

{\bf Condition \ref{(IV)}.} 
Recall that, for $k\ge0$, $\bar \xi_k$ denotes a random variable distributed under the size-biased version of $\mu_k$, and that $\zeta_k$ is uniform in $\{0,\ldots, \xi_k-1\}$. Under $\Pmu$, the $\bar \xi_k$'s have different distributions, and we cannot apply the Law of Large Numbers. 
However,  the $\bar \xi_k$'s are i.i.d.~under $\dP$.
We let $\xi_0$ be a random variable of distribution $\mu_0$. 
With this notation, for all $k\geq 0$,
\begin{center}
$ \dE[\bar \xi_k] = \dE[\bar \xi_0]
= \bE[\Emu[\bar \xi_0]]
=\bE[\Emu[\xi_0^2]]
=\bE[\text{Var}_{\bmu}(\xi_0)+\Emu[\xi_0]^2]
=\bE[\sigma_0^2+1]=\sigma^2+1$,
\end{center}
so that, for all $k\geq 0$,
\[\dE[\zeta_k]=\frac{\dE[\bar \xi_k]-1}{2}=\frac{\sigma^2}{2}.\]
By the strong Law of Large Numbers applied to the sequence $(\zeta_k)_{k\geq 0}$, 
\[
1=\dP\bigg(\lim_{n\to+\infty}n^{-1}\sum_{k=0}^{n-1}\zeta_k
=\sigma^2 /2\bigg)
=\bE\left[\Pmu\left(\lim_{n\to+\infty}n^{-1}\sum_{k=0}^{n-1}\zeta_k=\sigma^2 /2\right)\right].
\] 
Hence, $\Pmu(n^{-1}\sum_{k=0}^{n-1}\zeta_k\longrightarrow\sigma^2/2)= 1$  $\bP$-a.s., as required.

{\bf Condition \ref{(V)}.} Recall the definition of $\varphi_k$ and $\omega_k$ in \eqref{eq:phik} and below it. We have that, $\bP$-a.s., for all $k\ge0$,
\begin{equation}\label{augustoislate}\begin{split}
\varphi_k(0)&\ge \mu_k(\{0\}), \quad \varphi_k(1):=\lim_{s\uparrow1}\varphi_k(s)= \frac{\sigma_k^2}{2}\quad \text{ and }\\
\frac12\varphi_k(0)&\leq \varphi_k(s)\leq 2\varphi_k(1),\quad \text{ for all }s\in[0,1],
\end{split}
\end{equation}
where the second line is given by \cite[Proposition~1.4]{KerstingVatutinBook}, and the first line can be easily proved by a straightfoward computation,  and by using Taylor expansion together with the fact that $\Emu[\xi_k]=1$ $\bP$-a.s., see \cite[Equation (1.9)]{KerstingVatutinBook} for a similar computation. Using the definition of $\omega_k(\cdot)$ and the triangle inequality, this implies that, $\bP$-a.s., for all $\varepsilon\in[0,1]$ and for all $k\ge0$, $0\le \omega_k(\varepsilon)\le 2\sigma_k^2$ and thus $\omega_k(\varepsilon)$ is integrable.\\
For all $\varepsilon\in[0,1]$,  $(\omega_k(\varepsilon))_{k\ge0}$ is a sequence of i.i.d.~random variables, thus the strong Law of Large Numbers implies that, $\bP$-a.s.,
\begin{align*}
\lim_{n\to+\infty} \dfrac{\sum_{k=0}^n \omega_k(\varepsilon)}{n}& = \bE[\omega_0(\varepsilon)]\le 2\sigma^2.
\end{align*}
Now, note that $\omega_0(\cdot)$ is non-decreasing, hence the sequence $(\omega_0(1/m))_{m\ge1}$ is non-increasing and converging to $0$ $\bP$-a.s.~by \eqref{augustoislate}. By the monotone convergence theorem, we have that $\bE[\omega_0(1/m)]$ goes to $0$ as $m$ goes to infinity, and thus
\begin{align*}
\lim_{\varepsilon\downarrow0} \limsup_{n\to+\infty} \dfrac{\sum_{k=0}^n \omega_k(\varepsilon)}{n}& =0.
\end{align*}
This concludes the proof that Condition \ref{(V)} holds $\bP$-almost surely.
\end{proof}

\subsection{Survival up to some height}\label{subsec:survivalheight}
The following lemma will be useful in Section~\ref{sec:heightprocess}.

\begin{lemma}\label{lem:heightsurvivalproba}
Let $\bmu$ be a strictly critical environment such that Conditions \ref{(I)},\ref{(II)} and \ref{(V)} hold. 
Let $\mathcal T = \mathcal T^{\bmu}$ be a $\bmu$-BPVE.
Then, for all $\eta >0$, there exists $\varepsilon_0=\varepsilon_0(\bmu,\eta)>0$ and $M=M(\bmu,\eta,\varepsilon_0)$ such that for all $\varepsilon\in (0,\varepsilon_0)$ and $m\geq M$,
\begin{align}
\label{haaaaaaa}
&\frac{1}{2\sigma^2 m}\le \Pmu(h(\cT)\geq m)\le \frac{4}{cm},\\
\label{eqn:heightsurvivalproba}
&\Pmu(h(\cT)\geq (1+\varepsilon)m)
\geq (1-\eta)\Pmu(h(\cT)\geq m), \quad \text{ and }\\ 
\label{eqn:ZmTmnottoosmall}
&\Pmu(Z_m(\cT)\leq \varepsilon m\,\vert\, h(\cT)\geq m)
\leq \eta.
\end{align}
\end{lemma}

\begin{proof}
Remark that for all $m\geq 1$, $\varepsilon \mapsto \Pmu(h(\cT)\geq (1+\varepsilon)m) $ is non-increasing and $\varepsilon \mapsto  \Pmu(Z_m(\cT)\leq \varepsilon m\,\vert\, h(\cT)\geq m)$ is non-decreasing, so that it is enough to prove that the statements hold for $\varepsilon_0$ instead of $\varepsilon \in (0,\varepsilon_0)$.
\\
We start with the proof of \eqref{haaaaaaa}.
Fix $\eta >0$.
For all $k, m\geq 1$, we let $q_{k,m}:=\mathrm{P}_{(\mu_j)_{j\geq k}}(h(\cT)< m-k)$ 
be the $\Pmu$-probability that a given vertex at height~$k$ 
does not have descendants at height~$m$. 
By~\cite[Equation (2.11)]{KerstingVatutinBook}, we have
\begin{equation}\label{eqn:heightsurvivalexactformula}
\Pmu(h(\cT)\geq m)
=\frac{1}{1+\sum_{k=0}^{m-1} \varphi_k(q_{k,m})},
\end{equation}
where $\varphi_k$ is defined in~\eqref{eq:phik}.
Using~\eqref{augustoislate} together with Conditions~\ref{(I)} and~\ref{(II)}, we have that there exists $M_1=M_1(\bmu)$ such that, for all $m\ge M_1$,
\begin{equation}\label{eqn:heightsurvivalbounds}
\begin{split}
\sum_{k=0}^{m-1} \varphi_k(q_{k,m}) &\ge \tfrac12 \sum_{k=0}^{m-1}\varphi_k(0)\ge \tfrac12 \sum_{k=0}^{m-1}\mu_k(0)\ge \frac{c}{4}m,\\
1+\sum_{k=0}^{m-1} \varphi_k(q_{k,m})&\leq
1+2\sum_{k=0}^{m-1}\varphi_k(1)\le 2\sigma^2m,
\end{split}
\end{equation}
where $c$ is the constant provided by Condition \ref{(II)}. This implies that, for all $m\ge M_1$,
\[
\frac{1}{2\sigma^2 m}\le \Pmu(h(\cT)\geq m)\le \frac{4}{cm}
\]
and \eqref{haaaaaaa} follows. We now turn to the proof of \eqref{eqn:heightsurvivalproba}.
For all $m\ge0$ and $\varepsilon\in(0,1/2)$, we define
\begin{align*}
S_m&=\sum_{k=0}^{m-1} \varphi_k(q_{k,m}), \quad S_{(1+\varepsilon)m}=\sum_{k=0}^{\lfloor (1+\varepsilon)m\rfloor-1} \varphi_k\left(q_{k,\lfloor(1+\varepsilon)m\rfloor}\right),\\
A_{\varepsilon,m}&=\sum_{k=0}^{m-1}(\varphi_k\left(q_{k,\lfloor(1+\varepsilon)m\rfloor}\right)-\varphi_k(q_{k,m})),\quad\text{ and }\\
B_{\varepsilon,m}&=\sum_{k=m}^{\lfloor(1+\varepsilon)m\rfloor-1}\varphi_k\left(q_{k,\lfloor(1+\varepsilon)m\rfloor}\right).
\end{align*}
By \eqref{eqn:heightsurvivalexactformula}, we have that
\[
\Pmu(h(\cT)\geq  m)-\Pmu(h(\cT)\geq (1+\varepsilon) m)
{\le}\frac{A_{\varepsilon,m}+B_{\varepsilon,m}}{(1+S_m)(1+S_{(1+\varepsilon)m})}.
\]
Moreover, \eqref{augustoislate} yields that
\begin{equation}\label{eq:boundingB}
\limsup_{m\to\infty}\frac1m B_{\varepsilon,m}\leq \limsup_{m\to\infty} \frac1m\sum_{k=m}^{\lfloor(1+\varepsilon)m\rfloor-1}\sigma_{k}^2= \varepsilon \sigma^2,
\end{equation}
where we used Condition \ref{(I)} and Lemma \ref{lem:deterministicLLNcuttopieces} for the last equality. 
Thus, there exists $M_2=M_2(\bmu,\varepsilon)$ such that for all $m\ge M_2$, 
$B_{\varepsilon,m}\leq 2\varepsilon\sigma^2m$. 
Furthermore, by \eqref{eqn:heightsurvivalbounds}, we have that, for all $m\ge M_1$,
\[
(1+S_m)(1+S_{(1+\varepsilon)m})\geq \frac{c^2 m^2}{16}.
\] 
Therefore, for all $m\ge M_1\vee M_2$, we have that
\begin{equation}\label{eqn:lemmaheightsurvival1}
\Pmu(h(\cT)\geq  m)-\Pmu(h(\cT)\geq (1+\varepsilon)m)\leq 16\frac{A_{\varepsilon,m}}{c^2m^2}+32\frac{\varepsilon \sigma^2}{c^2 m}.
\end{equation}
We now aim at bounding $A_{\varepsilon, m}$ from above.
For all $k\ge1$ and $m\geq 1$, let
\begin{equation}\begin{split}
A_{\varepsilon,m}^{(1)}
&:=\sum_{k=0}^{ \lfloor(1-\varepsilon)m\rfloor-1}(\varphi_k(q_{k,\lfloor(1+\varepsilon)m\rfloor})-\varphi_k(q_{k,m})),\\
A_{\varepsilon,m}^{(2)}
&:=\sum_{k=\lfloor(1-\varepsilon)m\rfloor}^{ m-1}(\varphi_k(q_{k,\lfloor(1+\varepsilon)m\rfloor})-\varphi_k(q_{k,m})),
\end{split}
\end{equation}
so that $A_{\varepsilon,m}=A_{\varepsilon,m}^{(1)}+A_{\varepsilon,m}^{(2)}$.
First, proceeding as for $B_{\varepsilon,m}$ above, we obtain that there exists $M_3=M_3(\bmu,\varepsilon)$ such that, for all $m\ge M_3$,
\begin{equation}\label{eqn:lemmaheightsurvival2}
A_{\varepsilon,m}^{(2)} \leq 2\varepsilon\sigma^2m.
\end{equation}
Second,  by \eqref{augustoislate} and \eqref{eqn:heightsurvivalexactformula}, we have that
for all $k\leq (1-\varepsilon)m$, 
\begin{align*}
q_{k,m}
&=1-\mathrm{P}_{(\mu_j)_{j\geq k}}(h(\cT)\geq m-k) \\ 
&= 1- \frac{1}{1+\sum_{j=k}^{m-1} \varphi_j(q_{j,m})}\\
&\geq 1- \frac{1}{1+\tfrac12  \sum_{j=\lfloor (1-\varepsilon)m\rfloor}^{m-1} \mu_j(\{0\})}.
\end{align*}
Hence, by Condition \ref{(II)}, there exists $M_4=M_4(\bmu,\varepsilon)$, such that, for all $m\ge M_4$,
\begin{equation}\label{eq:above1}
\min_{0\le k\leq (1-\varepsilon)m}q_{k,m}\geq 1-\frac{4}{c\varepsilon m}\ge 1-\varepsilon.
\end{equation}
Using similar arguments, one can obtain that for all $m\ge M_4$,
\begin{equation}\label{eq:above2}
\min_{0\le k\leq (1-\varepsilon)m}q_{k,\lfloor(1+\varepsilon)m\rfloor}\geq 1-\varepsilon.
\end{equation}
Recall that $\omega_k(\varepsilon)=\sup_{1-\varepsilon\leq s\leq t\leq 1}\vert \varphi_k(s)-\varphi_k(t)\vert$. 
Using~\eqref{eq:above1},~\eqref{eq:above2} and the triangle inequality, 
we obtain that, for all $m\ge M_4$, 
and for all $k\le (1-\varepsilon)m$,
\begin{equation}\label{train2}
 \vert\varphi_k(q_{k,\lfloor(1+\varepsilon)m\rfloor}) -\varphi_k(q_{k,m})\vert\leq \omega_k(\varepsilon).
\end{equation}
Hence, for all $m\ge M_4$,
\[
A_{\varepsilon,m}^{(1)}\leq \sum_{k=0}^{m-1}\omega_k(\varepsilon).
\]
By Condition \ref{(V)}, 
there exists $\varepsilon=\varepsilon(\bmu,\eta)>0$ small enough and $M_5=M_5(\mu,\eta,\varepsilon)$, such that, for all $m\ge M_5$,
\begin{equation}\label{conv}
A_{\varepsilon,m}^{(1)}\leq \frac{c^2\eta}{48} m.
\end{equation}
Hence, putting \eqref{eqn:lemmaheightsurvival1}, \eqref{eqn:lemmaheightsurvival2} and \eqref{conv} together, we deduce the 
existence of $M_6=M_6(\bmu,\eta,\varepsilon)$, such that, for all $m\ge M_6$,
\begin{equation} \begin{split}
\Pmu(h(\cT)\geq  m)-\Pmu(h(\cT)\geq (1+\varepsilon)m)\leq \frac{\eta}{12\sigma^2 m}+16\frac{2\varepsilon\sigma^2}{c^2m}+32\frac{\varepsilon \sigma^2}{c^2 m}\le \frac{\eta}{4\sigma^2 m}.
\end{split}
\end{equation}
Using the above together with \eqref{haaaaaaa}, we obtain that there exists $M_7=M_7(\bmu,\eta,\varepsilon)$ such that, for all $m\ge M_7$, 
\begin{equation}\label{bluey}\begin{split}
\Pmu\left(h(\cT)
\geq (1+\varepsilon)m\right)
&\geq \Pmu\left(h(\cT)\geq  m\right)- \left(\Pmu\left(h(\cT)\geq  m\right)-\Pmu\left(h(\cT)\geq (1+\varepsilon)m\right)\right)\\ 
&\geq \Pmu(h(\cT)\geq m) - \frac{\eta}{4\sigma^2 m}\\
&\ge \Pmu(h(\cT)\geq m)\left(1-\frac{\eta}{2}\right)\ge \Pmu(h(\cT)\geq m)(1-{\eta}).
\end{split}\end{equation}
since $\eta>0$ was arbitrary, this concludes  the proof of~\eqref{eqn:heightsurvivalproba}, taking $\varepsilon_0=\varepsilon$ and $M=M_7$.

We finally establish \eqref{eqn:ZmTmnottoosmall}. Let $\eta >0$ and write $p_{\varepsilon_0,m}:=\Pmu\left(h(\mathcal{T})<(1+\varepsilon_0)m\right)$. By \eqref{eqn:heightsurvivalproba}, there exists $\varepsilon_0=\varepsilon_0(\bmu,\eta)>0$ and $M=M(\bmu,\eta, \varepsilon_0)\ge 0$ such that for all $m\ge M$ and $\alpha>0$, we have that
\begin{equation}\begin{split}
\Pmu(Z_m(\cT)\leq \alpha m\,\vert\,h(\cT)\geq m)
&\leq  \frac{p_{\varepsilon_0,m}+\Pmu\left(Z_{\lfloor(1+\varepsilon_0)m\rfloor}(\cT)>0\text{ and } 
0<Z_m(\cT)\leq \alpha m\right)}{\Pmu\left( h(\cT)\ge m\right)}\\
&\le\frac{\eta}{2}+\frac{\Pmu\left(Z_{\lfloor(1+\varepsilon_0)m\rfloor}(\cT)>0\text{ and } 
0<Z_m(\cT)\leq \alpha m\right)}{\Pmu\left( Z_m(\cT)>0\right)}\\
&\le\frac{\eta}{2}+\Pmu\big(Z_{\lfloor(1+\varepsilon_0)m\rfloor}(\cT)>0\,\vert\, 0<Z_m(\cT)\leq \alpha m\big)\\
&\le \frac{\eta}{2}+\alpha m \mathrm{P}_{(\mu_k)_{k\ge m}}\left(h(\mathcal{T})\ge \varepsilon_0 m\right),
\end{split}\end{equation}
wherer we used a union bound on the vertices of $Z_m(\cT)$ for the last inequality.
Using formulas similar to \eqref{eqn:heightsurvivalexactformula} and \eqref{eqn:heightsurvivalbounds} together with Lemma \ref{lem:deterministicLLNcuttopieces}, we obtain that there exists $M_8=M_8(\bmu,\eta,\varepsilon_0)$ such that, for all $m\ge M_8$ and for all $\alpha< \eta c \varepsilon_0/8$,
\begin{equation}\begin{split}
\Pmu\big(Z_m(\cT)\leq \alpha m\,\vert\,h(\cT)\geq m\big)
&
\leq \frac{\eta}{2}+\alpha m \frac{1}{1+\tfrac12 \sum_{k=m}^{\lfloor \varepsilon_0 m\rfloor} \varphi_k(0)}\\
&
\leq \frac{\eta}{2}+\alpha m \frac{4}{c \varepsilon_0 m }\\
&
\leq \frac{\eta}{2}+\frac{4\alpha}{c\varepsilon_0}\\
&\le \eta.
\end{split}\end{equation}
This concludes the proof.
\end{proof}

\subsection{Coarse shape of the forest $\cF_n$}\label{subsec:firstbounds}
Recall the definition \eqref{def:expl_forest} of $\cF_n$ the forest explored at step $n$. Let $t(\cF_n)$ be the number of trees in $\cF_n$ (including the last tree that is possibly not fully explored). Recall that $w(\cF_n)$ and $h(\cF_n)$ denote the width and the height of $\cF_n$, respectively, defined in \eqref{def:heightwidth}.

The proposition below states that at step $n$, the explored forest $\cF_n$ is contained in a square of height and width both of order $\sqrt{n}$: this means that the highest tree in $\cF_n$ is of order $\sqrt{n}$ and, that the total number of points of $\cF_n$ in a given generation is at most of order $\sqrt{n}$, see Figure \ref{fig:forest} for an illustration of this.

\begin{proposition}
\label{prop:firstbounds}
Let $\bmu$ be a strictly critical environment satisfying Conditions \ref{(I)} and \ref{(II)}. 
For all $\varepsilon >0$, there exist constants $a_1=a_1(\varepsilon), A_1=A_1(\varepsilon), a_2=a_2(\varepsilon), A_2=A_2(\varepsilon)>0$ depending only on $\sigma^2$, $\varepsilon$ and the constant $c$ from Condition \ref{(II)}, and 
there exists $N=N(\bmu,\varepsilon)$ such that, for all $n\ge N$,
\[
\Pmu\big(a_1\sqrt{n}\leq t(\cF_n)\leq w(\cF_n)\leq A_1\sqrt{n}\text{ and } a_2\sqrt{n}\leq h(\cF_n)\leq A_2\sqrt{n}\big)\geq 1-\varepsilon.
\]
\end{proposition}

The following lemma will be used in the proof of Proposition~\ref{prop:firstbounds}, 
and several times in the rest of the paper.
\begin{lemma}\label{lem:offspringvarianceandDoob}
Let $\bmu$ be a strictly critical environment, let $k_0\in\mathbb{N}$ and define $\bnu:=(\mu_{k_0+k})_{k\geq 0}$. Let $z_0\in \dN$ and let $(Z_k)_{k\ge0}$ be a $\bnu$-BPVE with $Z_0=z_0$. Then, for all $k_1\geq 1$, we have that
\begin{equation}\label{eqn:offspringvariance}
\Var_{\bnu}(Z_{k_1})=z_0(\sigma_{k_0}^2+\ldots +\sigma_{k_0+k_1-1}^2),
\end{equation}
and, for all $K>0$,
\begin{equation}\label{eqn:offspringDoobL2}
\mathrm{P}_{\bnu} \big(\max_{k\leq k_1}\vert Z_k-Z_0\vert >K\big)
\leq \frac{4z_0(\sigma_{k_0}^2+\ldots +\sigma_{k_0+k_1-1}^2)}{K^2}.
\end{equation}
\end{lemma}

\begin{proof}
We show (\ref{eqn:offspringvariance}) by induction on $k_1$. The initialisation at $k_1=1$ is straightforward: indeed $Z_1$ is simply a sum of $z_0$ independent random variables with distribution $\mu_{k_0}$ hence $\Var_{\bnu}(Z_1)=z_0\sigma^2_{k_0}$. Suppose now that (\ref{eqn:offspringvariance}) holds for some $k_1\geq 1$. Then 
by the law of total variance, 
\begin{equation}\label{eqqqqq1}
\Var_{\bnu}(Z_{k_1+1})= \Var_{\bnu}(\mathrm E_{\bnu}[Z_{k_1+1}\vert Z_{k_1}])+ \mathrm E_{\bnu}[\Var_{\bnu}(Z_{k_1+1}\vert Z_{k_1})].
\end{equation}
We have $\mathrm E_{\bnu}[Z_{k_1+1}\vert Z_{k_1}]=Z_{k_1}$ $\mathrm{P}_{\bnu}$-a.s., hence, using the induction hypothesis,
\begin{equation}\label{eqqqqq2}
\Var_{\bnu}(\mathrm E_{\bnu}[Z_{k_1+1}\vert Z_{k_1}])=\Var_{\bnu}(Z_{k_1})=z_0(\sigma_{k_0}^2+\ldots +\sigma_{k_0+k_1-1}^2).
\end{equation}
Moreover, we have that $\Var_{\bnu}(Z_{k_1+1}\vert Z_{k_1})=Z_{k_1}\sigma_{k_1}^2$ $\mathrm{P}_{\bnu}$-a.s., therefore, using that $\bmu$ is strictly critical, we have that
\begin{equation}\label{eqqqqq3}
\mathrm E_{\bmu}[\Var_{\bmu}(Z_{k_1+1}\vert Z_{k_1})]
=\mathrm E_{\bmu}[Z_{k_1}\sigma_{k_0+k_1}^2]=z_0\sigma_{k_0+k_1}^2.
\end{equation}
Combining \eqref{eqqqqq1}, \eqref{eqqqqq2} and \eqref{eqqqqq3}, we obtain that
\begin{equation}
\Var_{\bnu}(Z_{k_1+1})= z_0(\sigma_{k_0}^2+\ldots +\sigma_{k_0+k_1}^2).
\end{equation}
Hence, by induction, \eqref{eqn:offspringvariance} holds for all $k_1\ge1$.
\\
Finally, as $\bnu$ is strictly critical, $(Z_k)_{k\geq 0}$  is a martingale and  \eqref{eqn:offspringDoobL2} is a straightforward consequence of \eqref{eqn:offspringvariance} and Doob's $L^2$-inequality.
\end{proof}

We now turn to the proof of Proposition~\ref{prop:firstbounds}.
\begin{proof}[Proof of Proposition~\ref{prop:firstbounds}]
Let us fix $\delta =\varepsilon/12$ throughout the proof. We will proceed in five steps.

\textbf{Step 1.} 
Let us prove that there exists $N_0\in \dN$ such that, for all $n\ge N_0$,
\begin{equation}\label{step1}
\Pmu(t(\cF_n)\geq A_0\sqrt{n})\leq \delta,
\end{equation}
with $A_0=\max\{32\sigma^2,\tfrac{8}{\delta}\}$.\\ 
For all $k\geq 0$, let $Z_k$ denote the number of vertices at height~$k$ in the first~$\lfloor A_0\sqrt n\rfloor$ trees of~$\mathcal F$, so that $(Z_k)$ is a $\bmu$-BPVE with $Z_0=\lfloor A_0\sqrt n\rfloor$. 
By applying Lemma \ref{lem:offspringvarianceandDoob} with $k_0=0$, $k_1={\lfloor\delta\sqrt{n}\rfloor}$ and $K=A_0\sqrt{n}/2$,
we have 
\begin{align*}
\Pmu\hspace{-1mm}\left(\min_{k\leq \delta\sqrt{n}}Z_k{\le} \frac{A_0\sqrt{n}}{2}-1\right)\hspace{-1mm}
&\leq \Pmu\hspace{-1mm}\left(\max_{k\leq \delta\sqrt{n}}{\vert Z_k-Z_0\vert}> \frac{A_0\sqrt{n}}{2}\right)\hspace{-1mm}
\leq \frac{{16}\left(\sigma_0^2+\ldots+\sigma^2_{\lfloor{\delta\sqrt{n}}\rfloor}\right)}{A_0\sqrt{n}}.
\end{align*}
By Condition \ref{(I)}, there exists $N_0'=N_0'(\bmu,\delta)\in \dN$ such that, for all $n\geq N_0'$, we have that $(\sigma_0^2+\ldots+\sigma^2_{\lfloor\delta\sqrt{n}\rfloor})/\sqrt{n} \leq 2\delta\sigma^2$  and thus
\begin{equation}\label{jjjjj1}
\Pmu\left(\min_{k\leq \delta\sqrt{n}}Z_k{\le} \frac{A_0\sqrt{n}}{2}-1\right)\leq  \frac{32\delta\sigma^2\sqrt{n}}{A_0\sqrt{n}}\leq \delta.
\end{equation}
If $\min_{k\leq \delta\sqrt{n}}Z_k{>} (A_0\sqrt{n}/2)-1$, 
then there are at least $(A_0\sqrt{n}/2)-1$ vertices on the first ${\lfloor\delta\sqrt{n}\rfloor}$ 
generations of the first $\lfloor A_0\sqrt{n}\rfloor$ trees of $\mathcal F$. 
Hence, provided that $n\ge \delta^2\vee 2\delta^{-2}$, these trees together have a total number of vertices exceeding
\[
\left(\frac{A_0\sqrt{n}}{2}-1\right)\times\lfloor\delta \sqrt{n}\rfloor \geq \left(\frac{4\sqrt{n}}{\delta}-1\right)\times (\delta\sqrt{n}-1)\geq \frac{2\sqrt{n}}{\delta}\times \frac{\delta\sqrt{n}}{2}\geq n,
\]
thus  $\cF_n$ is included in these first $\lfloor A_0\sqrt{n}\rfloor$ trees. This proves that
\begin{equation}\label{jjjjjj2}
\left\{t(\cF_n)\geq A_0\sqrt{n}\right\}\subset \left\{\min_{k\leq \delta\sqrt{n}}Z_k{\le} \frac{A_0\sqrt{n}}{2}-1\right\}.
\end{equation}
Combining \eqref{jjjjj1} and \eqref{jjjjjj2} proves \eqref{step1} with $N_0(\bmu,\delta)=N_0'(\bmu,\delta)\vee\delta^2\vee 2\delta^{-2}$.

\medskip
\textbf{Step 2.}  Let us prove that there exists $N_2=N_2(\bmu,\delta)\in \dN$ such that, for all $n\ge N_2$,
\begin{equation}\label{step2}
\Pmu(h(\cF_n)\geq A_2\sqrt{n})\leq 2\delta,
\end{equation}
where $A_2=4A_0/(c\delta)=\max\{\tfrac{128\sigma^2}{c\delta},\tfrac{32}{c\delta^2}\}$.

{Let $(Z_k)_{k\ge0}$ be a $\bmu$-BPVE with $Z_0=1$.} By \eqref{eqn:heightsurvivalexactformula}, \eqref{augustoislate} and Condition \ref{(II)}, we have
\begin{equation}\label{eqn:heightdomination}
\limsup_{k\rightarrow +\infty}k\,\Pmu(Z_{k}>0)\leq 2/c,
\end{equation}
where $c>0$ is given by Condition \ref{(II)}.
{Using \eqref{eqn:heightdomination} above and \eqref{step1}, there exists $N_2=N_2(\bmu,\delta)\in\mathbb{N}$ such that, for all $n\ge N_2$,}
\begin{align*}
\Pmu(h(\cF_n)\geq A_2\sqrt{n})&\leq \Pmu(t(\cF_n){\ge}A_0\sqrt{n})+\Pmu(\{t(\cF_n){<} A_0\sqrt{n}\}\cap \{h(\cF_n)\geq A_2\sqrt{n} \})
\\
&\leq \delta +A_0\sqrt{n}\,\Pmu(Z_{\lceil A_2\sqrt{n}\rceil}>0\, \vert {Z_0=1})
\\
&\leq \delta +\frac{4A_0}{c}{\times \frac{\sqrt{n}}{\lceil A_2\sqrt{n}\rceil}}\\
&{\leq 2\delta,}
\end{align*} 
where we applied a union bound for the second inequality, and used that $A_2=4A_0/(c\delta)$ for the last inequality. This proves \eqref{step2}.

\medskip
\textbf{Step 3.} {Let us prove that there exists $N_1\in \dN$ such that, for all $n\ge N_1$,
\begin{equation}\label{step3}
\Pmu(w(\cF_n)\geq A_1\sqrt{n})\leq 4\delta,
\end{equation}}
with $A_1=A_0+\sqrt{8A_0A_2\sigma^2/\delta}$, where $A_0$ and $A_2$ are defined in \eqref{step1} and \eqref{step2}.

As in Step 1, for all $k\geq 0$, let $Z_k$ denote the number of vertices at height~$k$ in the first~$\lfloor A_0\sqrt n\rfloor$ trees of~$\mathcal F$,  so that $(Z_k)_k$ is a $\bmu$-BPVE with $Z_0=\lfloor A_0\sqrt n\rfloor$.  By \eqref{step1} and \eqref{step2}, we have that, for all $n\ge N_0\vee N_2$,
\begin{equation}\label{step31}\begin{split}
\Pmu(w(\cF_n)\geq A_1\sqrt{n}) \leq & \Pmu(t(\cF_n)> A_0\sqrt{n})\\
&+\Pmu(h(\cF_n)\geq A_2\sqrt{n})
+\Pmu\left(\max_{k\leq A_2\sqrt{n}}Z_k\geq A_1\sqrt{n}\right)
\\
\leq & 3\delta +\Pmu\left(\max_{k\leq A_2\sqrt{n}}\vert Z_k-Z_0\vert\geq (A_1-A_0)\sqrt{n}\right).
\end{split}\end{equation}
Noting that $A_1>A_0$ and applying Lemma \ref{lem:offspringvarianceandDoob}, we obtain
\[
\Pmu\left(\max_{k\leq A_2\sqrt{n}}\vert Z_k-Z_0\vert\geq (A_1-A_0)\sqrt{n}\right)
\leq \frac{4A_0\sqrt{n}}{(A_1-A_0)^2n}(\sigma_0^2+\ldots+\sigma_{\lfloor A_2\sqrt{n}\rfloor}^2).
\]
By Condition~\ref{(I)}, we know that 
$\lim_{n\rightarrow+\infty}({\sigma_0^2}+\ldots+\sigma_{\lfloor A_2\sqrt{n}\rfloor}^2)/\sqrt n= A_2\sigma^2$.
{Hence, there exists $N_1'=N_1'(\bmu,\delta)\in \dN$ such that for all $n\ge N_1'$,}
\begin{equation}\label{step32}
\Pmu\left(\max_{k\leq A_2\sqrt{n}}\vert Z_k-Z_1\vert\geq (A_1-A_0)\sqrt{n}\right)\leq \frac{{8}A_0A_2\sigma^2}{(A_1-A_0)^2}=\delta.
\end{equation}
Combining \eqref{step31} and \eqref{step32} concludes the proof of \eqref{step3} with $N_1(\bmu,\delta)=\max\{N_0,N_2,N_1'\}$.

\medskip
\textbf{Step 4.} {Let us prove that, for all $n\ge N_3$,
\begin{equation}\label{step4}
\Pmu(h(\cF_n)\leq a_2\sqrt{n})\leq 4\delta,
\end{equation}
with $a_2=1/(2A_1)$ and $N_3(\bmu,\delta)=\max\{N_1,4A_1^2\}$.
Note that, by definition of $\cF_n$, $w(\cF_n)$ and $h(\cF_n)$, we have that
\begin{align}\label{eqn:tailleforet}
n\le |\cF_n|\le (1+h(\cF_n))\times w(\cF_n),
\end{align}
which implies that 
\[
w(\cF_n)\geq \frac{n}{1+h(\cF_n)}.
\]
We conclude the proof of \eqref{step4} by noting that, for all $n\ge N_3$,
\begin{align*}
\Pmu(h(\cF_n)\leq a_2\sqrt{n}) &\le \Pmu\left(w(\cF_n)\ge \frac{n}{1+a_2\sqrt{n}}\right)\\
&\le \Pmu\left(w(\cF_n)\ge \frac{A_1\sqrt{n}}{\tfrac{A_1}{\sqrt{n}}+\tfrac{1}{2}}\right)\\
&\le \Pmu\left(w(\cF_n)\ge {A_1\sqrt{n}}\right)\\
&\le 4\delta,
\end{align*}
where we used that $a_2=1/(2A_1)$ for the second inequality, that $N_3\ge 4A_1^2$ for the third inequality and \eqref{step3} for the last inequality.
}

\medskip
\textbf{Step 5.} Let us prove that there exists  $N_4=N_4(\bmu,\delta)\in \dN$ such that, for all $n\ge N_4$,
\begin{equation}\label{step5}
\Pmu(t(\cF_n)\leq a_1\sqrt{n})\leq 2\delta,
\end{equation}
where
\[
a_1=\left(\frac{\delta c}{4\left(1+\sqrt{32\sigma^2/\delta^2 c}\right)}\right)^{1/2}.
\]
For all $k\ge 0$, let $Z_k$ be the number of vertices at height $k$ in the first $\lfloor a_1\sqrt{n}\rfloor$ trees of $\cF$, so that $(Z_k)_{k\ge0}$ is a $\bmu$-BPVE with $Z_0=\lfloor a_1\sqrt{n}\rfloor$. Let $\mathcal{F}^{a_1}$ be the forest consisting of the first $\lfloor a_1 \sqrt{n}\rfloor$ trees of $\mathcal{F}$. Using that $\vert \mathcal{F}^{a_1}\vert\le (1+h(\mathcal{F}^{a_1}))\times w(\mathcal{F}^{a_1})$, we obtain that
\begin{equation}\label{step41}\begin{split}
\Pmu\left(t(\cF_n)\leq a_1\sqrt{n}\right)&\leq \Pmu\left(\vert \mathcal{F}^{a_1}\vert\ge n\right)\\
&\leq \Pmu\left(h(\mathcal{F}^{a_1})\geq \frac{4a_1\sqrt{n}}{\delta c}-1\right)
+\Pmu\left(\max_{k\leq \frac{4a_1\sqrt{n}}{\delta c}-1}Z_k \geq\frac{\delta c}{4a_1}\sqrt{n}\right).
\end{split}\end{equation}
By \eqref{eqn:heightdomination}, there exists $N_4'=N_4'(\bmu,\delta)\in \dN$ such that for all $n\ge \max\{N_4',c^2\delta^2/a_1^2\}$, 
\begin{align}\label{step42}
\Pmu\left(h(\mathcal{F}^{a_1})\geq \frac{4a_1\sqrt{n}}{\delta c}-1\right)&\le \lfloor a_1 \sqrt{n}\rfloor \frac{3}{c\left(\frac{4a_1\sqrt{n}}{\delta c}-1\right)}\le \delta,
\end{align}
where we used a union bound on the $\lfloor a_1\sqrt{n}\rfloor$ trees of $\mathcal{F}^{a_1}$ for the first inequality.
Moreover, we have that
\begin{align*}
&\Pmu\left(\max\Big\{Z_k\colon {k\leq\frac{4a_1\sqrt{n}}{\delta c}-1}\Big\} >\frac{\delta c}{4a_1}\sqrt{n}\right)\\
&\leq \Pmu\left(\max\Big\{\vert Z_k-Z_0\vert\colon {k\leq \frac{4a_1\sqrt{n}}{\delta c}-1}\Big\} >\Big(\frac{\delta c}{4a_1}-a_1\Big)\sqrt{n}\right)
\\
&\leq \frac{4a_1}{(\frac{\delta c}{4a_1}-a_1)^2\sqrt{n}}\big(\sigma_0^2+\ldots+\sigma_{\lfloor \frac{4a_1\sqrt{n}}{\delta c}\rfloor}^2\big),
\end{align*}
where we used Lemma \ref{lem:offspringvarianceandDoob} for the last inequality, recalling that $Z_0=\lfloor a_1\sqrt{n}\rfloor$.
By Condition \ref{(I)}, there exists $N_4''=N_4''(\bmu,\delta)\in \dN$ such that, for all $n\ge N_4''$, 
\begin{align}\label{step43}
\Pmu\left(\max_{k\leq \frac{4a_1\sqrt{n}}{\delta c}-1}Z_k \geq\frac{\delta c}{4a_1}\sqrt{n}\right)&\le \frac{32a_1^2\sigma^2}{\delta c(\frac{\delta c}{4a_1}-a_1)^2}\le \delta,
\end{align}
where the last inequality is obtained from the definition of $a_1$ by a straightforward computation.
Combining \eqref{step41}, \eqref{step42} and \eqref{step43} concludes the proof of \eqref{step5} with $N_4(\bmu,\delta)=\max\{N_4',N_4'',c^2\delta^2/a_1^2\}$.\\
Finally, this concludes the proof of the lemma by combining Steps 2 to 4, noting that the inequality $t(\mathcal{F}_n)\le w(\mathcal{F}_n)$ is trivial, recalling that $\varepsilon=12 \delta$ and defining $N(\bmu,\varepsilon)=\max_{1\le i\le 4}N_i(\bmu,\delta)$.
\end{proof}

\section{Proofs of Theorem~\ref{thm:Luka} and Corollary \ref{cor:Luka}}\label{sec:Luka}

\subsection{Proof of Theorem~\ref{thm:Luka}}
In this section, we prove the convergence of the {\L}ukasiewicz path, i.e.~Theorem~\ref{thm:Luka}, which is a direct consequence of Proposition \ref{mart} and Lemmas \ref{mart1}, \ref{mart2} and \ref{mart3} below. We will consider a strictly critical environment $\bmu$ satisfying Conditions \ref{(I)}, \ref{(II)} and \ref{(III)}.
To prove the theorem, we will apply a criterion for the convergence of martingales given by~\cite[Theorem 2.1 $(ii)$]{Whitt}, which is a consequence of~\cite[Theorem 7.1]{EthierKurtz}. Recall that we want to prove that the process
\[
(M_n(t))_{0\le t\le  1}:=\left(\frac{X_{\lfloor nt\rfloor}}{\sqrt{n}}\right)_{0\leq t\leq 1}
\]
converges to a Brownian motion, where $(X_k)_{k\ge0}$ is the {\L}ukasiewicz path defined in \eqref{def:Luka}. It is straighforward to check that, as soon as $\bmu$ is strictly critical, $(X_k)_{k\ge0}$ is a martingale with respect to its natural filtration.\\
Due to technicalities to control the largest degree (see \eqref{eqn:WhittJump} below), 
we will in fact apply~\cite[Theorem 2.1$(ii)$]{Whitt}  to stopped martingales obtained from $(M_n)_{n\geq 1}$ by stopping the exploration processes when they exit a large rectangle. To make this more precise, let us prove the following easy consequence of Condition \ref{(III)}.

\begin{corollary}\label{bigrectangle}
Let $\bmu$ be an environment satisfying Condition \ref{(III)}. Then, there exist two sequences $(h_n)_{n\ge1}$ and $(w_n)_{n\ge1}$, both diverging to $+\infty$, such that
\begin{align*}
&h_n\le n^{1/3}\text{, for all }n\ge1,\\
&\frac{w_n}{\sqrt{n}}\sum_{k=0}^{\lfloor h_n\sqrt{n}\rfloor}\mathrm E_{\bmu}[\xi_{k}^2\mathbf{1}_{\xi_{k}^2\geq \varepsilon n}]\underset{n\rightarrow +\infty}{\longrightarrow} 0,
\end{align*}
where $\xi_k\sim \mu_k$, for all $k\ge0$.
\end{corollary}
\begin{proof}
First, note that if Condition \ref{(III)} holds for a sequence $(\widetilde{h}_n)$, then it trivially holds for $(h_n)$ where, for all $n\ge1$, $h_n=n^{1/3}\wedge \widetilde{h}_n$. Consider such a sequence $(h_n)_{n\ge 1}$ so that Condition \ref{(III)} holds and $h_n\le n^{1/3}$ for all $n\ge1$.\\
Define, for all $n\ge1$,
\[
w_n=\left(\frac{1}{\sqrt{n}}+\frac{\sum_{k=0}^{\lfloor h_n\sqrt{n}\rfloor}\mathrm E_{\bmu}[\xi_{k}^2\mathbf{1}_{\xi^2_{k}\geq \varepsilon n}]}{\sqrt{n}}\right)^{-1/2}.
\]
By Condition \ref{(III)}, we have that $w_n\to+\infty$ and
\[
\left\vert\frac{w_n}{\sqrt{n}}\sum_{k=0}^{\lfloor h_n\sqrt{n}\rfloor}\mathrm E_{\bmu}[\xi_{k}^2\mathbf{1}_{\xi^2_{k}\geq \varepsilon n}]\right\vert\le \left\vert\frac{1}{\sqrt{n}}\sum_{k=0}^{\lfloor h_n\sqrt{n}\rfloor}\mathrm E_{\bmu}[\xi_{k}^2\mathbf{1}_{\xi_{k}^2 \geq \varepsilon n}]\right\vert^{1/2}\underset{n\rightarrow +\infty}{\longrightarrow} 0,
\]
which concludes the proof.
\end{proof}

For all $n\ge1$,
define the random time
\begin{equation}\label{def:taun}
\tau_n:=n\wedge \min\{k \colon h(\cF_k)\geq h_n\sqrt{n}\text{ or }w(\cF_k)\geq w_n\sqrt{n}\}.
\end{equation}
Recall that $H_n$ and $L_n$, ${n\ge0}$, denote respectively the height and the number of children of the $n$-th node along the depth-first exploration of the forest $\mathcal{F}$, so that $L_n=1+X_{n+1}-X_n$.  Recall that $\sigma_i^2$ is the variance of $\mu_i$, for $i\ge0$.\\
We start by stating and proving a proposition providing three sufficient conditions for the convergence of the {\L}ukasiewicz path. These three conditions are then proved to hold in Lemmas \ref{mart1}, \ref{mart2} and \ref{mart3}, which altogether proves Theorem \ref{thm:Luka}.

\begin{proposition}\label{mart}
Let $\bmu$ be a strictly critical environment satisfying Conditions \ref{(I)}, \ref{(II)} and \ref{(III)}. Consider $(h_n)_{n\ge1}$ and $(w_n)_{n\ge1}$ provided by Corollary \ref{bigrectangle}, and $(\tau_n)_{n\ge1}$ defined in \eqref{def:taun}. The conclusion of Theorem \ref{thm:Luka} holds if the following three conditions are satisfied:
\begin{align}\label{eqn:WhittMaxVariance}
&\lim_{n\rightarrow +\infty}\max_{0\leq i\leq h_n\sqrt{n}}\frac{\sigma_i^2}{n}=0,\\ \label{eqn:WhittJump}
&\lim_{n\rightarrow +\infty}\bE_{\bmu}\left[\max_{0\leq i\leq \tau_n-1}\frac{L_i^2}{n}\right]=0,\\ \label{eqn:WhittAverageVariance}
&\lim_{n\rightarrow +\infty}{\sum_{k=0}^{n-1}\frac{\sigma_{H_k}^2}{n}}=\sigma^2\text{, in $\Pmu$-probability}. 
\end{align}
\end{proposition}
The first condition is a domination of the maximal jump of the quadratic variation of the {\L}ukasiewicz path. The second condition provides a control of the maximal square degree encoutered by the exploration process. 
The third condition gives the convergence of the quadratic variation of the {\L}ukasiewicz path.
\begin{proof}
For all $n\ge1$ and for all $k\ge0$, define
\[
\widetilde{X}^{(n)}_k:=X_{k\wedge \tau_n}.
\]

It is straightforward to check that $\tau_n$ is a stopping time with respect to the natural filtration 
of $(X_k)_{k\geq 0}$.
Hence, $\widetilde{X}^{(n)}$ is a stopped martingale and thus a martingale. 
For all $n\geq 1$, define the process
\[
\left(\widetilde{M}_n\right)_{0\le t\le 1}:=\left(n^{-1/2}\widetilde{X}^{(n)}_{\lfloor nt\rfloor}\right)_{0\leq t\leq 1}.
\]
By Proposition~\ref{prop:firstbounds} and using that both $(h_n)$ and $(w_n)$ diverge to infinity, we have that, for all $\varepsilon>0$,
there exist  $N=N(\mu,\varepsilon)\in\mathbb{N}$ such that
 for all $n\ge N$, $w_n\ge  A_1(\varepsilon)$ and $h_n\ge A_2(\varepsilon)$ and thus
\[
\Pmu\left(\tau_n<n\right)\le\Pmu\left(h(\cF_n)\geq h_n\sqrt{n}\text{ or }w(\cF_n)\geq w_n\sqrt{n}\right)<\varepsilon.
\]
The above implies that
\begin{equation}\label{eqn:modifiedmartingale}
\limsup_{n\to\infty}\Pmu(M_n\neq \widetilde{M}_n)
= \limsup_{n\to\infty}\Pmu(\exists k\leq n, X_k^{(n)}\neq \widetilde{X}^{(n)}_k)= \limsup_{n\to\infty}\Pmu(\tau_n<n)= 0.
\end{equation}
\noindent
Therefore, it is enough to prove the convergence in law of $\widetilde{M}_n$ to $(\sigma B_t)_{0\leq t\leq 1}$, as $n$ goes to infinity.\\
Recalling that
\[
\widetilde{X}^{(n)}_k=\sum_{i=0}^{k\wedge \tau_n-1}(L_i-1),
\]
it is straightforward to check that the angle-bracket process associated with $\widetilde{M}_n$ is given, for all $t\in[0,1]$, by
\[
\langle M_n(t)\rangle = \sum_{k=0}^{\lfloor nt \rfloor\wedge \tau_n-1} \frac{\sigma^2_{H_i}}{n}
\]

By \cite[Theorem 2.1 $(ii)$]{Whitt}, the  three following conditions are sufficient for this convergence to hold:
\begin{align}\label{whitt1}
&\lim_{n\to \infty} \bE_{\bmu}\left[\max_{0\le i\le \tau_n-1} \frac{\sigma^2_{H_i}}{n}\right]=0,\\ \label{whitt2}
&\lim_{n\to \infty} \bE_{\bmu}\left[\max_{0\le i\le \tau_n-1} \frac{(L_i-1)^2}{n}\right]=0,\\ \label{whitt3}
&\langle M_n(t)\rangle \Rightarrow \sigma^2 t \text{ in distribution, for all }t\in[0,1].
\end{align}
For all $i\le \tau_n-1$, using definition of $\tau_n$, we have that $0\le H_i<h_n\sqrt{n}$, hence
\[
\bE_{\bmu}\left[\max_{0\le i\le \tau_n-1} \frac{\sigma^2_{H_i}}{n}\right]\le \max_{i\le h_n\sqrt{n}} \frac{\sigma^2_{i}}{n},
\]
and \eqref{eqn:WhittMaxVariance} implies \eqref{whitt1}. For all $i\ge0$, $(L_i-1)^2\le L_i^2 +1$, thus we have that  \eqref{eqn:WhittJump} implies \eqref{whitt2}. Finally, \eqref{eqn:WhittAverageVariance} implies that, for all $t\in[0,1]$,
\[
\lim_{n\rightarrow +\infty}\frac{1}{n}{\sum_{k=0}^{\lfloor nt\rfloor-1}\sigma_{H_k}^2}=\sigma^2t
\]
in $\Pmu$-probability. Using that, for all $t\in[0,1]$, $\Pmu(\tau_n<\lfloor nt\rfloor)\to 0$, we have that \eqref{eqn:WhittAverageVariance} implies that the limit in \eqref{whitt3} holds in probablity and thus in distribution. This concludes the proof of the proposition.
\end{proof}

\begin{lemma}\label{mart1}
Let $\bmu$ be a strictly critical environment satisfying Conditions \ref{(I)}, \ref{(II)} and \ref{(III)}. Consider $(h_n)_{n\ge1}$  provided by Corollary \ref{bigrectangle}. We have that
\begin{align*}
&\lim_{n\rightarrow +\infty}\max_{0\leq i\leq h_n\sqrt{n}}\frac{\sigma_i^2}{n}=0.
\end{align*}
In other words, \eqref{eqn:WhittMaxVariance} is satisfied.
\end{lemma}
\begin{proof}
By Corollary \ref{bigrectangle}, we know that Condition \ref{(III)} holds with a sequence $(h_n)_{n\ge1}$ such that $h_n\le n^{1/3}$ for all $n\ge1$.
By Condition \ref{(I)}, we have that
\[
\limsup_{n\to \infty} \max_{0\leq i\leq h_n\sqrt{n}}\frac{\sigma_i^2}{n}\le \limsup_{n\to \infty}\frac{\displaystyle \sum_{i=0}^{\lfloor h_n\sqrt{n}\rfloor}\sigma_i^2}{n}\le n^{-1/6}\limsup_{n\to \infty}\frac{\displaystyle \sum_{i=0}^{\lfloor h_n\sqrt{n}\rfloor}\sigma_i^2}{h_n\sqrt{n}}=0.
\]
\end{proof}

\begin{lemma}\label{mart2}
Let $\bmu$ be a strictly critical environment satisfying Conditions \ref{(I)}, \ref{(II)} and \ref{(III)}. Consider $(h_n)_{n\ge1}$ and $(w_n)_{n\ge1}$ provided by Corollary \ref{bigrectangle}, and $(\tau_n)_{n\ge1}$ defined in \eqref{def:taun}.We have that
\begin{align*}
&\lim_{n\rightarrow +\infty}\bE_{\bmu}\left[\max_{0\leq i\leq \tau_n-1}\frac{L_i^2}{n}\right]=0. 
\end{align*}
In other words, \eqref{eqn:WhittJump} is satisfied.
\end{lemma}

In the proof below, we use the notation $\bE_{\bmu}[A;B]=\bE_{\bmu}[A\mathbf{1}_B]$, for a random variable $A$ and an event $B$.

\begin{proof}
Fix $\varepsilon >0$. For all $k\geq 0$, 
let $(\xi_{k,i})_{i\geq 1}$ be a sequence of i.i.d.~random variables with distribution $\mu_k$. 
For all $n\ge1$, we have that
\begin{align*}
\bE_{\bmu}\left[\max_{0\leq i\leq \tau_n-1}\frac{L_i^2}{n}\right]&\le \varepsilon +\bE_{\bmu}\left[\max_{0\leq i\leq \tau_n-1}\frac{L_i^2}{n};\max_{0\leq i\leq \tau_n-1}{L_i^2}\ge \varepsilon n\right]\\
&\le \varepsilon +\bE_{\bmu}\left[ \max_{\substack{0\leq k\leq  h_n\sqrt{n}\\ 1\le i\le w_n \sqrt{n}}}
\frac{\xi_{k,i}^2}{n};
\max_{\substack{0\leq k\leq  h_n\sqrt{n}\\ 1\le i\le w_n \sqrt{n}}}
{\xi_{k,i}^2}\ge \varepsilon n\right]\\
&\le \varepsilon +\frac{1}{n}\sum_{\substack{0\leq k\leq  h_n\sqrt{n}\\ 1\le i\le w_n \sqrt{n}}}\bE_{\bmu}\left[ \xi_{k,i}^2;\xi_{k,i}^2=\max_{\substack{0\leq k'\leq  h_n\sqrt{n}\\ 1\le i'\le w_n \sqrt{n}}}
\xi_{k',i'}^2 \ge \varepsilon n \right]\\
&\le \varepsilon+\frac{1}{n}\sum_{\substack{0\leq k\leq  h_n\sqrt{n}\\ 1\le i\le w_n \sqrt{n}}} \bE_{\bmu}\left[\xi_{k,i}^2\mathbf{1}_{\{\xi_{k,i}^2\ge \varepsilon n\}}\right]
\\
&\le \varepsilon+\frac{w_n}{\sqrt{n}}\sum_{k=0}^{\lfloor h_n\sqrt{n}\rfloor} \bE_{\bmu}\left[\xi_{k,1}^2\mathbf{1}_{\{\xi_{k,1}^2\ge \varepsilon n\}}\right],
\end{align*}
where the third inequality follows from a union bound, decomposing the event 
\begin{center}
$\{\max_{\substack{0\leq k\leq  h_n\sqrt{n}\\ 1\le i\le w_n \sqrt{n}}}
{\xi_{k,i}^2}\ge \varepsilon n\}$ 
\end{center}
according to which index $(k,i)$ realizes the maximum.
The above and Corollary \ref{bigrectangle} imply that, for all $\varepsilon>0$,
\begin{align*}
\limsup_{n\to\infty} \bE_{\bmu}\left[\max_{0\leq i\leq \tau_n-1}\frac{L_i^2}{n}\right]&\le \varepsilon,
\end{align*}
which yields the conclusion by taking $\varepsilon$ to zero.
\end{proof}

In order to conclude the proof of Theorem \ref{thm:Luka}, it remains to prove that \eqref{eqn:WhittAverageVariance} holds. This is done in Lemma \ref{mart3} below, whose proof is a refinement of the reasoning in Section~\ref{subsec:firstbounds}. Recall that the goal is to prove that the sum of the variances collected along the depth-first order exploration of $\mathcal{F}$ at step $n$ is roughly equal to $\sigma^2 n$.\\
The general idea of the proof is close, in spirit, to a coarse graining, as we will divide  the forest $\cF_n$ in mesoscopic blocks of width and height $\Theta(\sqrt{n})$, 
and show that the average value of $\sigma_{H_k}^2$ on most blocks is close to $\sigma^2$. Before stating and proving Lemma \ref{mart3}, we will need a few other lemmas.\\
Recall that, in Section \ref{sectorder}, we defined the depth-first labeling of $\mathcal{F}$ together with the lexicographical order of $\mathcal{F}$.

\begin{figure}
\begin{center}
\includegraphics[width=13.5cm]{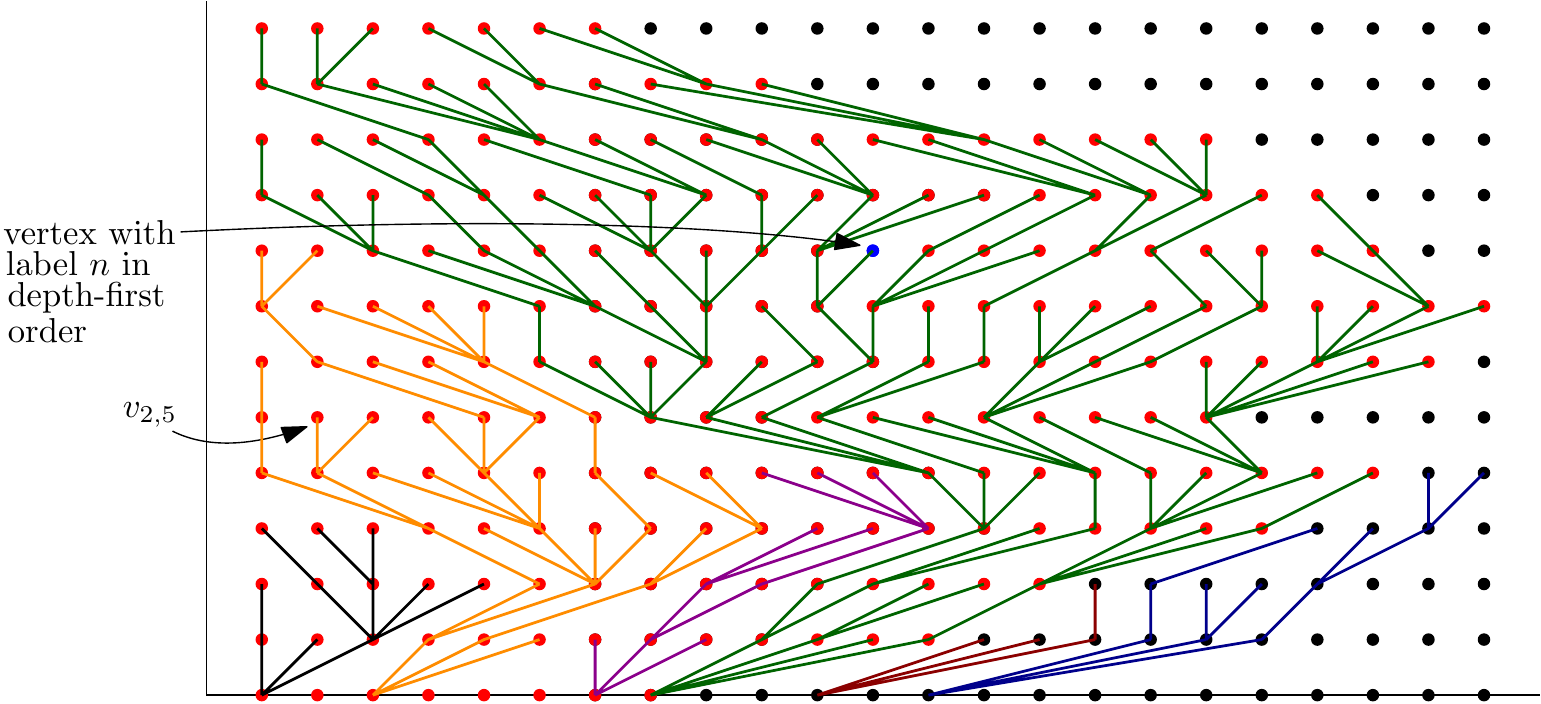}
\caption{A representation of the infinite forest $\cF_\infty$. The vertex $v_{k,\ell}$ is drawn in position $(k,\ell)$ on the plane. The different trees of the infinite forest have distinct colours. The vertices in red are the ones that belong to $\cF_n$.}
\label{fig:forest}
\end{center}
\end{figure}

Let us now define $(n,\delta,\gamma)$-blocks, or simply blocks, on the forest $\mathcal{F}$ with this lexicographical order. For this purpose, let us denote $\mathcal{T}(v)$ the set of descendants of $v$, that is the subtree rooted at $v$, for a vertex $v\in\mathcal{F}$. For $\gamma>0$ and $\delta>0$, for all $n\ge1$, for all $i\ge0$ and all $k\ge0$, we define the $(n,\delta,\gamma)$-block by
\begin{align}\label{defbij}
B^n_{i,k}=\bigcup_{j=i\lceil \delta\sqrt{n}\rceil+1}^{(i+1)\lceil\delta\sqrt{n}\rceil} \left\{v\in\mathcal{T}\left(v_{j,k\lceil \gamma\sqrt{n}\rceil}\right) : k\lceil \gamma\sqrt{n}\rceil\le h(v)< (k+1)\lceil \gamma\sqrt{n}\rceil\right\}.
\end{align}
See Figure \ref{fig:blocs} for a representation of these blocks. In words, $B^n_{i,k}$ is the collection of the $\lceil \delta\sqrt{n}\rceil$ finite subtrees of height (at most) $\lceil \gamma\sqrt{n}\rceil-1$ rooted at $v_{j,k\lceil \gamma\sqrt{n}\rceil}$, for  $i\lceil \delta\sqrt{n}\rceil<j\le (i+1)\lceil \delta\sqrt{n}\rceil$. In this notation, we ignore the dependence on $\delta$ and $\gamma$, as we will quickly fix these constants.
For a block $B_{i,k}^n$, $i,k\ge0$, we define the total variance of the block as
\begin{equation}\label{def:Wik}
W_{i,k}=\sum_{v\in B^n_{i,k}} \sigma^2_{h(v)}.
\end{equation}
In what follows, for $n\ge1$, $i,k\ge0$, and for $\varepsilon>0$, we say that the block $B^n_{i,k}$ is {\bf $\varepsilon$-good} if
\begin{equation}\label{defgood}
\left\vert \frac{W_{i,k}}{\vert B^n_{i,k}\vert\sigma^2}-1\right\vert\le 16\varepsilon.
\end{equation}
Moreover, if the above is not satisfied, then we say that $B^n_{i,k}$ is not {\bf $\varepsilon$-good}.

\begin{figure}
\begin{center}
\includegraphics[width=12cm]{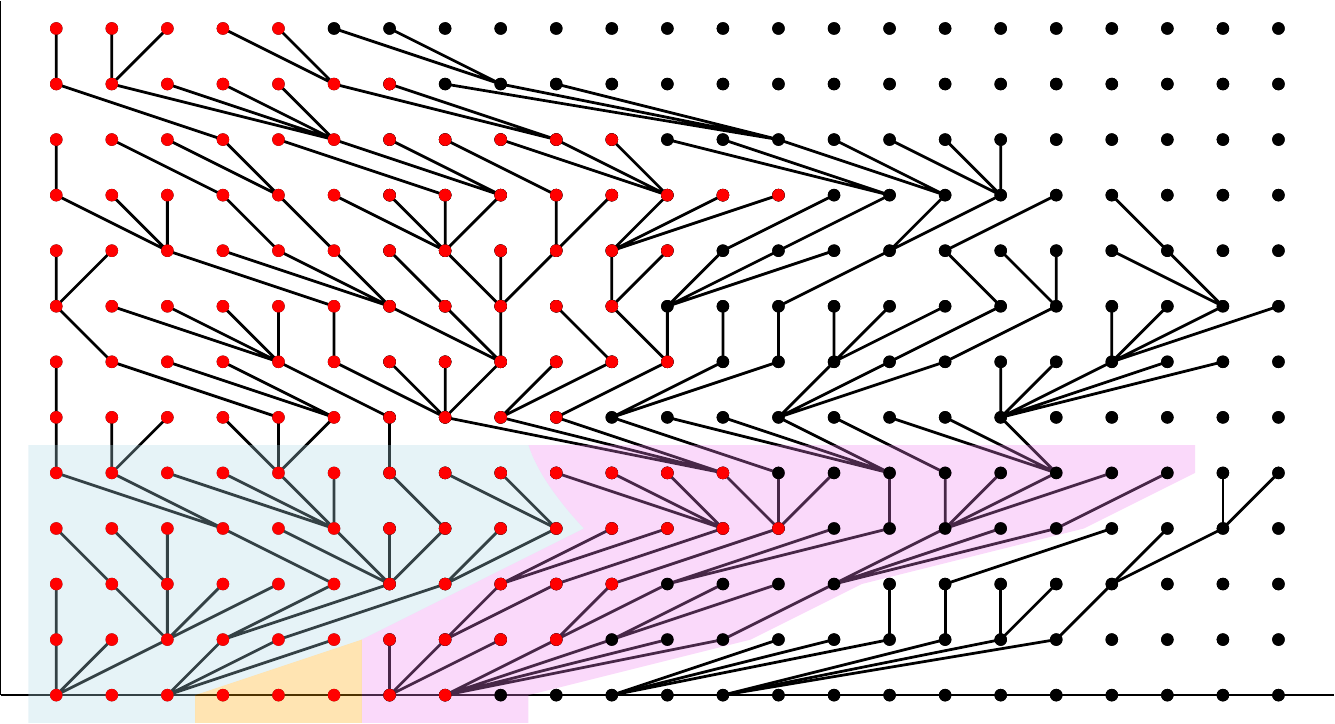}
\caption{A representation of the infinite forest $\cF_\infty$. If $\delta\sqrt{n}=3$ and $\gamma\sqrt{n}=5$, then
the block in blue is $B^n_{0,0}$, the block in orange is $B^n_{0,1}$, and the block in purple is $B^n_{0,2}$.}
\label{fig:blocs}
\end{center}
\end{figure}

\begin{remark}\label{remindep}
Note that, for all $n\ge1$, $\delta>0$ and $\gamma>0$, the pairs of random variables $(\vert B^n_{i,k}\vert, W_{i,k})$, $i\ge0$, $k\ge0$ are independent~under $\Pmu$ and their distributions depend on $k$ only, by definition of a BPVE. Therefore, the events $\left\{B^n_{i,k}\text{ is not }\varepsilon\text{-good}\right\}$, $i\ge0$, $k\ge0$ are independent.
\end{remark}

The following lemma is the main ingredient for the proof of \eqref{eqn:WhittAverageVariance} in Lemma~\ref{mart3}. It states that a given box has a large probability to be $\varepsilon$-good, and provides a control on the maximum volume and weight of a box (this will be used to prevent that a not $\varepsilon$-good box has too much importance in the sum in \eqref{eqn:WhittAverageVariance}).

\begin{lemma}\label{lemmart1}
Let $\bmu$ be a strictly critical environment satisfying Conditions \ref{(I)}, \ref{(II)} and \ref{(III)}. Fix $\varepsilon\in(0,1/2)$, $\delta>0$ and $\gamma>0$. For all $k\ge0$, there exists $N_k=N_k(\bmu,\varepsilon,\delta,\gamma)$ such that, for all $n\ge N_k$, we have that
\begin{equation}\label{eqn:balancedblock}
\Pmu\left(B^n_{0,k}\text{ is $\varepsilon$-good}\right)\geq 1-6\sigma^2\gamma\delta^{-1}\varepsilon^{-2},
\end{equation}
and, for all $K\ge6\gamma\delta(1+\sigma^2)$,
\begin{align*}
\Pmu\left(\vert B^n_{0,k}\vert> K n\right)&\le 6^3\gamma^3\delta \sigma^2 K^{-2},\\
\Pmu\left(\vert W_{0,k}\vert> K n\right)&\le 6^3\gamma^3\delta \sigma^6 K^{-2}.
\end{align*}
\end{lemma}

\begin{proof}
We start by proving that a block has a sufficient size with large probability.\\
For all $\ell$ such that $0\le \ell < \lceil \gamma\sqrt{n}\rceil$, let $Z_\ell$ denote the number of vertices of $B^n_{0,k}$ with height $k\lceil\gamma\sqrt{n}\rceil+\ell$. In particular, $(Z_\ell)_{0\le \ell < \lceil \gamma\sqrt{n}\rceil}$ is a piece of $(\mu_j)_{j\geq k\lceil \gamma\sqrt{n}\rceil}$-BPVE with $Z_0=\lceil \delta\sqrt{n}\rceil$. By Lemma \ref{lem:offspringvarianceandDoob}, we have that, for $M>0$,
\begin{equation}\begin{split}
\Pmu\left(\max_{0\leq \ell< \gamma\sqrt{n}}\bigg|\frac{Z_{\ell}}{Z_0}-1\bigg| \geq M\right)&\le  \frac{4\lceil \delta\sqrt{n}\rceil \left(\sigma^2_{k\lceil\gamma\sqrt{n}\rceil}+\dots+\sigma^2_{(k+1)\lceil \gamma\sqrt{n}\rceil}\right)}{\left(M\lceil \delta\sqrt{n}\rceil\right)^2}\\
&\le 4\gamma\delta^{-1} M^{-2}\frac{  \sigma^2_{k\lceil\gamma\sqrt{n}\rceil}+\dots+\sigma^2_{(k+1)\lceil \gamma\sqrt{n}\rceil}        }{\gamma\sqrt{n}}.
\end{split}
\end{equation}
By Condition \ref{(I)} and Lemma~\ref{lem:deterministicLLNcuttopieces}, there exists $\widetilde{N}_k=\widetilde{N}_k(\bmu,\varepsilon,\gamma)$ such that, for all $n\ge \widetilde{N}_k$,
\[
(1-\varepsilon)\sigma^2 \le \frac{  \sigma^2_{k\lceil\gamma\sqrt{n}\rceil}+\dots+\sigma^2_{(k+1)\lceil \gamma\sqrt{n}\rceil}        }{\gamma\sqrt{n}}\le (1+\varepsilon)\sigma^2,
\]
and therefore, for all $n\ge \widetilde{N}_k$, and for $M>0$,
\begin{equation}\label{Charliewarlie}\begin{split}
\Pmu\left(\max_{0\leq \ell< \gamma\sqrt{n}}\bigg|\frac{Z_{\ell}}{Z_0}-1\bigg| \geq M\right)
&\le 4(1+\varepsilon)\gamma\sigma^2\delta^{-1}M^{-2}\le 6\gamma\sigma^2\delta^{-1}M^{-2}.
\end{split}
\end{equation}
Note that, if $\max_{0\leq \ell< \gamma\sqrt{n}}\bigg|\frac{Z_{\ell}}{Z_0}-1\bigg| < M$ then we have that $(1-M)\lceil \delta\sqrt{n}\rceil\le Z_\ell\le (1+M)\lceil \delta\sqrt{n}\rceil$, for all $0\le \ell <\lceil\gamma\sqrt{n}\rceil$.
Let $N_k=\max\{\widetilde{N}_k, (\varepsilon\delta)^{-2}, (\varepsilon\gamma)^{-2}\}$ 
and observe that, for all $n\ge N_k$,
\begin{equation}\label{kkkkkk1}\begin{split}
\left\{    \max_{0\leq \ell< \gamma\sqrt{n}}\bigg|\frac{Z_{\ell}}{Z_0}-1\bigg| < M       \right\}&\subset\left\{W_{0,k}\le (1+M)\lceil\delta\sqrt{n}\rceil \left(\sigma^2_{k\lceil\gamma\sqrt{n}\rceil}+\dots+\sigma^2_{(k+1)\lceil \gamma\sqrt{n}\rceil}  \right)   \right\}\\
&\subset\left\{W_{0,k}\le (1+M)(1+\varepsilon)\gamma\sigma^2(\delta+{n}^{-1/2})n  \right\}\\
&\subset\left\{W_{0,k}\le (1+M)(1+\varepsilon)^2\gamma\delta\sigma^2 n  \right\}.
\end{split}\end{equation}
Moreover, we have that, for all $n\ge N_k$ and for $M>0$,
\begin{equation}\label{kkkkkk11}\begin{split}
\left\{    \max_{0\leq \ell< \gamma\sqrt{n}}\bigg|\frac{Z_{\ell}}{Z_0}-1\bigg| < M       \right\}
&\subset\hspace{-1mm}\left\{\vert B^n_{0,k}\vert\ge (1-M)\gamma\delta n \right\}\cap\hspace{-1mm}\left\{W_{0,k}\le (1+M)(1+\varepsilon)^2\gamma\delta\sigma^2 n  \right\}.
\end{split}\end{equation}
Following similar steps, one can prove that for all $n\ge N_k$ and for $M>0$,
\begin{equation}\label{kkkkkk112}\begin{split}
\left\{    \max_{0\leq \ell< \gamma\sqrt{n}}\bigg|\frac{Z_{\ell}}{Z_0}-1\bigg| < M       \right\}
\subset &\left\{W_{0,k}\ge (1-M)(1-\varepsilon)\gamma\delta\sigma^2 n  \right\}
\\
&\cap \left\{\vert B^n_{0,k}\vert\le (1+M)(1+\varepsilon)^2\gamma\delta n \right\}.
\end{split}\end{equation}

Applying \eqref{kkkkkk11} and \eqref{kkkkkk112} to $M=\varepsilon$, we obtain that, for all $n\ge N_k$,
\begin{align*}
\left\{    \max_{0\leq \ell< \gamma\sqrt{n}}\bigg|\frac{Z_{\ell}}{Z_0}-1\bigg| < \varepsilon       \right\}&\subset\left\{      \frac{W_{0,k}}{\vert B^n_{0,k}\vert \sigma^2}\le \frac{1+7\varepsilon}{1-\varepsilon}        \right\}
\subset\left\{      \frac{W_{0,k}}{\vert B^n_{0,k}\vert \sigma^2}\le 1+16\varepsilon        \right\},\\
\left\{    \max_{0\leq \ell< \gamma\sqrt{n}}\bigg|\frac{Z_{\ell}}{Z_0}-1\bigg| < \varepsilon       \right\}&\subset\left\{      \frac{W_{0,k}}{\vert B^n_{0,k}\vert \sigma^2}\ge 1-9\varepsilon        \right\}.
\end{align*}
Thus, by \eqref{Charliewarlie} applied with $M=\varepsilon$ and using the definition \eqref{defgood} of $\varepsilon$-good, we obtain that, for all $n\ge N_k$,
\[
\Pmu\left(B^n_{0,k}\text{ is $\varepsilon$-good}\right)\geq 1-6\sigma^2\gamma\delta^{-1}\varepsilon^{-2}.
\]
Now, combining \eqref{Charliewarlie} and \eqref{kkkkkk112} with $M=K3^{-1}\gamma^{-1}\delta^{-1}-1$, a straightforward computation using that $K\ge6\gamma\delta$ (and thus $M\geq K6^{-1}\gamma^{-1}\delta^{-1}$) yields
\[
\Pmu\left(\vert B^n_{0,k}\vert> K n\right)\le 6^3\gamma^3\delta \sigma^2 K^{-2}.
\]
Finally, by \eqref{Charliewarlie} and \eqref{kkkkkk1} with $M=K3^{-1}\gamma^{-1}\delta^{-1}\sigma^{-2}-1$, a similar computation using that $K\ge6\gamma\delta\sigma^2$ yields
\[
\Pmu\left(\vert W_{0,k}\vert> K n\right)\le 6^3\gamma^3\delta \sigma^6 K^{-2}.
\]
This concludes the proof.
\end{proof}

Let us define, for $\delta,\gamma>0$ and $A_1,A_2>0$, the set of indices
\begin{equation}\label{defI}
\mathcal{I}(\delta,\gamma,A_1,A_2)=\left\{(i,k) : 0\leq i\leq \lfloor \delta^{-1}A_1\rfloor, 0\le k\le \lfloor\gamma^{-1}A_2\rfloor\right\}.
\end{equation}
Note that, with this choice, the box
\begin{equation}\label{defBn}
B_n=\cup_{(i,k)\in\mathcal{I}(\delta,\gamma,A_1,A_2)} B^n_{i,k}
\end{equation}
is roughly of size $A_1\sqrt{n}$ by $A_2\sqrt{n}$.
Later, we will choose $A_1>0$ and $A_2>0$ so that the depth-first exploration of the tree is contained in $B_n$ with high probability. Hence, we start by controlling the behavior of the boxes $B^n_{i,k}$ for $(i,k)\in\mathcal{I}(\delta,\gamma,A_1,A_2)$.\\
Let us define, for $\varepsilon\in(0,1/2)$, $\delta>0$, $\gamma>0$, $A_1>0$, $A_2>0$ and $n\ge1$,
\begin{equation}\label{defbad}
\text{Bad}_n(\delta,\gamma,\varepsilon,A_1,A_2)=\left\vert\left\{(i,k)\in \mathcal{I}(\delta,\gamma,A_1,A_2): B^n_{i,k}\text{ is not }\varepsilon\text{ good}\right\}\right\vert.
\end{equation}

\begin{lemma}\label{lemmart2}
Let $\bmu$ be a strictly critical environment satisfying Conditions \ref{(I)}, \ref{(II)} and \ref{(III)}. Fix $\varepsilon\in(0,1/2)$, $\delta\in(0,1)$, $\gamma\in(0,1)$, $A_1>1$ and $A_2>1$.  Let $\mathcal{I}=\mathcal{I}(\delta,\gamma,A_1,A_2)$  and $\emph{Bad}_n=\emph{Bad}_n(\delta,\gamma,\varepsilon,A_1,A_2)$, recalling the definitions \eqref{defI} and \eqref{defbad}. Assume moreover that
\begin{equation}\label{vivo}
\delta<\frac{\varepsilon^{3}}{120A_1A_2(1+\sigma^2)}.
\end{equation}
There exists $N=N(\bmu,\varepsilon,\delta,\gamma)$ such that, for all $n\ge N$, we have that, for all $K\ge6\gamma\delta(1+\sigma^2)$,
\begin{align}\label{mart22}
\Pmu\left(\max_{(i,k)\in\mathcal{I}} \vert B^n_{i,k}\vert> K n\right)&\le 6^4 A_1A_2 \sigma^2 \gamma^2K^{-2},\\\label{mart23}
\Pmu\left(\max_{(i,k)\in\mathcal{I}} \vert W_{i,k}\vert> K n\right)&\le  6^4 A_1A_2 \sigma^6 \gamma^2K^{-2},\\
\label{mart21}
\Pmu\left(\emph{Bad}_n>\delta^{-3}\right)&\le \frac{\varepsilon}{5}.
\end{align}
\end{lemma}

\begin{proof}
Let $N=\max_{0\le k\le \lfloor \gamma^{-1}A_2\rfloor} N_k(\bmu, \varepsilon,\delta,\gamma)$, where the $N_k$'s are given by Lemma \ref{lemmart1}.\\
First, note that \eqref{mart22} and \eqref{mart23} are straightforward consequences of Lemma \ref{lemmart1} and a union bound over $\mathcal{I}$, observing that $\vert\mathcal{I}\vert\le 4A_1A_2\delta^{-1}\gamma^{-1}$.
Second, for $n\ge N$, by Lemma \ref{lemmart1} and Remark \ref{remindep}, $\text{Bad}_n$ is stochastically dominated by a binomial random variable $Y$ with parameters $\lfloor 4A_1A_2\delta^{-1}\gamma^{-1}\rfloor$ and $\min\{1,6\sigma^2\gamma\delta^{-1}\varepsilon^{-2}\}$. We will use the following bound, which can be deduced from a Chernoff bound. For all $\alpha>8$,
\begin{equation}\label{binomialeq}
\mathbb{P}[Y\ge (1+\alpha)\mathbb{E}[Y]]\le \exp\left((\alpha - (1+\alpha)\log (1+\alpha)) \dE[Y]\right)\le\exp\left(-\alpha\mathbb{E}[Y]\right)\le \frac{1}{\alpha\dE[Y]}
\end{equation}
Let $\alpha =\delta^{-3}\dE[Y]^{-1}-1$, so that by \eqref{vivo}, 
\[
\delta^3\mathbb{E}[Y]\le 24 A_1A_2\sigma^2\varepsilon^{-2} \delta\leq  \frac{\varepsilon}{5}<1/9
\]
and $\alpha >8$ . Using \eqref{binomialeq} together with the fact that $Y$ stochastically dominates $\text{Bad}_n$, we have that, for $n\ge N$,
\[
\Pmu\left(\text{Bad}_n\ge \delta^{-3}\right)\le \frac{\delta^3\dE[Y]}{\dE[Y](1-\delta^3\dE[Y])}\leq 2\delta^3\leq \frac{\varepsilon}{5}.
\]
This concludes the proof.

\end{proof}

\begin{lemma}\label{mart3}
Let $\bmu$ be a strictly critical environment satisfying Conditions \ref{(I)}, \ref{(II)} and \ref{(III)}. We have that
\begin{align*}
&\lim_{n\rightarrow +\infty}\frac{\sum_{k=0}^{n-1}\sigma_{H_k}^2}{n}=\sigma^2\text{, in $\Pmu$-probability}. 
\end{align*}
In other words, \eqref{eqn:WhittAverageVariance} is satisfied.
\end{lemma}
\begin{proof}
For $n\geq 1$, let us denote, $V_n$ the set of vertices of $\cF_n$ (i.e.~the vertices with depth-first label smaller than $n$), so that
\[
\sum_{k=0}^{n}\sigma_{H_k}^2=\sum_{v\in V_n} \sigma_{h(v)}^2.
\]
Now, for $\varepsilon\in(0,1/2)$, $\delta\in(0,1)$, $\gamma\in(0,1)$, $A_1>1$ and $A_2>1$, recalling the definitions \eqref{defI}, \eqref{defBn} and \eqref{defbad}, let us define
\begin{equation}\begin{split}
\mathcal{I}^{\mathrm{o}}&=\left\{(i,k)\in \mathcal{I}: B^n_{i,k}\subset V_n, B^n_{i,k}\text{ is $\varepsilon$-good}\right\}\\
\partial \mathcal{I}&=\left\{(i,k)\in \mathcal{I}: B^n_{i,k}\cap V_n\neq \emptyset, B^n_{i,k}\text{ is $\varepsilon$-good}\right\}\setminus \mathcal{I}^{\mathrm{o}}\\
\mathcal{I}^c&=\left\{(i,k)\in \mathcal{I}: B^n_{i,k}\text{ is not $\varepsilon$-good}\right\}.
\end{split}\end{equation}
Let us define the events
\begin{align}
E_1&=\left\{V_n\subset B_n\right\},\\
E_2&=\left\{\text{Bad}_n\le \delta^{-3}\right\},\\
E_3&=\left\{\max_{(i,k)\in\mathcal{I}} \vert B^n_{i,k}\vert \le \gamma^{1/2} n \right\},\\
E_4&=\left\{\left\vert\left\{(i,k)\in\mathcal{I}: \vert B^n_{i,k} \vert > \gamma\delta^{1/2}n\right\}\right\vert\le \delta^{-2}  \right\},\\
E_5&=\left\{\max_{(i,k)\in\mathcal{I}} \vert W_{i,k}\vert\le \delta^3\varepsilon n \right\}.
\end{align}
$E_1$ allows to locate $\cF_n$ inside $B_n$, $E_2$ bounds the number of not $\varepsilon$-good boxes. Then, to limitate the influence of those boxes, $E_3$ (resp. $E_5$ bounds the maximal volume (resp. weight) of a box. It turns out that the bound in $E_3$ is too rough, and $E_4$ bounds the number of boxes having some already large volume.
\\
On the event $E_1$,  using the definition \eqref{defgood} of $\varepsilon$-good blocks, we have that
\begin{equation}\label{redhot0}\begin{split}
\sum_{k=0}^{n}\sigma_{H_k}^2&\le \sum_{(i,k)\in\mathcal{I}^{\mathrm{o}}} W_{i,k} + \sum_{(i,k)\in\partial\mathcal{I}} W_{i,k}+\sum_{(i,k)\in\mathcal{I}^c} W_{i,k}\\
&\le (1+16\varepsilon)\sigma^2 \sum_{(i,k)\in\mathcal{I}^{\mathrm{o}}} \vert B^n_{i,k}\vert+ (1+16\varepsilon)\sigma^2 \sum_{(i,k)\in\partial \mathcal{I}} \vert B^n_{i,k}\vert +\sum_{(i,k)\in\mathcal{I}^c} W_{i,k}\\
&\le (1+16\varepsilon)\sigma^2 n + (1+16\varepsilon) \sigma^2 \sum_{(i,k)\in\partial \mathcal{I}} \vert B^n_{i,k}\vert + \vert\mathcal{I}^c\vert\max_{(i,k)\in\mathcal{I}}\vert W_{i,k}\vert,
\end{split}\end{equation}
and
\begin{equation}\label{redhot01}\begin{split}
\sum_{k=0}^{n}\sigma_{H_k}^2&\ge (1-16\varepsilon) \sum_{(i,k)\in\mathcal{I}^{\mathrm{o}}} \vert B^n_{i,k}\vert\\
&\ge (1-16\varepsilon)\sigma^2 \left(n- \sum_{(i,k)\in\partial \mathcal{I}} \vert B^n_{i,k}\vert -\vert\mathcal{I}^c\vert\max_{(i,k)\in\mathcal{I}}\vert B_{i,k}\vert\right)\\
&\ge \sigma^2 n -16\sigma^2 \varepsilon n -  \sigma^2\left(\sum_{(i,k)\in\partial \mathcal{I}} \vert B^n_{i,k}\vert +\vert\mathcal{I}^c\vert\max_{(i,k)\in\mathcal{I}}\vert B_{i,k}\vert\right).
\end{split}\end{equation}
Now, note that because the exploration is depth-first, we have that if $(i,k)\in\partial \mathcal{I}$, then $(j,k)\in \mathcal{I}^{\mathrm{o}}\cup \mathcal{I}^c$ for all $0\le j<i$ and $B^n_{\ell,k}\cap V_n=\emptyset$ for all  $\ell>i$. This implies that for each $k\ge0$, there is at most one index $i$ such that $(i,k)\in\partial \mathcal{I}$ and therefore
\[
\vert\partial\mathcal{I}\vert\le \gamma^{-1}A_2+1\le 2\gamma^{-1}A_2.
\]
Hence, on the event $E_3\cap E_4$, we have that, if
\begin{align}\label{vivo2}
\delta&\le \frac{\varepsilon^2}{4A_2^2},\\ \label{vivo3}
\gamma&\le {\varepsilon^2\delta^6},
\end{align}
then
\begin{equation}\label{redhot}
\sum_{(i,k)\in\partial \mathcal{I}} \vert B^n_{i,k}\vert \le \gamma^{1/2}n\times \delta^{-2}+\gamma\delta^{1/2}n\times 2\gamma^{-1}A_2= (\gamma^{1/2}\delta^{-2}+2\delta^{1/2} A_2)n\le 2\varepsilon n.
\end{equation}
Observe as well that $\vert\mathcal{I}^c\vert=\text{Bad}_n$, thus on the event $E_2\cap E_3\cap E_5$, we have that, if \eqref{vivo3} holds then
\begin{equation}\label{redhot2}\begin{split}
\vert\mathcal{I}^c\vert\max_{(i,k)\in\mathcal{I}}\vert W_{i,k}\vert& \le \delta^{-3}\times \delta^{3}\varepsilon n = \varepsilon n,\\
\vert\mathcal{I}^c\vert\max_{(i,k)\in\mathcal{I}}\vert B_{i,k}\vert& \le \delta^{-3}\times \gamma^{1/2} n \le \varepsilon n.
\end{split}
\end{equation}
Putting \eqref{redhot0}, \eqref{redhot01}, \eqref{redhot} and \eqref{redhot2} together and using that $\varepsilon<1/2$, we obtain that, if \eqref{vivo2} and \eqref{vivo3} hold then
\begin{equation}\label{nuit}\begin{split}
\left\{ \bigcap_{1\le i\le 5}E_i\right\}\subset\left\{\left\vert\frac{1}{n}\sum_{k=0}^{n}\sigma_{H_k}^2-\sigma^2 \right\vert \le 35(1+\sigma^2)\varepsilon\right\}.
\end{split}
\end{equation}
To conclude the proof, we will show that for all $\varepsilon\in(0,1/2)$, there exist $A_1>1$, $A_2>1$, $\delta\in(0,1)$, $\gamma\in(0,1)$ satisfying \eqref{vivo2} and \eqref{vivo3}, and there exists $N\in\mathbb{N}$ such that for all $n\ge N$,
\begin{equation}\label{finalei}
\Pmu\left(\left\{ \bigcap_{1\le i\le 5}E_i\right\}^c\right)\le \varepsilon.
\end{equation} 
Note that, by \eqref{defBn} and \eqref{defbij}, for all $A_1,A_2>1$, $\delta\in(0,1)$, $\gamma\in(0,1)$ and for all $n\ge1$,
\[
\left\{V_n\subset B_n\right\}^c\subset \left\{w(\mathcal{F}_n)>A_1\sqrt{n}\right\}\cup \left\{h(\mathcal{F}_n)>A_2\sqrt{n}\right\}.
\]
By Lemma \ref{prop:firstbounds}, for all $\varepsilon\in(0,1/2)$, there exists $A_1=A_1(\varepsilon)>1$, $A_2=A_2(\varepsilon)>1$ and $N_1=N_1(\bmu,\varepsilon)$ such that, for all $n\ge N_1$,
\begin{equation}\label{finale1}
\Pmu\left(E_1^c\right)\le \frac{\varepsilon}{5}.
\end{equation}
Now, let us choose
\begin{align*}
\delta&=\frac{\varepsilon^{3}}{6^6A_1A_2^2(1+\sigma^2)^2},\\
\gamma&=\frac{\varepsilon^3\delta^6}{6^5A_1 A_2(1+\sigma^2)^3},
\end{align*}
so that \eqref{vivo2} and \eqref{vivo3} are satisfied. Then \eqref{vivo} holds, so that we can apply Lemma \ref{lemmart2}: first, there exists $N_2=N_2(\bmu,\varepsilon)$ such that, for all $n\ge N_2$,
\begin{equation}\label{finale2}
\Pmu\left(E_2^c\right)\le \frac{\varepsilon}{5};
\end{equation}
second, using \eqref{mart22} with $K=\sqrt{\gamma}\ge 6\gamma\delta(1+\sigma^2)$, we have that, for all $n\ge N_2$
\begin{equation}\label{finale3}
\Pmu\left(E_3^c\right)\le 6^4 A_1A_2 \sigma^2 \gamma\le \frac{\varepsilon}{5};
\end{equation}
and third, using \eqref{mart23} with $K=\delta^3\varepsilon\ge 6\gamma\delta(1+\sigma^2)$, we obtain that, for all $n\ge N_2$,
\begin{equation}\label{finale5}
\Pmu\left(E_5^c\right)\le 6^4 A_1A_2 \sigma^6 \gamma^2\delta^{-6}\varepsilon^{-2}\le \frac{\varepsilon}{5}.
\end{equation}
Now, by Lemma \ref{lemmart1} with $K=\gamma\delta^{1/2}\ge 6\gamma\delta(1+\sigma^2)$, there exists $N_3=N_3(\bmu,\varepsilon)$ such that, for all $(i,k)\in\mathcal{I}$,
\[
\Pmu\left(\vert B^n_{i,k} \vert > \gamma\delta^{1/2}n\right)\le 6^3\sigma^2\gamma.
\]
Therefore,  $\left\vert\left\{(i,k)\in\mathcal{I}: \vert B^n_{i,k} \vert > \gamma\delta^{1/2}n\right\}\right\vert$ is dominated by a binomial random variable $Y$ with parameters $\lfloor 4A_1A_2\delta^{-1}\gamma^{-1}\rfloor$ and $6^3\sigma^2\gamma$. Using \eqref{binomialeq} and the definition of $\delta$, a straightforward computation yields that
\begin{equation}\label{finale4}
\Pmu\left(E_4^c\right)\le \mathbb{P}\left(Y>\delta^{-2}\right)\le \frac{\varepsilon}{5}.
\end{equation}
Finally, by \eqref{finale1}, \eqref{finale2}, \eqref{finale3}, \eqref{finale4} and \eqref{finale5}, for all $\varepsilon\in(0,1)$, choosing $A_1,A_2,\delta,\gamma$ as above and letting $N=\max\{N_1,N_2,N_3\}$, we have that \eqref{finalei} holds for all $n\ge N$. Using \eqref{nuit}, this yields that, for all $\varepsilon\in(0,1/2)$, there exists $N\in\mathbb{N}$ such that for all $n\ge N$,
\[
\Pmu\left(\left\vert\frac{1}{n}\sum_{k=0}^{n}\sigma_{H_k}^2-\sigma^2 \right\vert > 35(1+\sigma^2)\varepsilon\right)\le \varepsilon,
\]
which implies the required convergence in probability and concludes the proof.
\end{proof}

\subsection{Proof of Corollary \ref{cor:Luka}}
{Let $\phi$ be the map from $D([0,1])$ to itself such that for all $f\in D([0,1])$ and $t\in [0,1]$, $\phi(f)(t):=f(t)-\inf_{0\leq s\leq t} f(s)$. One checks straightforwardly that $\phi$ is continuous in Skorokhod space. Hence, Corollary~\ref{cor:Luka} follows from Theorem~\ref{thm:Luka} and the continuous mapping theorem applied to $\phi$. 
}

\section{Proof of Theorem~\ref{th:main}}\label{sec:heightprocess}

As we explain in Section~\ref{subsec:extraction}, Theorem~\ref{th:main} can be deduced from  Theorem~\ref{thm:jointcvLukaHeight}. Hence, most of this section is devoted to the proof of  Theorem~\ref{thm:jointcvLukaHeight}, stating the joint convergence of the {\L}ukasiewicz path and the height process.\\
To prove this convergence, we introduce Kesten's tree in varying environment $\cT^*$, also known as the Geiger tree, 
which is the local limit of the tree $\cT$ conditioned on having a large height.
The Geiger tree consists of an infinite spine along which the offspring distribution is size-biased at each generation, and on which are attached trees with the non-biased offspring distributions.
The first step of the proof of Theorem~\ref{thm:jointcvLukaHeight},  in Section~\ref{subsec:spinedecompoheight}, is to prove that for a typical
 vertex in $\cT^*$ with depth-first label $x$,  the ratio $(X_x-I_x)/H_x$ is close to $\sigma^2/2$, where we recall that $(X_x-I_x)$ is the {\L}ukasiewicz path reflected above its running minimum.
The second step of the proof of Theorem~\ref{thm:jointcvLukaHeight} is to transpose this to $\cT$ by comparing the distributions of $\cT$ and $\cT^*$, see Proposition~\ref{prop:badspineT}. 
This allows us to show that, for most vertices $x$ in $\cF_n$, 
the ratio $(X_x-I_x)/H_x$ is also close to $\sigma^2/2$, see Proposition~\ref{prop:fewbadvertices}. 
We conclude the proof of Theorem~\ref{thm:jointcvLukaHeight} by a continuity argument: this is done following the proof of~\cite[Theorem 2.3.1]{DuquesneLeGall} for the usual Bienaym\'e-Galton-Watson case, with more technicalities due to the varying environment.\\
Finally, in Section~\ref{subsec:extraction}, we explain how to obtain Theorem~\ref{th:main} from Theorem~\ref{thm:jointcvLukaHeight}.

\subsection{The height process via spine decomposition}\label{subsec:spinedecompoheight}
In this section, we study the height process of a BPVE and 
the main result is Proposition~\ref{prop:badspineT}.
We start by defining the distribution of the {Geiger tree $\cT^*$} on an environment $\bmu$. 
Let $(\overline{\xi}_m)_{m\geq 0}$ be a sequence of independent variables such that for all $m\geq 0$, 
$\overline{\xi}_m$ is distributed as a size-biased version of $\mu_m$. Formally,
$\mathrm P_{\bmu}(\overline{\xi}_m = i) = i\mathrm P_{\bmu}(\xi_m = i)/\mathrm E_{\bmu}[\xi_m]$ for all $i\geq 0$,
where $\xi_m$ is  random variable of distribution $\mu_m$. 
Let $(\mathcal T_{i, m})_{i\geq 1, m\geq 0}$ be a family of independent trees, independent of $(\overline{\xi}_m)_{m\geq 0}$, where for all $m\geq 0$ and $i\geq 1$, 
$\cT_{i, m}$ is a BPVE in environment $(\mu_{m+k})_{k\geq 0}$. 
We will define recursively the pair $(\cT^*,(v_m)_{m\ge0})$ consisting of the Geiger tree and a sequence of marked vertices, called the spine. We start with a pair $(\cT_0,v_0)$ where the tree $\cT_0$ consists of the single vertex $v_0$. Assume that for some $m\ge0$, we have constructed the pair $(\cT_m, (v_i)_{1\le i\le m})$ where $\cT_m$ is a tree of height $m+1$ and $v_1,\dots,v_m$ are spine vertices so that $v_{i+1}$ is an offspring of $v_i$, for all $i\ge0$. Then, we construct $(\cT_{m+1}, (v_i)_{1\le i\le m+1})$ as follows:
\begin{itemize}
\item assign $\overline{\xi}_m$ children to the spine vertex $v_m$,
\item choose a new spine vertex $v_{m+1}$ uniformly at random among the children of $v_m$,
\item attach $\cT_{i, m+1}$ to the $i$-th non-spine child of $v_m$, in lexicographical order.
\end{itemize}
As $m\rightarrow \infty$, this procedure produces the pair $(\cT^*,(v_m)_{m\ge0})$ where $\cT^*$ is the Geiger tree and  $(v_m)_{m\geq 0}$ is its spine. Note that, if $\bmu$ is strictly critical, then the trees $\cT_{i,m}$ are finite almost surely and $\cT^*$ contains a single infinite path. This implies that knowing $\cT^*$ is enough to recover  the spine $(v_m)_{m\ge0}$.\\

We want to study the {\L}ukasiewicz path on this infinite tree and compare it to the {\L}ukasiewicz path of a finite tree conditioned to be large. A difficulty is that the depth-first labeling and the {\L}ukasiewicz path are not well-defined on the Geiger tree $\mathcal{T}^*$, as the depth-first exploration will never be able to cross the spine. Nevertheless, the value of the {\L}ukasiewicz path reflected above its running minimum can be recovered from partial information. This value coincides with $\widetilde{X}_{x}$ introduced below, which is well-defined on finite, infinite trees or forests. At the beginning of Section \ref{sect:prooflukaheight}, we will provide details on how to relate $\widetilde{X}_x$ and the {\L}ukasiewicz path on the explored forest $\mathcal{F}_n$.\\
Let $\cF$ be a forest (possibly just a tree, infinite or finite) and recall the lexicographical order defined in Section \ref{sectorder}. We want to define the quantity $\widetilde{X}_x$ for a vertex $x$ of $\cF$. For a vertex $x$ at height $m\ge1$, there exists a unique $m$-uplet $(s_0,\dots,s_m)$ such that $x=v_{s_m,m}$ and, for all $i\in\{0,\dots,m-1\}$, $v_{s_i,i}$ is the ancestor of $x$ at height $i$. Now, for each $i\in\{0,\dots,m-1\}$, let $\zeta_i(x)$ be the number of offspring of $v_{s_i,i}$ with lexicographical label strictly larger than that of $v_{s_{i+1},i+1}$, i.e.~with label $v_{j,i+1}$ for some $j>s_{i+1}$. We define
\begin{equation}\label{deftilde}
\widetilde{X}_x=\zeta_0(x)+\dots+\zeta_{m-1}(x).
\end{equation}

%

For  a vertex $x\in \cF$,  denote $\mathcal{L}(x)$ the depth-first label of $x$, as explained in Section \ref{sectorder}. By definition of the height process, we have that $H_{\mathcal{L}(x)}=h(x)$, where $h(x)$ is the height of $x$ in $T$, as defined in Section \ref{sectorder}. Finally, note that, when applied to the spine vertex $v_m$, we have that
\begin{equation}\label{insidious}
\widetilde{X}_{v_m}=\zeta_0+\dots+\zeta_{m-1},
\end{equation}
where the random variables $\zeta_i$ are independent,  uniform in $\{0,\dots,\overline{\xi}_i-1\}$ with $\overline{\xi}$ has the size-biased distribution of $\mu_i$, as defined above \eqref{eq:phik}. These are the random variables used in Condition \ref{(IV)}.\\
%
For any tree $T$ and any $\varepsilon >0$, we say that $x$ is an \textbf{$\varepsilon$-bad vertex} if
\begin{equation}\label{badbadbad}
\left\vert \frac{\widetilde{X}_{x}}{h(x)}-\frac{\sigma^2}{2}\right\vert \geq \varepsilon,
\end{equation}
where $h(x)$ is the height of $x$ in $T$.
For $m\in \dN$, say that $T$ is $(\varepsilon,m)$-bad if there are at least $\varepsilon Z_m(T)$ $\varepsilon$-bad vertices among the $Z_m(T)$ vertices at generation $m$ of $T$. 

\begin{lemma}\label{lem:fewbadspineGeigertree}
Let $\bmu$ be a strictly critical environment satisfying Condition \ref{(IV)}. For all $\varepsilon >0$ and for all $\delta>0$, there exists $M=M(\bmu,\varepsilon,\delta)$ such that, for all $m\ge M$, 
\[
\Pmu\left(\cT^*\text{ is }(\varepsilon,m)\text{-bad}\right)\leq \delta.
\]
\end{lemma}

\begin{proof}
As above, we let $(v_m)_{m\ge0}$ be the spine of $\cT^*$, which is the only infinite path in the tree. We will use the  lexicographical order defined in Section \ref{sectorder}.
By~\cite[Lemma 1.2]{KerstingVatutinBook}, for $m\ge1$, for a tree $T$ with height at least $m$, and for $k\in\{1,\dots,Z_m(T)\}$,
we have that
\[
\Pmu\left(\mathcal{T}^*\stackrel{m}{=}T, v_m=v_{k,m}\right)=\frac{1}{Z_m(T)}\Pmu\left(\mathcal{T}^*\stackrel{m}{=}T\right),
\]
where the notation $\stackrel{m}{=}$ means that the two trees are identical up to height $m$ included.\\
Let $E_m^B$ be the countable set of trees of height exactly $m$ that are $(\varepsilon,m)$-bad. For a tree $T\in E_m^B$, we denote $V^B_m(T)$ the set of its vertices at height $m$ that are $\varepsilon$-bad. Hence, for $T\in E_m^B$, we have that $|V_m^B(T)|\ge \varepsilon Z_m(T)$. Using the above, recalling that $v_m$ denotes the spine vertex of $\cT^*$ at height $m$, we have that
\begin{equation}\label{encanto}
\begin{split}
\Pmu(\text{$\cT^*$ is $(\varepsilon,m)$-bad})&= \sum_{T\in E_m^B} \Pmu\left(\mathcal{T}^*\stackrel{m}{=}T\right)\\
&= \sum_{T\in E_m^B} \frac{Z_m(T)}{|V_m^B(T)|}\sum_{v\in V_m^B(T)} \Pmu\left(\mathcal{T}^*\stackrel{m}{=}T, v_m=v\right)\\
&\le \frac{1}{\varepsilon} \sum_{T\in E_m^B} \sum_{v\in V_m^B(T)}\Pmu\left(\mathcal{T}^*\stackrel{m}{=}T, v_m=v\right)\\
&= \frac{1}{\varepsilon} \sum_{T\in E_m^B} \Pmu\left(\mathcal{T}^*\stackrel{m}{=}T, v_m\text { is }\varepsilon\text{-bad}\right)\\
&\le \frac{1}{\varepsilon}  \Pmu\left( v_m\text { is }\varepsilon\text{-bad}\right).
\end{split}
\end{equation}
Hence, using \eqref{badbadbad}, it is enough to prove that, for $m$ large enough
\begin{equation}\label{eqn:vmbad}
\Pmu(\text{$v_m$ is $\varepsilon$-bad})\le \varepsilon\delta. 
\end{equation}
By \eqref{insidious}, $\widetilde{X}_{v_m}$ is a sum of $m$ independent random variables as in Condition \ref{(IV)}.
Hence, by Condition \ref{(IV)}, we have that $\mathrm P_\bmu$-almost surely, 
\[
\frac{\widetilde{X}_{v_m}}m \to \frac{\sigma^2}2,
\]
and thus this convergence holds in $\Pmu$-probability as well.
Hence, noting that $h(v_m)=m$, this implies that, for all $\varepsilon>0$ and for all $\delta>0$, there exists $M=M(\bmu,\varepsilon,\delta)$ such that, for all $m\geq M$,
\[
\Pmu(\text{$v_m$ is $\varepsilon$-bad})=\Pmu\left( \left|\frac{\widetilde{X}_{v_m}}{m}-\frac{\sigma^2}{2}\right| \ge \varepsilon\right)\le \varepsilon\delta.
\]
This concludes the proof.
\end{proof}

A similar result holds on the $\bmu$-BPVE $\cT$ conditioned on $\{h(\cT)\geq m\}$, as stated below.
\begin{proposition}\label{prop:badspineT}
Let $\bmu$ be a strictly critical environment that satisfies the four Conditions \ref{(I)},\ref{(II)},\ref{(IV)} and \ref{(V)}. 
For all $\varepsilon>0$ and $\delta>0$, there exists $M = M(\bmu,\varepsilon, \delta)$ such that, for all $m\geq M$ and for all $k\in\{0,\dots, m\}$,
\[
\Pmu\left( \left.\cT\text{ is }(\varepsilon,m)\text{-bad }\right|h(\cT)\ge k\right) \leq \delta.
\]
\end{proposition}

\begin{proof}
Fix $\varepsilon>0$ and $\delta>0$. Start by noting that, for all $m\ge0$ and all $k\in\{0,\dots,m\}$, using that a tree needs to be of height at leat $m$ to be $(\varepsilon,m)$-bad, we have that
\begin{align*}
\Pmu\left( \left.\cT\text{ is }(\varepsilon,m)\text{-bad }\right|h(\cT)\ge k\right)&=\Pmu\left( \left.\cT\text{ is }(\varepsilon,m)\text{-bad }\right|h(\cT)\ge m\right)\frac{\Pmu\left( h(\cT)\ge m\right)}{\Pmu\left(h(\cT)\ge k\right)}\\
&\le\Pmu\left( \left.\cT\text{ is }(\varepsilon,m)\text{-bad }\right|h(\cT)\ge m\right),
\end{align*}
hence it is enough to prove the statement for $k=m$.\\
By Lemma \ref{lem:heightsurvivalproba}, there exists $\delta'= \delta'(\bmu,\delta)>0$ and $M_0=M_0(\bmu,\delta)$ such that, for all $m\ge M_0$, we have that
\begin{equation}\label{netfliss1}
\Pmu\left(\left. Z_m(\cT)\le \delta'm\right| h(\cT)\ge m\right)\le \frac{\delta}{2} \text{ and } \Pmu\left(h(\cT)\ge m\right)\ge \frac{1}{2\sigma^2 m}.
\end{equation}
By Lemma \ref{lem:fewbadspineGeigertree}, there exists $M_1=M_1(\bmu,\varepsilon,\delta)$, such that, for all $m\ge M_1$, we have that
\begin{equation}\label{netfliss3}
\Pmu\left(\cT^*\text{ is }(\varepsilon,m)\text{-bad}\right)\leq \frac{\delta'\delta}{4\sigma^2}.
\end{equation}
As in the previous proof, for $m\geq 1$ 
and two  trees $T$ and $T'$, 
we write $T\overset{m}{=}T'$ when both trees are identical up to height $m$.  We will denote again $E_m^B$ the  set of trees of height exactly $m$ that are $(\varepsilon,m)$-bad, and we will denote $\widetilde{E}_m^B=\{T\in E_m^B: Z_m(T)>\delta'm\}$, which are both countable.\\
By~\cite[Lemma 1.2]{KerstingVatutinBook}, for all $m\ge1$ and for all $T\in E_m^B$, we have that
\begin{equation}\label{eq:KerstingVatutin12}
\Pmu\left(\cT\stackrel{m}{=} T\right)=\frac{\Pmu\left(\cT^*\stackrel{m}{=} T\right)}{Z_m(T)}.
\end{equation}
Define $M=M(\bmu,\varepsilon,\delta)=\max\{M_0,M_1\}$. Using \eqref{encanto}, \eqref{netfliss1}, \eqref{netfliss3} and \eqref{eq:KerstingVatutin12}, we conclude that for all $m\ge M$, 

\begin{align*}
\Pmu\left( \left.\cT\text{ is }(\varepsilon,m)\text{-bad }\right|h(\cT)\ge m\right)\le& \Pmu\left( \left. Z_m(\cT)\le \delta'm\right| h(\cT)\ge m\right)\\
& + \frac{\Pmu\left(\cT\text{ is }(\varepsilon,m)\text{-bad }, \ Z_m(\cT)> \delta' m\right)}{\Pmu\left(h(\cT)\ge m\right)}\\
\le& \frac{\delta}{2}
 + \frac{\Pmu\left(\cT\text{ is }(\varepsilon,m)\text{-bad }, \ Z_m(\cT)> \delta' m\right)}{\Pmu\left(h(\cT)\ge m\right)}\\
=& \frac{\delta}{2} + \sum_{T\in\widetilde{E}_m^B} \frac{\Pmu\left(\cT\stackrel{m}{=} T\right)}{\Pmu\left(h(\cT)\ge m\right)}\\
=& \frac{\delta}{2} + \sum_{T\in\widetilde{E}_m^B} \frac{\Pmu\left(\cT^*\stackrel{m}{=} T\right)}{Z_m(T)\Pmu\left(h(\cT)\ge m\right)}\\
\le& \frac{\delta}{2} +\frac{2\sigma^2}{\delta'} \sum_{T\in\widetilde{E}_m^B} \Pmu\left(\cT^*\stackrel{m}{=} T\right)\\
\le& \frac{\delta}{2} +\frac{2\sigma^2}{\delta'} \sum_{T\in {E}_m^B} \Pmu\left(\cT^*\stackrel{m}{=} T\right)\\
=& \frac{\delta}{2} +\frac{2\sigma^2}{\delta'}
\Pmu\left(\cT^*\text{ is }(\varepsilon,m)\text{-bad}\right)\\
\le& \frac{\delta}{2}+\frac{\delta}{2}=\delta.
\end{align*}
\end{proof}

Using the above, we can now prove that the proportion of bad vertices in the explored forest is arbitrarily small.

\begin{proposition}\label{prop:fewbadvertices}
Let $\bmu$ be a strictly critical environment satisfying Conditions \ref{(I)}, \ref{(II)},\ref{(IV)}  and \ref{(V)}. For every $\varepsilon \in(0,1/2)$ and all $\delta\in(0,1)$, there exists $N=N(\bmu,\varepsilon,\delta)$ such for all $n\geq N$, 
\[
\Pmu(\text{$\cF_n$ has at most $\varepsilon n$ $\varepsilon$-bad vertices})\geq 1-\delta.
\]
\end{proposition}

\begin{proof}
Let us fix $\varepsilon \in (0,1/2)$ and $\delta\in(0,1)$. 
We start by bounding the size of the explored forest $\mathcal{F}_n$ and the number of tall trees. By Proposition~\ref{prop:firstbounds} and Lemma \ref{lem:heightsurvivalproba}, there exist $A_1=A_1(\delta)\ge 1$, $A_2=A_2(\delta)\ge 1$, $M=M(\delta)\ge1$ and $N_1=N_1(\bmu,\delta)$ such that, for all $n\ge N_1$, we have that
\begin{equation}
\Pmu\left(\left\{t(\mathcal{F}_n)\le w(\mathcal{F}_n)\le A_1\sqrt{n},\ h(\mathcal{F}_n)\le A_2\sqrt{n}\right\}^c\right)\le \frac{\delta}{6},
\end{equation}
and
\begin{equation}
\Pmu\left(
\left| \left\{i\in\{1,\dots,\lfloor A_1\sqrt{n}\rfloor\}: h\left(\mathcal{T}^{(i)}\right)\ge A^{-1}\varepsilon^2\sqrt{n}\right\}\right|>M\right)\le \frac{\delta}{6}.
\end{equation}
For the second inequality above, we used the fact that the number of tall trees is stochastically dominated by a binomial with parameters $4c^{-1}A_1\varepsilon^{-2}n^{-1/2}$ and $\lfloor A_1\sqrt{n}\rfloor$, together with Hoeffding's inequality.\\
Hence, we have that, for all $n\ge N_1$,
\begin{equation}\label{huel0}
\Pmu\left(E_{n,1}^c\right)\le \frac{\delta}{4},
\end{equation}
where
\begin{equation}\begin{split}
E_{n,1}=&\left\{t(\mathcal{F}_n)\le w(\mathcal{F}_n)\le A_1\sqrt{n},\ h(\mathcal{F}_n)\le A_2\sqrt{n}\right\}\\
&\cap\left\{     \left| \left\{i\in\{1,\dots,\lfloor A_1\sqrt{n}\rfloor\}: h\left(\mathcal{T}^{(i)}\right)\ge A^{-1}\varepsilon^2\sqrt{n}\right\}\right|\le M       \right\}.
\end{split}
\end{equation}

The strategy of the rest of the proof is as follows. Ideally, we would like to control the number of bad vertices by using Proposition \ref{prop:badspineT} and a union bound over all generations up to $A_2\sqrt{n}$. We encounter two problems in doing so. First, Proposition \ref{prop:badspineT} does not apply to small generations, hence we distinguish below between the generations at a relatively small height, and the other generations. Second, a union bound would not be sufficient because we have a number of generations of order $\sqrt{n}$. To solve this second issue, we apply Proposition \ref{prop:badspineT} only to a bounded number of generations (the $m_j$'s defined below \eqref{pasdenom}), regularly distributed over $A_2\sqrt{n}$ generations.  Then, using that $\widetilde{X}$ is non-decreasing from a vertex to its descendants, we show that if a lot of vertices $x$ at some intermediate generation $i$ are bad by having a value too large for $\widetilde{X}_{x}$, then a lot of vertices at the next generation $m_j$ are bad, see \eqref{En4}. Similarly, if many vertices at generation $i$ are bad by having a value too low for $\widetilde{X}_{x}$, then many vertices at the previous generation $m_{j-1}$ are bad, see \eqref{En5}.

Let us split the vertices depending on whether their height is larger than $A_1^{-1}\varepsilon^2\sqrt{n}$ or not. Observe that, on the event $E_{n,1}$, we almost surely have that
\begin{equation}\label{huel1}
\begin{split}
\left|\left\{v\in\mathcal{F}_n:\ h(v)\le A_1^{-1}\varepsilon^2\sqrt{n}\right\}\right|\le A_1^{-1}\varepsilon^2\sqrt{n}\times A_1\sqrt{n}=\varepsilon^2 n<\frac{\varepsilon}{2}n.
\end{split}
\end{equation}
The above deals with the vertices with low height, thus now we need to deal with the vertices with large height. For this purpose, recall that
\[
\mathcal{F}_n=\cT^{(1)}\cup\dots\cup\cT^{(t(\mathcal{F}_n))},
\]
hence on the event $E_{n,1}$, we have that
\begin{equation}
\mathcal{F}_n\subset \bigcup_{i=1}^{\lfloor A_1\sqrt{n}\rfloor}\cT^{(i)}.
\end{equation}
Let us define the event
\begin{equation}\label{huel2}
\begin{split}
E_{n,2}=\bigcap_{i=1}^{\lfloor A_1\sqrt{n}\rfloor}& \left\{\left\{w\left(\cT^{(i)}\right)>A_1\sqrt{n}\right\}\right.\\
&\left.\cup\bigcap_{m=\lceil A_1^{-1}\varepsilon^2\sqrt{n}\rceil}^{\lfloor A_2\sqrt{n}\rfloor}  \left\{\left|\left\{v\in\cT^{(i)}: h(v)=m, \ v \text{ is }\varepsilon\text{-bad}\right\}\right|<\frac{\varepsilon}{2MA_2} \sqrt{n}\right\}\right\}.
\end{split}
\end{equation}
Note that by \eqref{huel1}, \eqref{huel2} and the fact that on $E_{n,1}$, the trees in $\mathcal{F}_n$ have a width smaller than $A_1\sqrt{n}$ and less than $M$ of them have height larger than $A^{-1}\varepsilon^2\sqrt{n}$, we have that
\begin{equation}
E_{n,1}\cap E_{n,2}\subset \left\{\left|\left\{v\in\mathcal{F}_n: \ v \text{ is }\varepsilon\text{-bad}\right\}\right|<\varepsilon n\right\}.
\end{equation}
We now want to prove that
\begin{equation}\label{huel3}
\Pmu\left(\left(E_{n,1}\cap E_{n,2}\right)^c\right)\le \delta.
\end{equation}
For this purpose, let us define $\gamma= \varepsilon^3\delta/(2^{13}A_1A_2(1+\sigma^2))$, $\varepsilon_1={\delta\gamma^2\varepsilon^2}/(2^{12} A_1^4A_2^3M^3(1+\sigma^2))$ and, for all $i\ge1$, all $m\ge0$ and all $k\ge 0$,
\begin{equation}
\begin{split}
\text{Bad}^{(-,i)}_m&=\left\{v\in\cT^{(i)}:\ h(v)=m,\ \frac{\widetilde{X}_v}{h(v)}<\frac{\sigma^2}{2}-\varepsilon_1\right\},\\
\text{Bad}^{(-,i)}_m(k)&=\left\{v\in\cT^{(i)}:\ h(v)=m+k,\ v\text{ has an ancestor in }   \text{Bad}^{(-,i)}_m \right\},\\
Z^{(-,i,m)}_k &=\left|\text{Bad}^{(-,i)}_m(k)\right|,
\end{split}
\end{equation}
as well as, using the lexicographical order,
\begin{equation}\label{pasdenom}
\begin{split}
\text{Bad}^{(+)}_m&=\left\{i\in\{1,\dots,\lfloor A_1\sqrt{n}\rfloor\}: \ \frac{\widetilde{X}_{v_{i,m}}}{h(v_{i,m})}>\frac{\sigma^2}{2}+\varepsilon\right\}.
\end{split}
\end{equation}
Write $m_j=\lfloor j\gamma \sqrt{n}\rfloor$ for $j\geq 0$. For a set $S$ of vertices at a same height $h_S$, we denote $Z(S)$ the number of descendants at generation $m_{j_S}$ of the vertices in $S$, where $j_S$ is the smallest index $j$ such that $m_{j}\ge h_S$. Let us define the events
\begin{align}
E_{n,3}&=\bigcap_{i=1}^{\lfloor A_1^{-1}\sqrt{n}\rfloor} \bigcap_{j=\lfloor\gamma^{-1} A_1\varepsilon^2\rfloor}^{\lceil \gamma^{-1} A_2\rceil} \left\{\cT^{(i)}\text{ is }\left(\varepsilon_1,m_j\right)\text{-bad}\right\}^c\\ \label{En4}
\begin{split}
E_{n,4}&= \bigcap_{i=1}^{\lfloor A_1\sqrt{n}\rfloor}  \bigcap_{j=\lfloor\gamma^{-1} A_1^{-1}\varepsilon^2\rfloor}^{\lceil \gamma^{-1} A_2\rceil} \left\{Z_{m_j}\left(\cT^{(i)}\right)>A_1\sqrt{n} \right\}\\
&\qquad \cup \left\{ \max_{k\in\{0,\dots,\lfloor 2\gamma \sqrt{n}\rfloor\}}\left\{ \left( Z^{(-,i,m_j)}_k-Z^{(-,i,m_j)}_0\right)\le \frac{\varepsilon}{8MA_2}\sqrt{n}\right\}\right\}
\end{split}\\ \label{En5}
E_{n,5}&=\bigcap_{j=\lceil A_1^{-1}\varepsilon^2 \sqrt{n}\rceil}^{\lfloor A_2\sqrt{n}\rfloor}\left\{\left|\text{Bad}^{(+)}_j\right|<\frac{\varepsilon\sqrt{n}}{8A_2}\right\}\cup\left\{ Z(\text{Bad}^{(+)}_j)>\varepsilon_1 A_1\sqrt{n}\right\}
\end{align}

Now, note that for $\lfloor \gamma^{-1} A_1^{-1}\varepsilon^2\rfloor \le j \le \lceil \gamma^{-1} A_2\rceil$,  and for all $0\le k \le \lfloor 2\gamma\sqrt{n}\rfloor$, if  $v$ is a vertex at height $m_j$ which is not  $\varepsilon_1$-bad  and $x$ is a vertex with $h(x)=k+m_j$ and $y$ such that $h(y)=m_j-k$, then we have that, for all $n\ge N_2(\bmu,\delta,\varepsilon)=N_1+16\varepsilon^{-2}$,
\begin{equation}\label{sugar1}
\begin{split}
\frac{\widetilde{X}_x}{h(x)}&\ge \frac{\widetilde{X}_{v}}{h(v)}\times \frac{m_j}{m_j+2\gamma\sqrt{n}}\\
&\ge \frac{\widetilde{X}_{v}}{h(v)}\times \left(1-\frac{8\gamma}{ A_1^{-1}\varepsilon^2  }\right)\\
&\ge \left(\frac{\sigma^2}{2}-\varepsilon_1\right)\times  \left(1-\frac{8\gamma}{ A_1\varepsilon^2  }\right)\\
&\ge \frac{\sigma^2}{2}-4\varepsilon^{-2}\sigma^2\gamma-\varepsilon_1\\
&\ge \frac{\sigma^2}{2}-\frac{\varepsilon}{4}-\varepsilon_1\\
&\ge \frac{\sigma^2}{2}-\frac{\varepsilon}{2},
\end{split}
\end{equation}
and
\begin{equation}\label{sugar2}
\begin{split}
\frac{\widetilde{X}_y}{h(y)}&\le \frac{\widetilde{X}_{v}}{h(v)}\times \frac{m_j}{m_j-2\gamma\sqrt{n}}\\
&\le \frac{\widetilde{X}_{v}}{h(v)}\times \left(1+\frac{8\gamma}{ A_1\varepsilon^2  }\right)\\
&\le \left(\frac{\sigma^2}{2}+\varepsilon_1\right)\times  \left(1+\frac{8\gamma}{ A_1\varepsilon^2  }\right)\\
&\le \frac{\sigma^2}{2}+4\varepsilon^{-2}\sigma^2\gamma+2\varepsilon_1\\
&\le \frac{\sigma^2}{2}+\frac{\varepsilon}{4}+\varepsilon_1\\
&\le \frac{\sigma^2}{2}+\frac{\varepsilon}{2}.
\end{split}
\end{equation}

Thus, by \eqref{sugar1} and \eqref{sugar2}, for $n\ge N_2$, if there is a $\varepsilon$-bad vertex at some given generation $k\in \{\lceil A_1^{-1}\varepsilon^2\sqrt{n}\rceil,\dots, \lfloor A_2\sqrt{n}\rfloor\}$, then there exists $j$ such that $m_j\le k\le m_{j+1}$, and either it has an ancestor at generation $m_j$ that is $\varepsilon_1$-bad or all of its descendants (if any) at generation $m_{j+1}$ are $\varepsilon_1$-bad.\\

Using the above, for $n\ge N_2$ and on the event $E_{n,1}\cap E_{n,3} \cap E_{n,4}\cap E_{n,5}$,  for all $1\le i\le \lfloor A_1\sqrt{n}\rfloor$, for all $\lceil A^{-1}_1\varepsilon^2\sqrt{n}\rceil \le m\le \lfloor A_2\sqrt{n}\rfloor$, we have that, if $w(\cT^{(i)})\le A_1\sqrt{n}$ then
\[
\left|\left\{    v\in\cT^{(i)}:\ h(v)=m, \ v \text{ is }\varepsilon\text{-bad}   \right\}\right|\le \varepsilon_1A_1\sqrt{n}+\frac{\varepsilon}{4MA_2}\sqrt{n}<\frac{\varepsilon}{2MA_2}\sqrt{n},
\]
thus, we obtain that
\begin{equation}
E_{n,1}\cap E_{n,3} \cap E_{n,4}\cap E_{n,5}\subset E_{n,1}\cap E_{n,2}.
\end{equation}
Hence, we have that, for $n\ge N_4$,
\[
\Pmu\left(  \left(E_{n,1}\cap E_{n,2}\right)^c  \right)\le \Pmu(E_{n,1}^c)+\Pmu(E_{n,1}\cap E_{n,3}^c)+\Pmu(E_{n,1}\cap E_{n,3}\cap E_{n,4}^c)+\Pmu(E_{n,5}^c),
\]
and, by \eqref{huel0}, in order to prove \eqref{huel3} and conclude the proof of the proposition, it is enough to prove that
\begin{align}\label{huel4}
\Pmu(E_{n,3}^c)&\le \frac{\delta}{4},\\ \label{huel5}
\Pmu( E_{n,3}\cap E_{n,4}^c)&\le \frac{\delta}{4},\\ \label{huel6}
\Pmu(E_{n,5}^c)&\le \frac{\delta}{4}.
\end{align}
By a union bound, Lemma \ref{lem:heightsurvivalproba} and Proposition \ref{prop:badspineT}, there exists $N_3=N_3(\bmu,\varepsilon,\varepsilon_1,\delta,\gamma,A_1,A_2)$ such that, for all $n\ge N_3$,
\begin{equation}
\begin{split}
\Pmu(E_{n,3}^c)&\le A_1\sqrt{n}\times \gamma^{-1}A_2\max_{A_1^{-1}\varepsilon^2\sqrt{n}/2\le m \le 2A_2\sqrt{n}}\Pmu\left( \mathcal{T}\text{ is }(\varepsilon_1,m)\text{-bad}   \right)\\
&\le A_1A_2\gamma^{-1} \sqrt{n} \times \Pmu(h(\cT)\ge A_1^{-1}\varepsilon^2\sqrt{n}/2)\times ( \delta c A_1^{-2}A_2^{-1}\gamma \varepsilon^2/24)\\
&\le A_1A_2\gamma^{-1} \sqrt{n} \frac{8}{c A_1^{-1}\varepsilon^2\sqrt{n}} \times ( \delta c A_1^{-2} A_2^{-1}\gamma \varepsilon^2/24)\\
&\le \frac{\delta}{4},
\end{split}
\end{equation}
where $c$ is the constant from Condition \ref{(II)}. This proves \eqref{huel4}.\\

Now, note that by  Lemma \ref{lem:deterministicLLNcuttopieces} and Condition \ref{(I)}, there exists $N_4=N_4(\bmu,\gamma,A_1,A_2)$ such that, for all $n\ge N_4$
\begin{equation}\label{beast}
\begin{split}
\max_{\lfloor \gamma^{-1}A_1^{-1}\varepsilon^2\rfloor\le j\le \lceil \gamma^{-1}A_2\rceil}\frac{\sigma_{m_j}^2+\ldots +\sigma_{m_j+\lfloor 2\gamma\sqrt{n}\rfloor-1}^2}{2\gamma \sqrt{n}}&\le 2\sigma^2.
\end{split}
\end{equation}
Observe that, for $\lfloor \gamma^{-1}A_1^{-1}\varepsilon^2\rfloor\le j\le \lceil \gamma^{-1}A_2\rceil$,  on $E_{n,3}\cap \{Z_{\lfloor j\gamma\sqrt{n}\rfloor}\le A_1\sqrt{n}\}$, we have that 
\[
Z_0^{(i,\lfloor j \gamma\sqrt{n}\rfloor)}\le \varepsilon_1A_1\sqrt{n},
\]
for all $1\le i \le A_1\sqrt{n}$. Hence, on $E_{n,3}\cap E_{n,4}^c$, for $\lfloor \gamma^{-1}A_1^{-1}\varepsilon^2\rfloor\le j\le \lceil \gamma^{-1}A_2\rceil$ and $1\le i \le A_1\sqrt{n}$,  we have that $(Z_k^{(i,\lfloor j \gamma\sqrt{n}\rfloor)})_{\ge0}$ is stochastically dominated by a $\bmu_j$-BPVE $(Z_k^j)_{k\ge0}$ with $Z_0^j=\lfloor \varepsilon_1 A_1\sqrt{n}\rfloor$ and where $\bmu_j=(\mu_{\lfloor j\gamma\sqrt{n}\rfloor +k})_{k\ge 0}$. Thus, we obtain that, using Lemma \ref{lem:offspringvarianceandDoob} and \eqref{beast},  for all $n\ge N_4$
\begin{equation}\label{guillaume109}
\begin{split}
\Pmu(E_{n,3}\cap& E_{n,4}^c)
\\
&\le A_1\sqrt{n}\times 2\gamma^{-1}A_2\\
&\quad \times \hspace{-1mm}\max_{\lfloor \gamma^{-1}A_1\varepsilon^2\rfloor\le j\le \lceil \gamma^{-1}A_2\rceil}\hspace{-1mm} \mathrm{P}_{\bmu_j} \big(\max_{k\leq \lfloor2\gamma \sqrt{n}\rfloor}\hspace{-1mm}\vert Z_k^{(1,\lfloor j \gamma\sqrt{n}\rfloor)}-Z_0^{(1,\lfloor j \gamma\sqrt{n}\rfloor)}\vert >\frac{\varepsilon}{4MA_2}\sqrt{n}\big)\\
&\leq \varepsilon_1\times \frac{2^8 A_1^2A_2^3M^2 \gamma^{-2}}  {\varepsilon^2}\hspace{-2mm} \max_{\lfloor \gamma^{-1}A_1\varepsilon^2\rfloor\le j\le \lceil \gamma^{-1}A_2\rceil}\hspace{-2mm}\frac{\sigma_{\lfloor j\gamma \sqrt{n}\rfloor}^2+\ldots +\sigma_{\lfloor j\gamma \sqrt{n}\rfloor+\lfloor 2\gamma\sqrt{n}\rfloor-1}^2}{2\gamma \sqrt{n}}\\
&\le 
\varepsilon_1\times {2^{9} A_1^2A_2^3M^2 \gamma^{-2}\sigma^2\varepsilon^{-2}} \le \frac{\delta}{4}.
\end{split}
\end{equation}
This proves \eqref{huel5}.\\

Finally, to prove \eqref{huel6}, we explore, for a generation $j\ge0$, the offspring of the vertices $\{v_{i,j}: 1\le i\le A_1\sqrt{n}\}$, generation by generation, so that we can define the stopping time
\[
\tau_B=\min\left\{j\in \{\lceil A_1^{-1}\varepsilon^2\sqrt{n}\rceil,\dots,\lfloor A_2 \sqrt{n}\rfloor\}:  \left|\text{Bad}^{(+)}_j\right|\ge \frac{\varepsilon\sqrt{n}}{8A_2}\right\}.
\]
We let $k_B$ be the smalled integer $i$ such that $\tau_b+i=m_j$ for some $j$. Note that $k_B$ is $\tau_B$-measurable.
By applying the strong Markov property at time $\tau_B$, proceeding as in \eqref{guillaume109} by letting $Z_k$ be a BPVE with $Z_0=\lfloor \varepsilon \sqrt{n}/(8A_2)\rfloor$ and by using  again Lemma \ref{lem:offspringvarianceandDoob}  together with \eqref{beast}, we have that, for $n\ge  N_4$,
\begin{align*}
\Pmu(E_{n,5}^c)&\le \Emu\left[       \mathrm{P}_{\bmu_{\tau_B}}\left(          Z_{k_B}<\varepsilon_1 A_1\sqrt{n}              \right)    \right] \\
&\le \Emu\left[     \frac{2^{12}A_2\varepsilon^{-1}}{ 8\sqrt{n}  } \left(      \sigma^2_{\tau_B}+\dots+\sigma^2_{\tau_B+k_B-1}     \right)    \right] \\
&\le           2^{10}A_2\varepsilon^{-1}\gamma           \max_{\lfloor \gamma^{-1}A_1\varepsilon^2\rfloor\le j\le \lceil \gamma^{-1}A_2\rceil}\frac{\sigma_{\lfloor j\gamma \sqrt{n}\rfloor}^2+\ldots +\sigma_{\lfloor j\gamma \sqrt{n}\rfloor+\lfloor 2\gamma\sqrt{n}\rfloor-1}^2}{2\gamma \sqrt{n}}               \\
&\le  2^{11}A_2\varepsilon^{-1}\sigma^2\gamma  \\
&\le \frac{\delta}{4}.
\end{align*}
This concludes the proof.
\end{proof}

\subsection{Proof of Theorem \ref{thm:jointcvLukaHeight}}\label{sect:prooflukaheight}

We are ready for the proof of Theorem \ref{thm:jointcvLukaHeight}, which states that the height process is proportional to the {\L}ukasiewicz path reflected above its running minimum. In the next section, we will prove Theorem~\ref{th:main} by extracting the largest tree from the explored forest. We start by establishing a relation between the quantity $\widetilde{X}_x$ from the previous section and the value of the {\L}ukasiewicz path on the explored forest $\mathcal{F}_n$. We then give an outline of the proof.\\

Consider the forest $\mathcal{F}_n$ and recall that it consists of a sequence of finite trees $\cT^{(i)}$, $i\ge1$.  For  a vertex $x\in \cF_n$,  denote again $\mathcal{L}(x)$ the depth-first label of $x$, as explained in Section \ref{sectorder}. From the definition \eqref{def:Luka}, one can observe that, if $x\in\mathcal{T}^{(1)}$, then  $\widetilde{X}_{x}$ is precisely the value of the {\L}ukasiewicz path  right before revealing the offspring of $x$, i.e.~$X_{\mathcal{L}(x)}$ (this is the number of red vertices in Figure \ref{fig:spine}). Note however that every time the {\L}ukasiewicz path has fully explored a tree, this creates an increment $-1$  (corresponding in fact to the difference between the number of edges
and the number of vertices of the tree). More precisely, if we denote $n_i$ the depth-first label of the root of the tree $\cT^{(i)}$ for $i\ge1$, we have that
\begin{equation}\label{massagegun}
\min_{n_{i}\le k<n_{i+1}} \left(X_k-X_{n_i}\right)\ge 0\text{, and } X_{n_{i+1}}-X_{n_i}=-1.
\end{equation}
From \eqref{massagegun} and the definition of the {\L}ukasiewicz path, we have that, for a vertex $x\in\cF_n$,
\begin{equation}\label{pipi}
\widetilde{X}_x=X_{\mathcal{L}(x)}-I_{\mathcal{L}(x)},
\end{equation}
where we recall that $I_n=\min\{X_k: 0\le k\le n\}$. \\

{Now, we explain the proof strategy, which is similar to that of~\cite[Theorem2.3.1]{DuquesneLeGall}. From Proposition~\ref{prop:fewbadvertices}, we know that for every $\varepsilon>0$, with high probability, every interval of length $\varepsilon n$ in $[0,n]$ contains at least one integer $k$ such that 
\begin{equation}\label{eq:explaincvheight}
\vert (X_{k}-I_{k})/H_{k}-\sigma^2/2\vert \leq \varepsilon.
\end{equation} 
Let $k_1, \ldots, k_{K+1}$ be a sequence of such integers, for some  $K\geq 1$ (we drop the dependency on $n$ to lighten the notation) such that $k_{i+1}-k_i\leq \varepsilon n$ for all $i\leq K$. Using Corollary~\ref{cor:Luka} and the fact that the scaling limit of $X-I$ is continuous (and thus uniformly continuous on $[0,1]$), we can bound $\Delta:=\max_{1\leq i\leq K} \sup_{k_i\leq j<\ell \leq k_{i+1}}\vert (X_j-I_j)\, -\, (X_{\ell}-I_{\ell}) \vert $ by some constant $\gamma>0$ (which can be chosen arbitrarily close to $0$ as $\varepsilon \rightarrow 0$) with arbitrarily large probability. Therefore, it remains to control the variations of $H$ on the intervals $[k_i,k_{i+1}]$. 
\\
Controlling the negative variations of $H$ (i.e.~$\min_{j\in [k_i,k_{i+1}]}H_j-H_{k_i}$) does not seem obvious at all. In fact, the depth-first walk could in principle have large downwards jumps, when a vertex and a large number of its nearest ancesters are the last explored among their siblings (in Figures~\ref{fig:spine} and~\ref{sommetsv}, this would correspond to many consecutive blue vertices having no red neighbours). Fortunately, it is enough to control only the positive variations of $H$: Indeed, combining the upper bound on $\Delta$ and \eqref{eq:explaincvheight}, via the triangle inequality, allows us to control $\max_{i\leq K}\vert H_{k_{i+1}}-H_{k_i}\vert$ appropriately. Hence, negative and positive variations of $H$ in $[k_i,k_{i+1}]$ must be of the same order.
\\
To handle these variations, we bound the height of the largest subtree built during $\varepsilon n$ steps of the exploration, using the estimates of Lemma~\ref{lem:heightsurvivalproba}. This is the purpose of Lemma~\ref{KingCharles} below. We first establish this lemma, and then proceed to the proof of Theorem~\ref{thm:jointcvLukaHeight}.}

\begin{lemma}\label{KingCharles}
Let $\bmu$ be a strictly critical environment satisfying Conditions \ref{(I)} to \ref{(V)}.
For all $\gamma>0$ and all $\delta>0$, there exist $\varepsilon=\varepsilon(\gamma,\delta)>0$ and $N=N(\bmu,\delta)$ such that, for all $n\ge N$, we have that
\begin{equation}
\mathrm{P}_{\bmu}\left(\max_{0\le k\le n} \max_{0\le i\le \varepsilon n}\left( H_{k+i}-H_k\right)>\gamma \sqrt{n}\right)<\delta.
\end{equation}
\end{lemma}

\begin{proof}
We start by bounding the maximal height of the first $n$ vertices explored. By Proposition \ref{prop:firstbounds}, there exist $A_2=A_2(\delta)>0$ and $N_1=N_1(\bmu,\delta)$ such that, for all $n\ge N_1$,
\begin{equation}\label{oats}
\Pmu\left(\max_{0\le k\le n} H_k>A_2\sqrt{n}\right)= \Pmu\big(h(\cF_n)> A_2\sqrt{n}\big)|< \frac{\delta}{4}.
\end{equation}
Now, we will control the increments of $(H_k)_{0\le k\le n}$ on vertices with height smaller than $A_2\sqrt{n}$.
Fix $\gamma>0$ and $\delta>0$. For all $n\ge0$, let us define the sequence of stopping times $(\tau_k^{(n)})_{k\ge0}$ by $\tau_0=0$ and, for $k\ge1$, 
\begin{equation}\label{eqn:taukdef}
\tau_k^{(n)}:=\inf\{t> \tau_{k-1}, \, H_t\geq \min_{\tau_{k-1}\leq s\leq t}H_{s}+\gamma\sqrt{n}\}.
\end{equation}
{Recall that, for the vertex with depth-first label $n$ and for a height $m\in\{0,\dots,H_n-1\}$, we have, from \eqref{deftilde}, that $\zeta_{m}(n)$ is the number of vertices at height $m+1$ which are strictly on the right of the spine from the vertex $n$ to the root of its tree, and whose parent is on this spine. }

{
For $n\ge0$, let us define the $\sigma$-algebra
\[
\mathcal{G}_n=\sigma\left(X_0, \ldots, X_n\right).
\]
In particular, note that $H_n$ and  $\{\zeta_m(n), m\in\{0,\dots,H_n-1\}$ are $\mathcal{G}_n$-measurable. Finally, we will use the notation $\mathcal{G}_k^{(n)}$ for $\mathcal{G}_{\tau_k^{(n)}}$.
}
Our next goal is to prove that for all $\delta>0$ there exists a constant $\eta=\eta(\gamma,\delta)>0$ and there exists $N_2=N_2(\bmu,\gamma,\delta)$ such that, for all $n\ge N_2$ and for all $k\geq 0$,
\begin{equation}\label{eqn:LeGalltauks}
\1_{\left\{H_{\tau_k^{(n)}}\le A_2\sqrt{n}\right\}} \cdot \mathrm{P}_{\bmu}\left(\left.\tau_{k+1}^{(n)}-\tau_k^{(n)} < \eta n\right|\mathcal{G}_k^{(n)}\right)\leq \frac{\delta}{2}\,\, \text{a.s.}
\end{equation}
We will first prove the above, and then show how this implies the conclusion.

Fix integers $n\ge0$, $k\geq 0$ and suppose that we know $\mathcal G_k^{(n)}$, 
i.e.\ we have explored the forest until the node with depth-first label $\tau_k^{(n)}$.
Recall the definition \eqref{deftilde} of $\widetilde{X}_n$ and let us define the shorthand notation
\begin{equation}
M=1+\widetilde{X}_{\tau_k^{(n)}}=1+\zeta_0(\tau_k^{(n)})+\dots+\zeta_{H_{\tau_k^{(n)}}-1}(\tau_k^{(n)}),
\end{equation}
which counts $\tau_k^{(n)}$ plus the total number of vertices attached to the spine from the root of the tree of $\tau_k^{(n)}$ to $\tau_k^{(n)}$, and strictly on the right of this spine, see Figure \ref{sommetsv}. Let us denote these vertices $v_1,\dots,v_M$ in depth-first order. There exists an index $i_{k,n}$ such that $\tau_k^{(n)}$ belongs to the tree $\mathcal{T}^{(i_{k,n})}$ in $\mathcal{F}$. We will denote, for $j\ge1$, $v_{M+j}$ the root of the tree $\mathcal{T}^{(i_{k,n}+j)}$ (which has not yet been seen by the exploration).

\tikzset{every picture/.style={line width=0.75pt}} 
\begin{center}
\begin{figure}
\begin{tikzpicture}[x=0.75pt,y=0.75pt,yscale=-1.5,xscale=1.5]

\draw    (235.33,268.33) -- (254.67,319) ;
\draw    (240,227) -- (235.33,268.33) ;
\draw    (224,193.67) -- (240,227) ;
\draw    (229.25,164.5) -- (224,193.67) ;
\draw  [dash pattern={on 0.84pt off 2.51pt}]  (237.25,307) -- (254.67,319) ;
\draw  [dash pattern={on 0.84pt off 2.51pt}]  (228.75,243) -- (235.33,268.33) ;
\draw  [dash pattern={on 0.84pt off 2.51pt}]  (216.25,255) -- (235.33,268.33) ;
\draw  [dash pattern={on 0.84pt off 2.51pt}]  (207.75,169.67) -- (224,193.67) ;
\draw  [dash pattern={on 0.84pt off 2.51pt}]  (199.25,184.17) -- (224,193.67) ;
\draw  [dash pattern={on 0.84pt off 2.51pt}]  (220.75,166.17) -- (224,193.67) ;
\draw    (274.75,290.5) -- (254.67,319) ;
\draw    (258.75,290.5) -- (254.67,319) ;
\draw    (248.75,194.5) -- (240,227) ;
\draw    (272.25,195.5) -- (240,227) ;
\draw    (266.75,228) -- (235.33,268.33) ;
\draw    (253.25,169) -- (224,193.67) ;
\draw  [dash pattern={on 0.84pt off 2.51pt}]  (100.25,319.5) .. controls (74.25,294.5) and (134.25,159.5) .. (220.75,166.17) ;
\draw  [color={rgb, 255:red, 208; green, 200; blue, 27 }  ,draw opacity=1 ][fill={rgb, 255:red, 208; green, 200; blue, 27 }  ,fill opacity=1 ] (250.88,319) .. controls (250.88,317.16) and (252.57,315.67) .. (254.67,315.67) .. controls (256.76,315.67) and (258.46,317.16) .. (258.46,319) .. controls (258.46,320.84) and (256.76,322.33) .. (254.67,322.33) .. controls (252.57,322.33) and (250.88,320.84) .. (250.88,319) -- cycle ;
\draw  [color={rgb, 255:red, 74; green, 144; blue, 226 }  ,draw opacity=1 ][fill={rgb, 255:red, 74; green, 144; blue, 226 }  ,fill opacity=1 ] (231.54,268.33) .. controls (231.54,266.49) and (233.24,265) .. (235.33,265) .. controls (237.43,265) and (239.13,266.49) .. (239.13,268.33) .. controls (239.13,270.17) and (237.43,271.67) .. (235.33,271.67) .. controls (233.24,271.67) and (231.54,270.17) .. (231.54,268.33) -- cycle ;
\draw  [color={rgb, 255:red, 74; green, 144; blue, 226 }  ,draw opacity=1 ][fill={rgb, 255:red, 74; green, 144; blue, 226 }  ,fill opacity=1 ] (236.21,227) .. controls (236.21,225.16) and (237.91,223.67) .. (240,223.67) .. controls (242.09,223.67) and (243.79,225.16) .. (243.79,227) .. controls (243.79,228.84) and (242.09,230.33) .. (240,230.33) .. controls (237.91,230.33) and (236.21,228.84) .. (236.21,227) -- cycle ;
\draw  [color={rgb, 255:red, 74; green, 144; blue, 226 }  ,draw opacity=1 ][fill={rgb, 255:red, 74; green, 144; blue, 226 }  ,fill opacity=1 ] (220.21,193.67) .. controls (220.21,191.83) and (221.91,190.33) .. (224,190.33) .. controls (226.09,190.33) and (227.79,191.83) .. (227.79,193.67) .. controls (227.79,195.51) and (226.09,197) .. (224,197) .. controls (221.91,197) and (220.21,195.51) .. (220.21,193.67) -- cycle ;
\draw  [color={rgb, 255:red, 208; green, 2; blue, 27 }  ,draw opacity=1 ][fill={rgb, 255:red, 208; green, 2; blue, 27 }  ,fill opacity=1 ] (225.46,167.83) .. controls (225.46,165.99) and (227.16,164.5) .. (229.25,164.5) .. controls (231.34,164.5) and (233.04,165.99) .. (233.04,167.83) .. controls (233.04,169.67) and (231.34,171.17) .. (229.25,171.17) .. controls (227.16,171.17) and (225.46,169.67) .. (225.46,167.83) -- cycle ;
\draw  [color={rgb, 255:red, 208; green, 2; blue, 27 }  ,draw opacity=1 ][fill={rgb, 255:red, 208; green, 2; blue, 27 }  ,fill opacity=1 ] (270.96,290.5) .. controls (270.96,288.66) and (272.66,287.17) .. (274.75,287.17) .. controls (276.84,287.17) and (278.54,288.66) .. (278.54,290.5) .. controls (278.54,292.34) and (276.84,293.83) .. (274.75,293.83) .. controls (272.66,293.83) and (270.96,292.34) .. (270.96,290.5) -- cycle ;
\draw  [color={rgb, 255:red, 208; green, 2; blue, 27 }  ,draw opacity=1 ][fill={rgb, 255:red, 208; green, 2; blue, 27 }  ,fill opacity=1 ] (254.96,290.5) .. controls (254.96,288.66) and (256.66,287.17) .. (258.75,287.17) .. controls (260.84,287.17) and (262.54,288.66) .. (262.54,290.5) .. controls (262.54,292.34) and (260.84,293.83) .. (258.75,293.83) .. controls (256.66,293.83) and (254.96,292.34) .. (254.96,290.5) -- cycle ;
\draw  [color={rgb, 255:red, 208; green, 2; blue, 27 }  ,draw opacity=1 ][fill={rgb, 255:red, 208; green, 2; blue, 27 }  ,fill opacity=1 ] (262.96,228) .. controls (262.96,226.16) and (264.66,224.67) .. (266.75,224.67) .. controls (268.84,224.67) and (270.54,226.16) .. (270.54,228) .. controls (270.54,229.84) and (268.84,231.33) .. (266.75,231.33) .. controls (264.66,231.33) and (262.96,229.84) .. (262.96,228) -- cycle ;
\draw  [color={rgb, 255:red, 208; green, 2; blue, 27 }  ,draw opacity=1 ][fill={rgb, 255:red, 208; green, 2; blue, 27 }  ,fill opacity=1 ] (244.96,194.5) .. controls (244.96,192.66) and (246.66,191.17) .. (248.75,191.17) .. controls (250.84,191.17) and (252.54,192.66) .. (252.54,194.5) .. controls (252.54,196.34) and (250.84,197.83) .. (248.75,197.83) .. controls (246.66,197.83) and (244.96,196.34) .. (244.96,194.5) -- cycle ;
\draw  [color={rgb, 255:red, 208; green, 2; blue, 27 }  ,draw opacity=1 ][fill={rgb, 255:red, 208; green, 2; blue, 27 }  ,fill opacity=1 ] (249.46,169) .. controls (249.46,167.16) and (251.16,165.67) .. (253.25,165.67) .. controls (255.34,165.67) and (257.04,167.16) .. (257.04,169) .. controls (257.04,170.84) and (255.34,172.33) .. (253.25,172.33) .. controls (251.16,172.33) and (249.46,170.84) .. (249.46,169) -- cycle ;
\draw  [color={rgb, 255:red, 208; green, 2; blue, 27 }  ,draw opacity=1 ][fill={rgb, 255:red, 208; green, 2; blue, 27 }  ,fill opacity=1 ] (268.46,195.5) .. controls (268.46,193.66) and (270.16,192.17) .. (272.25,192.17) .. controls (274.34,192.17) and (276.04,193.66) .. (276.04,195.5) .. controls (276.04,197.34) and (274.34,198.83) .. (272.25,198.83) .. controls (270.16,198.83) and (268.46,197.34) .. (268.46,195.5) -- cycle ;
\draw  [color={rgb, 255:red, 208; green, 2; blue, 27 }  ,draw opacity=1 ][fill={rgb, 255:red, 208; green, 2; blue, 27 }  ,fill opacity=1 ] (307.63,319.5) .. controls (307.63,317.66) and (309.32,316.17) .. (311.42,316.17) .. controls (313.51,316.17) and (315.21,317.66) .. (315.21,319.5) .. controls (315.21,321.34) and (313.51,322.83) .. (311.42,322.83) .. controls (309.32,322.83) and (307.63,321.34) .. (307.63,319.5) -- cycle ;
\draw  [color={rgb, 255:red, 208; green, 2; blue, 27 }  ,draw opacity=1 ][fill={rgb, 255:red, 208; green, 2; blue, 27 }  ,fill opacity=1 ] (326.13,319.5) .. controls (326.13,317.66) and (327.82,316.17) .. (329.92,316.17) .. controls (332.01,316.17) and (333.71,317.66) .. (333.71,319.5) .. controls (333.71,321.34) and (332.01,322.83) .. (329.92,322.83) .. controls (327.82,322.83) and (326.13,321.34) .. (326.13,319.5) -- cycle ;
\draw  [color={rgb, 255:red, 208; green, 2; blue, 27 }  ,draw opacity=1 ][fill={rgb, 255:red, 208; green, 2; blue, 27 }  ,fill opacity=1 ] (288.13,319.5) .. controls (288.13,317.66) and (289.82,316.17) .. (291.92,316.17) .. controls (294.01,316.17) and (295.71,317.66) .. (295.71,319.5) .. controls (295.71,321.34) and (294.01,322.83) .. (291.92,322.83) .. controls (289.82,322.83) and (288.13,321.34) .. (288.13,319.5) -- cycle ;
\draw  [dash pattern={on 0.84pt off 2.51pt}]  (349.5,320) -- (372.75,320) ;
\draw  [color={rgb, 255:red, 0; green, 0; blue, 0 }  ,draw opacity=1 ][fill={rgb, 255:red, 0; green, 0; blue, 0 }  ,fill opacity=1 ] (106.21,318.67) .. controls (106.21,316.83) and (107.91,315.33) .. (110,315.33) .. controls (112.09,315.33) and (113.79,316.83) .. (113.79,318.67) .. controls (113.79,320.51) and (112.09,322) .. (110,322) .. controls (107.91,322) and (106.21,320.51) .. (106.21,318.67) -- cycle ;
\draw  [color={rgb, 255:red, 0; green, 0; blue, 0 }  ,draw opacity=1 ][fill={rgb, 255:red, 0; green, 0; blue, 0 }  ,fill opacity=1 ] (125.71,318.67) .. controls (125.71,316.83) and (127.41,315.33) .. (129.5,315.33) .. controls (131.59,315.33) and (133.29,316.83) .. (133.29,318.67) .. controls (133.29,320.51) and (131.59,322) .. (129.5,322) .. controls (127.41,322) and (125.71,320.51) .. (125.71,318.67) -- cycle ;
\draw  [color={rgb, 255:red, 0; green, 0; blue, 0 }  ,draw opacity=1 ][fill={rgb, 255:red, 0; green, 0; blue, 0 }  ,fill opacity=1 ] (227.21,319.17) .. controls (227.21,317.33) and (228.91,315.83) .. (231,315.83) .. controls (233.09,315.83) and (234.79,317.33) .. (234.79,319.17) .. controls (234.79,321.01) and (233.09,322.5) .. (231,322.5) .. controls (228.91,322.5) and (227.21,321.01) .. (227.21,319.17) -- cycle ;
\draw  [color={rgb, 255:red, 0; green, 0; blue, 0 }  ,draw opacity=1 ][fill={rgb, 255:red, 0; green, 0; blue, 0 }  ,fill opacity=1 ] (207.71,319.17) .. controls (207.71,317.33) and (209.41,315.83) .. (211.5,315.83) .. controls (213.59,315.83) and (215.29,317.33) .. (215.29,319.17) .. controls (215.29,321.01) and (213.59,322.5) .. (211.5,322.5) .. controls (209.41,322.5) and (207.71,321.01) .. (207.71,319.17) -- cycle ;
\draw  [dash pattern={on 0.84pt off 2.51pt}]  (145,318.5) -- (198.75,319) ;
\draw  [color={rgb, 255:red, 74; green, 144; blue, 226 }  ,draw opacity=1 ][fill={rgb, 255:red, 74; green, 144; blue, 226 }  ,fill opacity=1 ] (239.21,288.67) .. controls (239.21,286.83) and (240.91,285.33) .. (243,285.33) .. controls (245.09,285.33) and (246.79,286.83) .. (246.79,288.67) .. controls (246.79,290.51) and (245.09,292) .. (243,292) .. controls (240.91,292) and (239.21,290.51) .. (239.21,288.67) -- cycle ;

\draw (199.5,149.4) node [anchor=north west][inner sep=0.75pt]  [font=\small]  {$\tau _{k}^{( n)} =v_{1}$};
\draw (251.5,155.9) node [anchor=north west][inner sep=0.75pt]  [font=\small]  {$v_{2}$};
\draw (248.04,182.4) node [anchor=north west][inner sep=0.75pt]  [font=\small]  {$v_{3}$};
\draw (275,185.4) node [anchor=north west][inner sep=0.75pt]  [font=\small]  {$v_{4}$};
\draw (272,219.4) node [anchor=north west][inner sep=0.75pt]  [font=\small]  {$v_{5}$};
\draw (251,275.9) node [anchor=north west][inner sep=0.75pt]  [font=\small]  {$v_{6}$};
\draw (277.5,280.9) node [anchor=north west][inner sep=0.75pt]  [font=\small]  {$v_{7}$};
\draw (285,306.4) node [anchor=north west][inner sep=0.75pt]  [font=\small]  {$v_{8}$};
\draw (307,306.4) node [anchor=north west][inner sep=0.75pt]  [font=\small]  {$v_{9}$};
\draw (322.5,306.4) node [anchor=north west][inner sep=0.75pt]  [font=\small]  {$v_{10}$};
\draw (98.5,325.9) node [anchor=north west][inner sep=0.75pt]  [font=\footnotesize]  {$\mathcal{T}^{( 1)}$};
\draw (123.5,326.4) node [anchor=north west][inner sep=0.75pt]  [font=\footnotesize]  {$\mathcal{T}^{( 2)}$};
\draw (243.5,325.4) node [anchor=north west][inner sep=0.75pt]  [font=\footnotesize]  {$\mathcal{T}^{( i_{k,n})}$};

\end{tikzpicture}
\caption{Black vertices are roots of fully explored trees. The yellow vertex is the root of the tree being explored.  
{Blue vertices are on the spine of the vertex with label $n$. Red vertices are parents of subtrees ($v_1$ to $v_7$) and trees (from $v_8$ onwards) to be explored after the vertex $n$.}
}\label{sommetsv}
\end{figure}
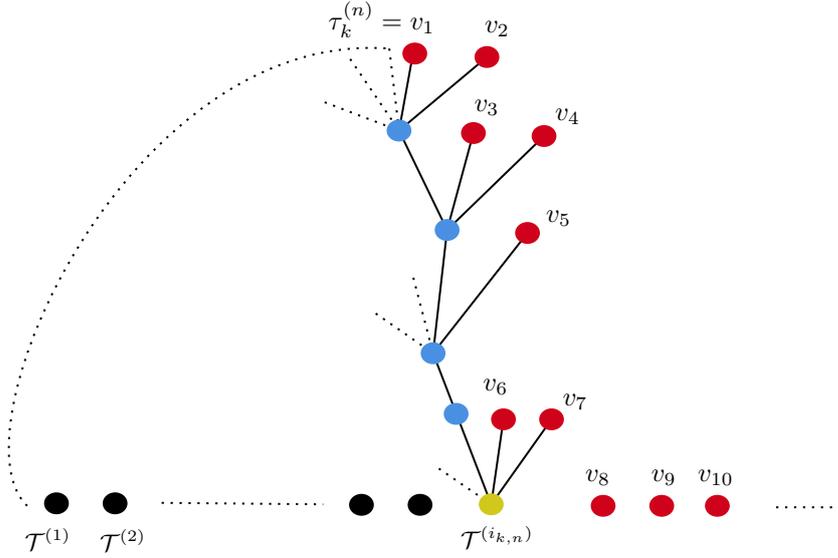
\end{center}

The offspring of these nodes $v_j$ for $j\ge1$ have not yet been explored. Hence, given $\mathcal{G}_k^{(n)}$, we can attach to each node $v_j$ a tree $T_j$ distributed like a BPVE in environment $\bmu_{h(v_j)}:=(\mu_k)_{k\geq h(v_j)}$, such that the trees $T_j$ are independent.\\
Note that the sequence $(H_{v_j})_{j\ge1}$ is non-increasing and that $\max_{j\geq 1}H_{v_j}\leq H_{\tau_k^{(n)}}$ (because the exploration is depth-first), hence the node $\tau_{k+1}^{(n)}$ belongs to the first tree $T_j$, $j\ge1$ whose height is at least $\gamma\sqrt{n}$.
Let us define the constant
\begin{equation}\label{defkappa}
\kappa=\frac{\delta\gamma c}{2^3},
\end{equation}
where $c$ is the constant from Condition \ref{(II)}.\\
Using the previous observations,  we have that conditionally on $\mathcal{G}_k^{(n)}$, for all $\eta>0$,
\begin{equation}
 \left\{\max_{1\leq j\leq \lfloor \kappa \sqrt{n}\rfloor}h(T_j)< \gamma\sqrt{n}\right\}\cap \left\{\sum_{j=1}^{\lfloor \kappa\sqrt{n}\rfloor} \left|T_j\right|\ge\eta n\right\}\subset\left\{\tau_{k+1}^{(n)}-\tau_k^{(n)}\ge\eta n\right\}.
\end{equation}

Therefore, in order to prove \eqref{eqn:LeGalltauks}, it is enough to show that there exist $\eta=\eta(\gamma,\delta)>0$ and  $N_2=N_2(\bmu,\delta)$ such that for all $n\ge N_2$ and for all $k\ge0$, the following holds almost surely:
\begin{equation}\label{spartanbeast}
\begin{split}
\1_{\left\{H_{\tau_k^{(n)}}\le A_2\sqrt{n}\right\}} \cdot\Pmu\left(\left. \max_{1\leq j\leq \lfloor \kappa \sqrt{n}\rfloor}h(T_j)\geq \gamma\sqrt{n}\,\right|\mathcal{G}_k^{(n)}\right)\leq \frac{\delta}{4},\\
\1_{\left\{H_{\tau_k^{(n)}}\le A_2\sqrt{n}\right\}} \cdot\Pmu\left(\left. \sum_{j=1}^{\lfloor \kappa\sqrt{n}\rfloor} \left|T_j\,\right| < \eta n\right|\mathcal{G}_k^{(n)}\right)\leq \frac{\delta}{4}.
\end{split}
\end{equation}

Let us prove the first inequality above. For an integer $0\le i \le A_2\sqrt{n}$, we denote $\bmu_{i}:=(\mu_k)_{k\geq i}$, and let us write simply $\cT$ for a $\bmu_i$-BPVE. By \eqref{augustoislate} and \eqref{eqn:heightsurvivalexactformula}, we have that
\begin{equation}\label{spartan}
\mathrm{P}_{\bmu_i}\left( h\left(\mathcal{T}\right)\ge \gamma \sqrt{n}\right)\le \frac{1}{\sum_{m=i}^{i+\lfloor\gamma\sqrt{n}\rfloor -1} \mu_m(\{0\})}. 
\end{equation}
By Condition \ref{(II)}, there exists $N_3=N_3(\bmu,A_2,\gamma)$ such that, for all $n\ge N_3$,
\[
\min_{0\le j\le 2A_2\gamma^{-1}+2} \frac{\sum_{m=j\lceil\gamma\sqrt{n}/2\rceil }^{(j+1)\lceil\gamma\sqrt{n}/2\rceil -1} \mu_m(\{0\})}{\sqrt{n}}\ge \frac{\gamma c}{4},
\]
hence, by \eqref{spartan}, for all $n\ge N_3$,
\begin{equation}
\max_{0\le i \le A_2\sqrt{n}}\mathrm{P}_{\bmu_i}\left( h\left(\mathcal{T}\right)\ge \gamma \sqrt{n}\right)\le \frac{4}{\gamma c \sqrt{n}}. 
\end{equation}
Therefore, using the fact that if $H_{\tau_k^{(n)}}\le A_2\sqrt{n}$ then $h(v_j)\le A_2\sqrt{n}$ for all $j\ge 1$, applying a union bound and using \eqref{defkappa}, we obtain that, for $n\ge N_3$ and for all $k\ge0$,
\begin{equation}\label{oats2}
\1_{\left\{H_{\tau_k^{(n)}}\le A_2\sqrt{n}\right\}} \cdot\Pmu\left(\left. \max_{1\leq j\leq \lfloor \kappa \sqrt{n}\rfloor}h(T_j)\geq \gamma\sqrt{n}\right|\mathcal{G}_k^{(n)}\right)\leq \frac{4\kappa}{\gamma c}\le \frac{\delta}{4}.
\end{equation}
This proves the first inequality of \eqref{spartanbeast} and we now turn to the proof of the second inequality. Note that \eqref{oats} and \eqref{oats2} do not depend on $\eta$. To prove this second inequality, define the constants
\begin{equation}\label{defcsts}
\begin{split}
\alpha &= \frac{c^2\kappa}{2^{15}\sigma^6 \ln(4/\delta)},\ 
\beta = \frac{\alpha c}{2^{10}\cdot\sigma^2},\
\eta = \frac{\alpha c }{2^6}=\frac{c^4\delta\gamma}{2^{24}\sigma^6 \ln(4/\delta)}.
\end{split}
\end{equation}
For $0\le i \le A_2\sqrt{n}$, by \eqref{augustoislate} and \eqref{eqn:heightsurvivalexactformula}, we have that
\begin{equation}\label{spartan2}
\frac{1}{1+\sum_{m=i}^{i+\lfloor\alpha\sqrt{n}\rfloor +1} \sigma^2_m}\le \mathrm{P}_{\bmu_i}\left( h\left(\mathcal{T}\right)\ge \alpha \sqrt{n}\right)\le \frac{1}{\sum_{m=i}^{i+\lfloor\alpha\sqrt{n}\rfloor -1} \mu_m(\{0\})}. 
\end{equation}
By Condition \ref{(I)} and Lemma \ref{lem:deterministicLLNcuttopieces} for the lower bound, and Condition \ref{(II)} for the upper bound, there exists $N_4=N_4(\bmu,A_2,\alpha)$ such that, for all $n\ge N_4$,
\begin{equation}\label{pop0}\begin{split}
\max_{0\le j\le A_2\alpha^{-1}/2+1} \frac{\sum_{m=j\lceil2\alpha\sqrt{n}\rceil }^{(j+1)\lceil2\alpha\sqrt{n}\rceil}\sigma^2_m}{\sqrt{n}}&\le 4\alpha \sigma^2,\\
\min_{0\le j\le 2A_2\alpha^{-1}+2} \frac{\sum_{m=j\lceil\alpha \sqrt{n}/2\rceil }^{(j+1)\lceil\alpha\sqrt{n}/2\rceil -1} \mu_m(\{0\})}{\sqrt{n}}&\ge \frac{\alpha c}{4}
\end{split}
\end{equation}
hence, by \eqref{spartan2}, for all $n\ge N_4$,
\begin{equation}\label{pop1}
\frac{1}{8\alpha \sigma^2 \sqrt{n}}\le \min_{0\le i \le A_2\sqrt{n}}\mathrm{P}_{\bmu_i}\left(h\left(\mathcal{T}\right)\ge \alpha \sqrt{n}\right)\le 
\max_{0\le i \le A_2\sqrt{n}}\mathrm{P}_{\bmu_i}\left(h\left(\mathcal{T}\right)\right)\le 
\frac{4}{\alpha c \sqrt{n}}, 
\end{equation}
where we assumed w.l.o.g.~that $N_4\ge 1/(4\alpha\sigma^2)^2$. 
To lighten the following equations, denote $Z_{\lfloor \alpha\sqrt{n}\rfloor}:=Z_{\lfloor \alpha\sqrt{n}\rfloor}(\cT)$ the number of vertices at height $\lfloor \alpha\sqrt{n}\rfloor$ in $\cT$. 
By Paley-Zygmund's inequality, we have that for all $0\le i \le A_2\sqrt{n}$,
\begin{equation}\label{pop2}
\begin{split}
\mathrm{P}_{\bmu_i}\left(\left. Z_{\lfloor \alpha\sqrt{n}\rfloor}\ge \frac{\alpha c \sqrt{n}}{32}\right| h(\mathcal{T})\ge \lfloor \alpha \sqrt{n}\rfloor\right)& \ge 
\frac{  \mathrm{E}_{\bmu_i}\left[\left. Z_{\lfloor \alpha \sqrt{n}\rfloor}\right|Z_{\lfloor \alpha \sqrt{n}\rfloor}>0\right]^2    }{   \mathrm{E}_{\bmu_i}\left[\left. Z_{\lfloor \alpha \sqrt{n}\rfloor}^2\right|Z_{\lfloor \alpha \sqrt{n}\rfloor}>0\right]    }\\
& \quad \times\left(1-\frac{\alpha c \sqrt{n}}{32 \mathrm{E}_{\bmu_i}\left[\left. Z_{\lfloor \alpha \sqrt{n}\rfloor}\right|Z_{\lfloor \alpha \sqrt{n}\rfloor}>0\right]}\right)^2
\end{split}
\end{equation}
By \eqref{pop0} and \eqref{pop1}, for $n\ge N_4$ and for all $0\le i \le A_2\sqrt{n}$, we have that
\begin{equation}
\begin{split}
  \mathrm{E}_{\bmu_i}\left[\left. Z_{\lfloor \alpha \sqrt{n}\rfloor}\right|Z_{\lfloor \alpha \sqrt{n}\rfloor}>0\right]&\ge \frac{\alpha c \sqrt{n}}{4},\\
\mathrm{E}_{\bmu_i}\left[\left. Z^2_{\lfloor \alpha \sqrt{n}\rfloor}\right|Z_{\lfloor \alpha \sqrt{n}\rfloor}>0\right]
&\le   8\alpha\sigma^2\sqrt{n}  \left(1+\sum_{m=i}^{i+\lfloor \alpha \sqrt{n}\rfloor-1} \sigma_m^2\right)
\le 2^5\alpha^2\sigma^4 n.
\end{split}
\end{equation}
By \eqref{pop2}, we have that, for all $n\ge N_4$, for all $0\le i \le A_2\sqrt{n}$
\begin{equation}\label{pop2bis}
\begin{split}
\mathrm{P}_{\bmu_i}\left(\left. Z_{\lfloor \alpha\sqrt{n}\rfloor}\ge \frac{\alpha c \sqrt{n}}{32}\right| h(\mathcal{T})\ge \lfloor \alpha \sqrt{n}\rfloor\right)& \ge 
\frac{c^2}{512 \sigma^4}\times \frac{15}{16}\ge 
\frac{c^2}{2^{10} \sigma^4}.
\end{split}
\end{equation}
Now, for all $0\le i \le A_2\sqrt{n}$, letting $(Z_k')$ be a BPVE in environment $\bmu_{i+\lfloor \alpha \sqrt{n}\rfloor}$ with $Z_0'=\lceil \alpha c \sqrt{n}/32\rceil$, we have by Lemma \ref{lem:offspringvarianceandDoob}
\begin{equation}\begin{split}
\mathrm{P}_{\bmu_{i}}\left(\left.Z_{\lfloor \alpha\sqrt{n}\rfloor+\lfloor \beta\sqrt{n}\rfloor}\left(\mathcal{T}\right)< \eta \sqrt{n}\right|  Z_{\lfloor \alpha\sqrt{n}\rfloor}\ge \frac{\alpha c \sqrt{n}}{32} \right)&\le \mathrm{P}_{\bmu_{i+\lfloor \alpha \sqrt{n}\rfloor}}\left(Z_{\lfloor \beta\sqrt{n}\rfloor}'< \frac{\alpha c}{64}\sqrt{n}\right)\\
&\le 
\mathrm{P}_{\bmu_{i+\lfloor \alpha \sqrt{n}\rfloor}}\left(\left|Z_{\lfloor \beta\sqrt{n}\rfloor}'-Z_0'\right|> \frac{Z_0}{2}'\right)\\
&\le \frac{128\left(\sum_{m=i+\lfloor \alpha \sqrt{n}\rfloor}^{i+\lfloor \alpha \sqrt{n}\rfloor+\lfloor \beta \sqrt{n}\rfloor}\sigma_m^2\right)}{\alpha c \sqrt{n}}\\
&\le \frac{128}{\alpha c}\max_{0\le j\le (A_2+\alpha)\beta^{-1}} \hspace{-1mm}\hspace{-1mm}\hspace{-3mm}\frac{\sum_{m=j\lfloor 2\beta \sqrt{n}\rfloor}^{(j+1)\lfloor 2\beta \sqrt{n}\rfloor}\sigma_m^2}{\sqrt{n}}.
\end{split}
\end{equation} 
Again, by Condition \ref{(I)} and Lemma \ref{lem:deterministicLLNcuttopieces}, there exists $N_5=N_5(\bmu,\beta)$ such that, for all $n\ge N_5$, we have that
\begin{equation}\label{pop3}
\begin{split}
\mathrm{P}_{\bmu_{i}}\left(\left.Z_{\lfloor \alpha\sqrt{n}\rfloor+\lfloor \beta\sqrt{n}\rfloor}\left(\mathcal{T}\right)< \eta \sqrt{n}\right|  Z_{\lfloor \alpha\sqrt{n}\rfloor} \ge \frac{\alpha c \sqrt{n}}{32} \right)
&\le \frac{512 \beta\sigma^2}{\alpha c}\le \frac{1}{2},
\end{split}
\end{equation} 
where we used the definition \eqref{defcsts} of $\beta$ in the second inequality.
Using \eqref{pop1}, \eqref{pop2bis} and \eqref{pop3} together with \eqref{defcsts}, we have that, for all $n\ge N_4\vee N_5$ and for all $0\le i\le A_2\sqrt{n}$,
\begin{equation}
\begin{split}
\mathrm{P}_{\bmu_i}\left(\left|\mathcal{T}\right|\ge \eta \sqrt{n}\right)\ge & 
\mathrm{P}_{\bmu_i}\left(h\left(\mathcal{T}\right)\ge \alpha \sqrt{n}\right)\\
&\times  
\mathrm{P}_{\bmu_i}\left(\left. Z_{\lfloor \alpha\sqrt{n}\rfloor}\left(\mathcal{T}\right)\ge \frac{\alpha c \sqrt{n}}{32}\right| h(\mathcal{T})\ge \lfloor \alpha \sqrt{n}\rfloor\right)\\
     & \times   
   \mathrm{P}_{\bmu_{i}}\left(\left.Z_{\lfloor \alpha\sqrt{n}\rfloor+\lfloor \beta\sqrt{n}\rfloor}\left(\mathcal{T}\right)\ge \eta \sqrt{n}\right|  Z_{\lfloor \alpha\sqrt{n}\rfloor}\left(\mathcal{T}\right) \ge \frac{\alpha c \sqrt{n}}{32} \right)\\
   \ge &\frac{1}{8\alpha\sigma^2\sqrt{n}}\times \frac{c^2}{2^{10}\sigma^4}\frac{1}{2}\\
   \ge &\frac{2\ln(4/\delta)}{\kappa\sqrt{n}}.
\end{split}
\end{equation}
Using the above, we obtain that, for all $n\ge N_4\vee N_5$ and if $n\ge 4\kappa^{-2}$, then 
\begin{equation}
\begin{split}
\1_{\left\{H_{\tau_k^{(n)}}\le A_2\sqrt{n}\right\}} \cdot\Pmu\left(\left. \sum_{j=1}^{\lfloor \kappa\sqrt{n}\rfloor} \left|T_j\right| < \eta n\right|\mathcal{G}_k^{(n)}\right)\le \left(1-\frac{2\ln(4/\delta)}{\kappa\sqrt{n}}\right)^{\lfloor\kappa\sqrt{n}\rfloor}\le \frac{\delta}{4}.
\end{split}
\end{equation}
This concludes the proof of \eqref{spartanbeast} and thus of \eqref{eqn:LeGalltauks}, with $N_2=N_3\vee N_4\vee N_5\vee (4\kappa^{-2})$ and $\eta=\eta(\gamma,\delta)$ as in \eqref{defcsts}.\\

We now want to prove that the sequence $(\tau_k^{(n)})_{k\geq 1}$ does not have too small increments before $\tau_k^{(n)}$ becomes larger than $n$.\\
By \eqref{eqn:LeGalltauks}, there exists $u>0$ and $N_6=N_6(\bmu)$ such that, for all $n\ge N_6$ and all $k\ge0$, 
\begin{equation}
\1_{\left\{H_{\tau_k^{(n)}}\le A_2\sqrt{n}\right\}} \cdot \mathrm{P}_{\bmu}\left(\left.\tau_{k+1}^{(n)}-\tau_k^{(n)} < u n\right|\mathcal{G}_k^{(n)}\right)\leq \frac{1}{2}.
\end{equation}

Note that as long as  $H_{\tau_k^{(n)}}\le A_2\sqrt{n}$, the random variable $\tau_{k+1}^{(n)}-\tau_k^{(n)}$ stochastically dominates $(un)\cdot B_k$, where $(B_k)_{k\ge0}$ is a sequence of i.i.d.~Bernoulli random variables with parameter $1/2$. Therefore, for all $K\ge0$, for all $n\ge N_6$,
\[
\mathrm{P}_{\bmu}\left(  \tau_{K}^{(n)} \le n, \max_{0\le k\le n}H_k\le A_2\sqrt{n} \right)\le \mathrm{P}_{\bmu}\left(   \sum_{k=0}^{K-1}B_k \le \frac{1}{u} \right).
\]
Because the sequence $(B_k)_{k\geq 0}$ satisfies the law of large numbers, there exists $K$ (depending on $u$ only, and $u$ does not depend on other parameters), such that, for all $n\ge N_2$, 
\begin{equation}\label{fast}
\mathrm{P}_{\bmu}\left(  \tau_{K}^{(n)} \le n, \max_{0\le k\le n}H_k\le A_2\sqrt{n} \right)\le \frac{\delta}{8}.
\end{equation}
Using \eqref{eqn:LeGalltauks} again, there exists $\varepsilon=\varepsilon(\delta,\gamma)>0$ and $N_7=N_7(\bmu,\delta)\ge N_6$ such that, for all $n\ge N_7$ and all $k\ge0$, 
\begin{equation}
\1_{\left\{H_{\tau_k^{(n)}}\le A_2\sqrt{n}\right\}} \cdot \mathrm{P}_{\bmu}\left(\left.\tau_{k+1}^{(n)}-\tau_k^{(n)} < \varepsilon n\right|\mathcal{G}_k^{(n)}\right)\leq \frac{\delta}{8K}.
\end{equation}
Thus, we obtain that for all $n\ge N_7$,
\begin{equation}\label{haaaafin}
\begin{split}
&\mathrm{P}_{\bmu}\left(\left\{\max_{0\le k \le n}H_k\le A_2\sqrt{n}\right\}\cap\bigcup_{k\ge0}\left\{\tau_k^{(n)}\le n, \tau_{k+1}^{(n)}-\tau_k^{(n)}<\varepsilon n\right\}\right)\\
&\le \mathrm{P}_{\bmu}\left(\max_{0\le k \le n}H_k\le A_2\sqrt{n} , \tau_K^{(n)}\le n\right)+\sum_{k=0}^{K-1}  \mathrm{P}_{\bmu}\left(H_{\tau_k^{(n)}\le A_2\sqrt{n}}, \tau_{k+1}^{(n)}-\tau_{k}^{(n)}< \varepsilon n\right)\\
&\le \frac{\delta}{8}+K\times\frac{\delta}{8K}\le \frac{\delta}{4}.
\end{split}
\end{equation}
To conclude, define the events
\begin{align*}
E_{1,n}&=\left\{\max_{0\le k \le n}H_k\le A_2\sqrt{n}\right\},\\
E_{2,n}&= \bigcap_{k\ge0}\left(\left\{\tau_k^{(n)}>n\right\}\cup\left\{\tau_{k+1}^{(n)}-\tau_k^{(n)}\ge \varepsilon n\right\}\right).
\end{align*}
Observe that
\[
E_{1,n}\cap E_{2,n}\subset \left\{   \max_{0\le k\le n} \max_{0\le i\le \varepsilon n}\left( H_{k+i}-H_k\right)\le\gamma \sqrt{n}   \right\}
\]
Hence, using \eqref{oats} and \eqref{haaaafin}, we have that, for all $n\ge N_1\vee N_7$,
\begin{align*}
\mathrm{P}_{\bmu}\left(\max_{0\le k\le n} \max_{0\le i\le \varepsilon n}\left( H_{k+i}-H_k\right)>\gamma \sqrt{n}\right)&\le \mathrm{P}_{\bmu}\left(E_{1,n}^c\right)+\mathrm{P}_{\bmu}\left(E_{1,n},E_{2,n}^c\right)\\
&\le \frac{\delta}{4}+\frac{\delta}{4}<\delta.
\end{align*}
This proves the statement of the lemma, with $N=N_1\vee N_7$.
\end{proof}

We are now ready to prove Theorem \ref{thm:jointcvLukaHeight}.

\begin{proof}[Proof of Theorem \ref{thm:jointcvLukaHeight}]
Let $\bmu$ be a strictly critical environment satisfyring Conditions \ref{(I)}-\ref{(V)}.
By Corollary~\ref{cor:Luka}, it is enough to prove that for all $\gamma>0$, and all $\delta>0$, there exists $N=N(\bmu,\gamma,\delta)$ such that, for all $n\ge N$,  
\[
\Pmu\left(\sup_{1\leq k\leq n} \frac{1}{\sqrt{n}}\bigg|  H_k- \frac{2(X_k-I_k)}{\sigma^2}\bigg|  >\gamma\right)\le \delta.
\]
For $\varepsilon\in(0,1/2)$, define
 $\cE_{1,n}(\varepsilon)$ the event that there exist (random) integers $K\in\mathbb{N}$, and $0=k_0< k_1< \ldots < k_K<k_{K+1}= n$ such that: 
\begin{itemize}
\item[$(i)$] for all $0\leq i\leq K$, $k_{i+1}-k_i \leq 1+\varepsilon n$,
\item[$(ii)$] for all $1\leq i \leq K$, $\left|\frac{(X_{k_i}-I_{k_i})}{H_{k_i}}-\frac{\sigma^2}{2}\right| \leq \varepsilon$.
\end{itemize}
Recall the definition \eqref{badbadbad} of a $\varepsilon$-bad vertex and recall the relation \eqref{pipi}. Using these two, note that if an index $k$ is not $\varepsilon$-bad, then it satisfies the item $(ii)$ above, and if there are less than $\varepsilon n$ $\varepsilon$-bad vertices, then we can find a collection of indices that are not $\varepsilon$-bad such that there are at most $\varepsilon n$ $\varepsilon$-bad vertices between any two indices.  Hence, we have that
\[
\{\mathcal F_n\text{ has at most }\varepsilon n \ \varepsilon\text{-bad vertices}\}\subset \cE_{1,n}(\varepsilon).
\]
Thus, by Proposition~\ref{prop:fewbadvertices}, 
there exists $N_1=N_1(\bmu,\varepsilon,\delta)$ such that for all $n\ge N_1$,
\begin{equation}
\Pmu\left(\cE_{1,n}^c(\varepsilon)\right)\le \frac{\delta}{10}.
\end{equation}
For all $n\ge0$, $\varepsilon>0$ and $\gamma>0$, let us define the event
\begin{equation}
\cE_{2,n}(\varepsilon,\gamma)=\bigcap_{\stackrel{1\leq j, k\leq n,}{\vert j-k\vert\leq 2\varepsilon n}}\left\{\frac{2(X_k-I_k)}{\sigma^2\sqrt{n}} < \frac{1}{\varepsilon^{1/2}},\ 
 \frac{1}{\sqrt{n}}\left\vert \frac{2(X_k-I_k)}{\sigma^2}- \frac{2(X_j-I_j)}{\sigma^2}\right\vert < \frac{\gamma}{10}\right\}.
\end{equation}

{To estimate the probability of the event above, we will use Corollary~\ref{cor:Luka}.  Let $(B_t)_{t\ge 0}$ be a standard Brownian motion under some measure $\mathbb{P}$. It is a.s.~uniformly continuous (hence bounded) on $[0,1]$. Thus there exists $\varepsilon_0=\varepsilon_0(\delta,\gamma)\in(0,1/2)$ such that for all $\varepsilon\in(0,\varepsilon_0)$,
\begin{equation}
\begin{split}
\mathbb{P}\left(\sup_{0\le t \le 1}\vert B_t\vert \ge \frac{1}{2\varepsilon^{1/2}}\right)+\mathbb{P}\left(\sup_{\stackrel{0\le s\le t \le 1}{|t-s|\le 2\varepsilon}}\vert B_t-B_s\vert \ge \frac{\gamma}{20}\right)  &\le \frac{\delta}{10}.
\end{split}
\end{equation}
By Corollary~\ref{cor:Luka}, for all $\varepsilon\in(0,\varepsilon_0)$, there exists $N_2=N_2(\bmu,\varepsilon,\delta,\gamma)\ge N_1(\bmu,\varepsilon,\delta)$ such that for all $n\ge N_2$,
\begin{equation}
\Pmu\left(\cE_{1,n}^c(\varepsilon)\cup \cE_{2,n}^c(\varepsilon,\gamma)\right)\le \frac{\delta}{5}.
\end{equation}
}

On $\cE_{1,n}(\varepsilon)\cap \cE_{2,n}(\varepsilon,\gamma)$, for all $1\leq i\leq K$, we have that
\begin{equation}\label{e2}
\left\vert \frac{X_{k_i}-I_{k_i}}{H_{k_i}}-\frac{\sigma^2}{2}\right\vert
\leq \varepsilon,
\end{equation}
and thus
\[
H_{k_i}\le \frac{2\left(X_{k_i}-I_{k_i}\right)}{\sigma^2+2\varepsilon}.
\]
Putting the two displays above together, on $\cE_1^n(\varepsilon)\cap \cE_2^n(\varepsilon,\gamma)$, we have for all $1\leq i\leq K$
\begin{equation}\label{teletele}
\left|\frac{2(X_{k_i}-I_{k_i})}{\sigma^2}-H_{k_i}\right|
\leq \left|\frac{X_{k_i}-I_{k_i}}{H_{k_i}}-\frac{\sigma^2}{2}\right|\cdot \frac{2 H_{k_i}}{\sigma^2}
\le \frac{4\varepsilon\left(X_{k_i}-I_{k_i}\right)}{\sigma^2(\sigma^2+2\varepsilon)}
\le\frac{4\varepsilon^{1/2}}{\sigma^2(\sigma^2+2\varepsilon)}\sqrt{n}.
\end{equation}
Let us define the event
\begin{equation}\label{teletele2}
\cE_{3,n}(\varepsilon,\gamma)=\left\{\max_{0\le k\le n} \max_{0\le i\le \varepsilon n}\left( H_{k+i}-H_k\right)\le\frac{\gamma}{10} \sqrt{n}\right\}.
\end{equation}
By Lemma \ref{KingCharles}, we have that
for all $\gamma>0$ and all $\delta>0$, there exists $\varepsilon_1=\varepsilon_1(\gamma,\delta)\in(0,\varepsilon_0(\gamma,\delta))$, such that for all $\varepsilon\in(0,\varepsilon_1)$, there exists $N_3=N_3(\bmu,\varepsilon,\delta,\gamma)\ge N_2(\bmu,\varepsilon,\delta,\gamma)$ such that if $n\ge N_3$, then
\begin{equation}\label{e3}
\mathrm{P}_{\bmu}\left( \cE_{3,n}^c(\varepsilon,\gamma) \right)<\frac{\delta}{10}.
\end{equation}

Using \eqref{teletele} and \eqref{teletele2}, we have that, on $\cE_{1,n}(\varepsilon)\cap \cE_{2,n}(\varepsilon,\gamma)\cap \cE_{3,n}(\varepsilon,\gamma)$, for all $j\in\{0,\dots,n\}$, there exists $i(j)\in\{1,\dots,K\}$ such that $|j-k_{i(j)}|\le \varepsilon n$, hence, we have that
\begin{equation}
\begin{split}
\left| H_j-\frac{2\left( X_j-I_j\right)}{\sigma^2}\right|
\le & \left\vert   H_j-H_{k_{i(j)}}              \right\vert+\left\vert H_{k_{i(j)}}-\frac{2\left(X_{k_{i(j)}}-I_{k_{i(j)}}\right)}{\sigma^2}\right\vert \\
&+ \left\vert \frac{2\left(X_{k_{i(j)}}-I_{k_{i(j)}}\right)}{\sigma^2}-\frac{2\left(X_j-I_j\right)}{\sigma^2}\right\vert\\
\le & \frac{\gamma}{10}\sqrt{n}+\frac{4\varepsilon^{1/2}}{\sigma^2(\sigma^2+2\varepsilon)}\sqrt{n}+\frac{\gamma}{10}\sqrt{n}.
\end{split}
\end{equation}
Choosing $\varepsilon$ small enough, depending on $\gamma$ and $\sigma^2$, the RHS of the last line can be made smaller than $\gamma\sqrt{n}$.\\
Hence, for all $\gamma>0$ and $\delta>0$, using the above together with \eqref{e2} and \eqref{e3}, there exists $\varepsilon=\varepsilon(\gamma,\delta)$ and $N=N(\bmu,\gamma,\delta)$ such that, for all $n\ge N$
\[
\Pmu\left(\max_{0\le j\le n} \frac{1}{\sqrt{n}}\left| H_j-\frac{2\left( X_j-I_j\right)}{\sigma^2}\right|>\gamma\right)\le \mathrm{P}_{\bmu}\left( \cE_{1,n}^c(\varepsilon) \cup\cE_{2,n}^c(\varepsilon,\gamma)\cup  \cE_{3,n}^c(\varepsilon,\gamma) \right)<\delta,
\]
which concludes the proof of the theorem.
\end{proof}

\subsection{Extracting large trees from the forest}\label{subsec:extraction}
In this section, we prove Theorem~\ref{th:main}, as a consequence of Theorem~\ref{thm:jointcvLukaHeight}. 
\\
Before doing so, we derive from Theorem~\ref{thm:jointcvLukaHeight} 
the convergence in distribution of the ordered excursions of the height process, following a standard reasoning originally due to Aldous~\cite{AldousLimic}.\\
For all $n\ge 1$ and $i\ge1$, let ${\bf e}_i^{(n)}$ be the $i$-th longest excursion above 0 of $n^{-1/2}(H_{\lfloor nt\rfloor})_{0\leq t\leq 1}$ that ends 
before time 1. For $n\ge1$ and $i\ge1$, 
let $l_i^{(n)}$ be the starting time of ${\bf e}_i^{(n)}$, and $r_i^{(n)}$  its ending time. Similarly, for $n\ge1$ and $i\ge1$, define ${\bf e}_i^{B}$, $l_i^B $ and $r_i^B$ the $i$-th longest excursion, its starting time and ending time, of $\frac{2}{\sigma}(\vert B_t\vert)_{0\leq t\leq 1}$, where $B$ is a standard Brownian motion. 
\begin{lemma}\label{almostdone}
Let $\bmu$ be a strictly critical environment satisfying Condition \ref{(I)}-\ref{(V)}. We have that
\begin{equation*}
({\bf e}_i^{(n)},l_i^{(n)})_{i\geq 1}\overset{(d)}{\longrightarrow }({\bf e}_i^B,l_i^B)_{i\geq 1},
\end{equation*}
as $n$ goes to infinity, in Skorokhod topology for the first coordinate and in product topology for the second.
\end{lemma}

\begin{proof}
We are going to apply Proposition \ref{propapp}, stated in the appendix, as it is a slight modification of  Lemma~5.8 in~\cite{CKG}. To apply it, we need to have a better separation of the excursions of $(H_n)$.\\
Recall that, in \eqref{massagegun}, we established that every time $(X_n)$ has explored a whole tree, it reaches a new strict minimum that beats the previous minimum by $-1$. In other words, $(I_n)_n$, the running minimum of $(X_n)$, is equal to minus the number of trees fully explored at time $n$. On the other hand, $(H_n)$ starts at $0$ and returns to $0$ only when a tree has just been fully explored. Therefore, the process $(\frac{\sigma^2}{2}H_n+I_n)_n$ achieves a new strict minimum (by $-1$) every time a tree has been fully explored.\\
Let us define, for all $n\ge1$ and all $t\in[0,1]$,
\[
\widetilde{H}_n(t)=\frac{1}{\sqrt{n}}\left(\frac{\sigma^2}{2}H_{\lfloor nt\rfloor}+I_{\lfloor nt\rfloor}\right).
\] 
By Theorem \ref{thm:Luka}, Theorem~\ref{thm:jointcvLukaHeight} and the continuous mapping theorem, we have that $(\widetilde{H}_n(t))_{t\in[0,1]}$ converges to $(\sigma B_t)_{t\in[0,1]}$.\\
Now we can apply Proposition \ref{propapp} with $f=\sigma B$ and, for all $n\ge1$, $f_n=\widetilde{H}_n(\cdot)$. Moreover, let us define
{$J_n=-I_n+1$ (i.e.~the number of trees in $\mathcal{F}_n$), for all $1\le i\le J_n$,
 $t_{n,i}'= j(i)/n$ where $j(i)$ is the depth-first label of the root of the $i$-th tree, i.e.~the $i$-th time (multiplied by $n$) $\widetilde{H}_n(\cdot)$ is at a new strict minimum. 
}

The assumptions on $f$ are obtained by well known properties of Brownian motion. The facts that $t_{n,i}\in[0,1]$, that $f_n(t_{n,i})<f_n(s)$ for all $s<t_{n,i}$ and that $f_n(s)>f_n(t_{n,J_n})$ for all $s>t_{n,J_n}$ are easily seen to be satisfied by construction. In addition, remark that $\sqrt{n}\max_{1\le i < J_n}(f_n(t_{n,i})-f_n(t_{n,i+1}))=1$, 
thus we have that $\max_{1\le i < J_n}(f_n(t_{n,i})-f_n(t_{n,i+1}))$ goes to $0$ almost surely. Therefore, we have that the starting times (the $t_{n,i}$'s) and lengths of the excursions of $\widetilde{H}_n(\cdot)$ above its running minimum converge to those of $\sigma B$.\\
Now, remark that the starting times and lengths of the excursions of $\widetilde{H}_n(\cdot)$ are the starting times and lengths of the excursions of $(H_{\lfloor nt\rfloor}/\sqrt{n})$ strictly above $0$. This yields the convergence of $(l^{(n)}_i,r^{(n)}_i-l^{(n)}_i)$ towards $(l^{B}_i,r^{B}_i-l^{B}_i)$. Finally, the convergence of ${H}_{\lfloor n\cdot\rfloor}/\sqrt{n}$ towards $\frac{2}{\sigma}{B}$ on $[0,1]$ imply the convergence on the excursion intervals (which are non-empty, as $,r_i^B>l_i^B$ for all $i\geq 1$), and this concludes the proof of the lemma.
\end{proof}

\begin{proof}[Proof of Theorem~\ref{th:main}]
By Skorokhod's representation theorem, we can assume that this convergence holds $\Pmu$-almost surely. Then, for every $\varepsilon>0$, there exists $\delta=\delta(\varepsilon)$ small enough such that, 
with probability at least $1-\varepsilon$, 
the first excursion ${\bf e}^B(\delta)$ above 0 of $(\vert B_t\vert)_{0\leq t\leq 1}$ of length at least $\delta$ is entirely comprised in $[0,1/2]$. Moreover, there exists $\delta'=\delta'(\delta,\varepsilon)$ small enough 
such that with probability at least $1-\varepsilon$, no other excursion starting in $[0,1/2+\delta']$ has length at least $\delta-2\delta'$. This is due to the fact that for every $\delta >0$, if ${\bf e}$ is an excursion above $0$ of a standard Brownian motion, then the map $x\mapsto \dP(\ell({\bf e}) >x\,\vert \ell({\bf e}) >\delta/2)$ is continuous on $(\delta, \infty)$, where $\ell({\bf e})$ denotes the length of ${\bf e}$. 
\\
Denote ${\bf e}^{(n)}(\delta,\delta')$ the first excursion of $n^{-1/2}(H_{\lfloor nt\rfloor})_{0\leq t\leq 1}$ having length at least $\delta-\delta'$. Then Lemma \ref{almostdone} yields that, on these events,
\[
{\bf e}^{(n)}(\delta,\delta')\underset{n\to \infty}{\longrightarrow } {\bf e}^{B}(\delta)
\]
in Skorokhod space, as ${\bf e}^B(\delta)$ is the only excursion long enough starting in $[0,1/2+\delta']$ towards which ${\bf e}^{(n)}(\delta,\delta')$ can converge. But the distribution of ${\bf e}^{(n)}(\delta,\delta')$ coincides with that of the height process of $\cT_{\lceil (\delta-\delta') n\rceil}$, 
on the event that $n^{-1/2}(H_{\lfloor nt\rfloor})_{0\leq t\leq 1}$ has such an excursion before time 1 (which has probability at least $1-2\varepsilon$ for $n$ large enough). 
Hence Theorem~\ref{th:main} follows from this, 
from the construction of the CRT (by \cite[Equation~(8)]{CKG}, convergence of the excursions in Skorokhod space implies that of the underlying trees in Gromov-Hausdorff-Prokhorov metric) and from the usual scaling property of the Brownian motion.
\end{proof}

\begin{appendix}
\section*{}
\begin{proposition}\label{propapp}
Let $f:[0,1]\to \mathbb{R}$ be a continuous function, let $E(f)$ be the set of non-empty intervals $e=(l,r)\subset [0,1]$ such that $f(l)=\inf_{s\leq l}f(s)=f(r)$ and $f(s)>f(l)$ for all $s\in(l,r)$. 
We call such an interval $e$ an excursion of $f$ above its running minimum, and we call $r-l$ its length. We will assume furthermore that
\begin{itemize}
\item for all $\varepsilon>0$,  $E(f)$ contains only finitely many excursions of length greater than or equal to $\varepsilon$;
\item the set $[0,\sup_{e_\in E(f), s\in e} s]\setminus \bigcup_{e\in E(f)} e$ has Lebesgue measure $0$;
\item if $(l_1,r_1),(l_2,r_2)\in E(f)$ and $l_1<l_2$, then $f(l_1)>f(l_2)$.
\end{itemize}
Let $(f_n)_{n\ge1}$ be a sequence of c\`adl\`ag functions such that $f_n\to f$ in the Skorokhod sense. For each $n\ge1$, let $J_n\ge 1$, $(t_{n,i})_{1\le i\le J_n}$ be a collection of strictly increasing numbers in $[0,1]$, with $t_{n,1}=0$, such that $f_n(t_{n,i})< f_n(s)$ for all $s<t_{n,i}$, $f_n(s)\ge f_n(t_{n,J_n})$ for all $s>t_{n,J_n}$,  and $\lim_{n\to \infty} \max_{1\le i < J_n}(f_n(t_{n,i})-f_n(t_{n,i+1}))=0$.\\
Then, we have that
\[
\left\{(t_{n,i},t_{n,i+1}-t_{n,i}); {1\le i< J_n}\right\}\longrightarrow \left\{(l,r-l); (l,r)\in E(f)\right\},
\]
as $n$ goes to infinity and where the convergence holds in the topology of vague convergence of counting measures on $[0,1]\times (0,1]$. 
\end{proposition}
\begin{proof}
This is a slight modification of Lemma~5.8 in~\cite{CKG}, whose proof can be easily adapted. The main difference is that we assume here that $f$ is continuous and supported on $[0,1]$. 
We assume in our statement that conditions (4), (5) and (6) of \cite{CKG} are satisfied. Conditions (1), (3) of \cite{CKG} are easily checked, by continuity, while condition (2) of \cite{CKG} can be dropped because $f$ is supported on $[0,1]$. 
Similarly, we do not need to assume the second part of condition $(i)$ of \cite{CKG} on the sequences $(t_{n,i})$.\\
The proof of Lemma 5.8 in \cite{CKG} can be followed line by line, with simple adaptations to our case.
\end{proof}
\end{appendix}
\bibliographystyle{alpha} 
\bibliography{gwre}       


\end{document}